\documentclass[12pt, twoside]{amsart}

\usepackage{amssymb, amsthm, amsfonts, mathrsfs, graphicx,graphics,psfrag,mathtools, bm, bbm}

\usepackage[all]{xy}
\usepackage{tikz}
\usetikzlibrary{decorations.pathreplacing}

\usepackage[utf8]{inputenc}
\usepackage[bbscaled=.95,cal=boondoxo]{mathalfa}

\usepackage{hyperref}
\hypersetup{colorlinks}
\definecolor{darkred}{rgb}{0.5,0,0}
\definecolor{darkgreen}{rgb}{0,0.5,0}
\definecolor{darkblue}{rgb}{0,0,0.5}

\usepackage[left= 3.4 cm, right= 3.4 cm, top= 2.5 cm, bottom=2.7 cm, foot= 1.2 cm, marginparwidth=2.7 cm, marginparsep=0.5 cm]{geometry}
\newcommand{\on}{\operatorname}
\renewcommand{\d}{{\on{d}}}
\newcommand{\hol}{{\on{hol}}}

\newcommand\bs{\smallsetminus}

\newcommand\ev{\on{ev}}

\newcommand\sig{\sigma}
\newcommand\eps{\epsilon}
\newcommand\Om{\Omega}

\newcommand\qu{/\kern-.7ex/} 

\newcommand{\beq}{\begin{equation}}
\newcommand{\eeq}{\end{equation}}
\newcommand{\beqn}{\begin{equation*}}
\newcommand{\eeqn}{\end{equation*}}
\newcommand{\ov}{\overline}
\newcommand{\mb}{\mathbb}
\newcommand{\mc}{\mathcal}
\newcommand{\mf}{\mathfrak}
\newcommand{\ms}{\mathscr}

\newcommand{\wt}{\widetilde}

\newcommand{\dad}{d_a^{\,*}}
\newcommand{\cu}{\check{u}}
\newcommand{\ck}{\check}
\newcommand{\ud}{\underline}

\newtheorem{theorem}{Theorem}[section]
\newtheorem{assumption}[theorem]{Assumption}
\newtheorem{cor}[theorem]{Corollary}
\newtheorem{prop}[theorem]{Proposition}
\newtheorem{lemma}[theorem]{Lemma}
\newtheorem*{claim}{Claim}

\newtheorem{defn}[theorem]{Definition}
\theoremstyle{remark}
\newtheorem{remark}[theorem]{Remark}
\newtheorem{notation}[theorem]{Notation}

\makeatletter
\@addtoreset{theorem}{section}
\makeatother

\makeatletter 
\@addtoreset{equation}{section}
\makeatother  

\numberwithin{equation}{section}

\title{Local model for the moduli space of affine vortices}

\author[Venugopalan]{Sushmita Venugopalan}
\address{Institute of Mathematical Sciences, CIT Campus, Taramani, Chennai 600113, India}
\email{sushmita@imsc.res.in}

\author[Xu]{Guangbo Xu}
\address{Department of Mathematics, Princeton University, Fine Hall, Washington Road, Princeton, NJ 08544 USA}
\email{guangbox@math.princeton.edu}

\begin{document}

\begin{abstract} 
We show that the moduli space of regular affine vortices, which are solutions of the symplectic vortex equation over the complex plane, has the structure of a smooth manifold. The construction uses Ziltener's Fredholm theory results \cite{Ziltener_book}. We also extend the result to the case of affine vortices over the upper half plane. These results are necessary ingredients in defining the ``open quantum Kirwan map'' proposed by Woodward \cite{Woodward_toric}.


\end{abstract}

\maketitle

\setcounter{tocdepth}{1}
\tableofcontents

\section{Introduction}

A vortex consists of a connection on a complex vector bundle over a Riemann surface and a section holomorphic with respect to that connection. They also satisfy an equation involving the curvature of the connection and a nonlinear term depending on the section.  Vortices initially arose in Ginzburg-Landau theory of superconductivity (\cite{Taubes_vortex}, \cite{Jaffe_Taubes}), but they have since appeared in various areas of mathematics and physics. For example, they are related to Weil's scheme of torsion quotients \cite{BCK}\cite{BDHW}, knot invariant \cite{dimofte:vortex} and Chern-Simons theory \cite{dupei:vortex}. In symplectic geometry they are also generalized to nonlinear target spaces \cite{Cieliebak_Gaio_Salamon_2000} \cite{Mundet_thesis, Mundet_2003}, which form the foundation of the symplectic gauged Gromov-Witten theory (see also \cite{Cieliebak_Gaio_Mundet_Salamon_2002} \cite{Mundet_Tian_2009}). 

The main objects we consider in this paper are affine vortices, which are vortices over the affine complex line ${\mb C}$. Our motivation comes from the role of affine vortices in the quantum Kirwan map conjecture proposed by D. Salamon and studied by Gaio-Salamon \cite{Gaio_Salamon_2005}, Woodward \cite{Woodward_15} and Ziltener \cite{Ziltener_book}. Affine vortices will also play a similar role in the project of the second named author and G. Tian (see \cite{Tian_Xu, Tian_Xu_2, Tian_Xu_3}) on the mathematical construction of the gauged linear $\sigma$-model. 

Our basic setting is as follows. Let $K$ be a compact Lie group and $X$ be a Hamiltonian symplectic $K$-manifold. An affine vortex is a pair ${\bf v} = (A, u)$, where $A$ is a connection on the trivial principal bundle $P:={\mb C} \times K$ and $u$ is a section on the associated bundle $P \times_K X$ that satisfies the holomorphicity and vortex equations:
\begin{align*}
&\ \ov\partial_A u=0, &\ F_A + \mu (u) ds dt =0. 
\end{align*}
Here $F_A$ is the curvature $A$ and $\mu :X  \simeq {\mf k}$ is the moment map. The above equation has a gauge symmetry and each gauge equivalence class of solutions represents an equivariant homology class $\beta \in H_2^K(X; {\mb Z})$ (see Remark \ref{rem28}). We can define a topology, called the compact convergence topology (see Definition \ref{defn24}) on the moduli space $M^K({\mb C}, X; \beta)$ of gauge equivalence classes of solutions representing the class $\beta$.

In this paper, we study the local structure of the moduli space of affine vortices. This is an important step in defining the counting of affine vortices, or more generally, the virtual integration over their moduli spaces. More specifically, for an affine vortex ${\bf v}$, we define a linear differential operator $\hat{\mc D}_{\bf v}$ (see \eqref{eqn25} below), which describes the infinitesimal deformation of affine vortex. ${\bf v}$ is called {\em regular} if $\hat{\mc D}_{\bf v}$ is surjective (after extending to a bounded operator between Sobolev spaces). The following is our main result.

\begin{theorem}\label{thm11} 
Let $(X, \omega_X)$ be a Hamiltonian $K$-manifold with moment map $\mu: X \to {\mf k}$ such that the $K$-action on $\mu^{-1}(0)$ is free. Given a $K$-invariant, $\omega_X$-compatible almost complex structure $J$ on $X$.
\begin{enumerate}
\item The moduli space $M^K({\mb C}, X; \beta)^{\rm reg}$ of gauge equivalence classes of regular affine vortices representing an equivariant homology class $\beta \in H_2^K (X; {\mb Z})$ has the structure of a smooth manifold of dimension $\dim(X)-2\dim(K) + 2\langle c_{\, 1}^K(TX), \beta\rangle$. Here $c_{\,1}^K(TX)\in H_K^2(X; {\mb Z})$ is the equivariant first Chern class of $X$.

\item The manifold topology coincides with the compact convergence topology.
\end{enumerate}
\end{theorem}

\begin{remark}
The main result of this paper seems obvious to symplectic geometers, especially for that Ziltener has already obtained a Fredholm result and given the dimension formula in \cite{Ziltener_book}. We would like to explain in this remark the subtlety of our result. Typically, in a moduli problem such as defining Gromov-Witten invariants or Floer homology, a moduli space can be realized as the zero locus of a smooth section ${\mc F}: {\mc B} \to {\mc E}$ where ${\mc B}$ is a Banach manifold and ${\mc E}$ is a Banach bundle over ${\mc B}$. The way of setting up such a framework using Sobolev or weighted Sobolev spaces in Gromov-Witten and Floer theory has been standard knowledge. It is also an easy regularity result that the topology on the moduli space defined by a rather weak notion of convergence is equivalent to the topology induced from the Banach manifold ${\mc B}$. Therefore, in particular, if ${\mc F}$ is transverse, the moduli space inherits a smooth manifold structure.

For affine vortices, this stereotype cannot be taken for granted. In \cite{Ziltener_book}, Ziltener's study of a linear Fredholm operator ${\mc D}: {\mc X} \to {\mc Y}$ only describes the infinitesimal deformation {\it formally}, since he neither explicit defined any Banach manifold that has tangent space ${\mc X}$, nor proved that the kernel of ${\mc D}$ parametrizes a neighborhood of the moduli space in the transverse case. We would like to establish Theorem \ref{thm11} under essentially the same framework as \cite{Ziltener_book}, by adopting the Banach space ${\mc X}$ as a local chart of a Banach manifold. The difficult part is to show that the topology of the moduli space, which is defined by convergence (modulo gauge transformation) over compact subsets, coincides with the (stronger) topology induced from the Banach space ${\mc X}$, which has certain condition on the rate of decay at infinity. This is essentially a regularity result, whose proof occupies a large portion of this paper. For example, we need to establish the local slice theorem for the noncompact domain ${\mb C}$ and the specifically chosen weighted Sobolev norm. Only after proving those technical results, one can say that Ziltener's work on the linear level is indeed the correct one for studying infinitesimal deformations of affine vortices.
\end{remark}

\begin{remark}
Our theorem is proved for a fixed almost complex structure $J$ on $X$. In \cite{Xu_glue} and the forthcoming \cite{Woodward_Xu}, to achieve transversality, one needs to use almost complex structures which depend on the points of the domain, and which coincide with a fixed $J$ near infinity. The restriction of using a fixed $J$ won't invalidate our application. More generally, one can construct virtual cycles on the moduli spaces using abstract perturbations in the absence of transversality. Our regularity result is also a prerequisite for such construction.
\end{remark}

There is also an algebraic approach towards the local structure of the moduli space of affine vortices. In \cite{Woodward_15} Woodward used the Behrend-Fantechi machinery \cite{Behrend_Fantechi} to construct a perfect obstruction theory on the moduli space of affine vortices and proved Salamon's conjecture on the quantum Kirwan map in the algebraic case. Nonetheless, a symplectic geometric description of the moduli space has the advantage of being generalizable to the {\em open} case, i.e. to the case of vortices defined on the upper half of the complex plane with boundary in a Lagrangian submanifold. Conjectured by Woodward in \cite{Woodward_toric}, counting vortices on the half-plane leads to an {\em open quantum Kirwan map}, which is supposed to intertwine the $A_\infty$ structures on $X$ and its symplectic quotient $X \qu K$. The definition of this map is the ongoing project of the second named author with Woodward (see \cite{Wang_Xu, Xu_glue, Woodward_Xu}). Towards this direction, in the appendix of this paper, the second named author proves the analogue of Theorem \ref{thm11} in the open case.

The remaining of this paper is organized as follows. In Section \ref{section2} we recall the basic notions and preliminary results about affine vortices. In Section \ref{section3} we prove the main theorem, while leaving the proof of four technical results to Section \ref{section4}--\ref{section7}. In particular in Section \ref{section4} we reproduce the proof of Ziltener's Fredholm result. In the appendix the parallel discussion for affine vortices over the upper half plane is provided.

{\bf Acknowledgements.} We would like to thank Chris Woodward for sharing many ideas and many valuable discussions. We would like to thank Fabian Ziltener for valuable comments on the preliminary version of this paper. 

The first named author was a post-doctoral fellow at Chennai Mathematical Institute when part of the paper was written.
\section{Preliminaries}\label{section2} 

Let $K$ be a compact connected Lie group with Lie algebra ${\mf k}$. Suppose $K$ acts on a symplectic manifold $(X,\omega_X)$. Suppose the action is Hamiltonian, namely, there is a moment map, i.e., a $K$-equivariant map $\mu: X \to {\mf k}^\vee$ such that
\beqn
\omega_X ( {\mc X}_a,\cdot) = d \langle \mu, a \rangle , \quad \forall a \in {\mf k}.
\eeqn
Here ${\mc X}_a\in \Gamma(TX)$ is the infinitesimal action of $a$, and we follow the convention that $a \mapsto -{\mc X}_a$ is a Lie algebra homomorphism.

We choose, once and for all, a $K$-invariant, $\omega_X$-compatible almost complex structure $J$ on $X$. Let $G$ be the complexification of $K$, with Lie algebra ${\mf g} \simeq {\mf k}\otimes {\mb C}$. The infinitesimal action extends to a linear map ${\mf g} \to \Gamma(TX)$ as $
{\mf g} \ni a + {\bm i} b \mapsto {\mc X}_a + J {\mc X}_b$.

Choose on ${\mf k}$ an ${\rm Ad}$-invariant metric which identifies ${\mf k}$ with ${\mf k}^\vee$. We view $\mu$ as ${\mf k}$-valued. 

An important assumption throughout this paper is the following.	
\begin{assumption}\label{assumption21} 
The $K$-action on $\mu^{-1}(0)$ is free.
\end{assumption}

Then we obtain a smooth symplectic quotient 
\beqn
\bar{X}:= X\qu K:= \mu^{-1}(0)/K.
\eeqn

We also need a $K$-invariant Riemannian metric with respect to which $\mu^{-1}(0)$ is totally geodesic. To construct such a metric, start with an arbitrary one such that $\mu^{-1}(0)$ is totally geodesic, then integrate over $K$ against the Haar measure to obtain a $K$-invariant metric. Notice that the integration will preserve the totally geodesic condition. Let $\exp$ be its exponential map.

\subsection{Symplectic vortices} 

Let $B$ be a closed subset of ${\mb C}$ with smooth boundary (the bordered case will be discussed in Appendix \ref{appendixa}), on which one has the standard Euclidean coordinate $z = s + {\bm i} t$. In many cases we are particularly interested in $B = B_R$, the radius $R$ closed disk centered at the origin, and $B = C_R$, the complement of ${\rm Int}B_R$.
\begin{defn}\hfill
\begin{enumerate}
\item A {\bf gauged map} from $B$ to $X$ is a triple ${\bf v}= (u, \phi, \psi): B \to X \times {\mf k}\times {\mf k}$.

\item A {\bf vortex} on $B$ is a gauged map ${\bf v} = (u, \phi, \psi)$ that solves the equation 
\beq\label{eqn21}
\partial_s u + {\mc X}_\phi(u) + J (\partial_t u + {\mc X}_\psi (u)) = 0,\ \partial_s \psi - \partial_t \phi + [\phi, \psi] + \mu(u) = 0.
\eeq 

\item The {\bf energy} of a gauged map ${\bf v} = (u, \phi, \psi)$ is 
\beqn
E({\bf v}):= \frac{1}{2}  \left( \|\partial_s u +{\mc X}_\phi(u)\|_{L^2(B)}^2 + \|\partial_t u + {\mc X}_\psi(u) \|_{L^2(B)}^2 + \| \partial_s \psi - \partial_t \phi + [\phi, \psi] \|_{L^2(B)}^2 + \| \mu(u) \|_{L^2(B)}^2 \right). 
\eeqn

\item A gauged map ${\bf v} = (u, \phi, \psi)$ is called {\bf bounded} if it has finite energy and the image of $u$ has compact closure. A collection of gauged maps ${\bf v}_\alpha = (u_\alpha, \phi_\alpha,\psi_\alpha)$ is called {\bf uniformly bounded} if their energies have an upper bound and the images of all $u_\alpha$ are contained in a common compact set.

\item An {\bf affine vortex} is a bounded vortex on $B = {\mb C}$.

\end{enumerate}
\end{defn}

\begin{remark}
A few classification results of affine vortices are worth mentioning here. The first is due to Taubes (\cite{Taubes_vortex}, \cite{Jaffe_Taubes}) in the case when the target is the complex line with the standard action of $S^1$. This result was generalized to the case of the standard $S^1$ action on ${\mb C}^n$ by the second named author \cite{Guangbo_vortex}. A classification in the general toric case was provided by Gonz\'alez-Woodward \cite{GW_Toric}. In \cite{VW_affine} the first named author and Woodward proved a Hitchin-Kobayashi correspondence for affine vortices when the target has the structure of an affine or projective variety. 
\end{remark}

There is a coordinate-free formulation of \eqref{eqn21} over a general Riemann surface. However in this paper we do not need the general form but express everything in the local coordinates. A gauge map ${\bf v} = (u, \phi,\psi)$ is also denoted by a pair $(u, a)$ where $a = \phi ds + \psi dt$ is a ${\mf k}$-valued 1-form. We use these two notations interchangeably. The curvature of $a$ is 
\beqn
F_a = \partial_s \psi - \partial_t \phi + [\phi, \psi].
\eeqn 
Moreover, for any ${\mf k}$-valued 1-form $\alpha = \eta ds + \zeta dt$, write
\begin{align*}
&\ d_a \alpha = \partial_s \zeta + [\phi, \zeta] - \partial_t \eta - [\psi, \eta],&\  \dad \alpha = - \Big( \partial_s \eta + [\phi, \eta] + \partial_t \zeta + [\psi, \zeta]\Big).
\end{align*}

A map $g: B \to K$ is called a {\bf gauge transformation}, which acts (on the left) on a gauged map ${\bf v} = (u, \phi, \psi)$ as
\beqn
  g\cdot (u, \phi, \psi) = ( g\cdot u,\  g \phi g^{-1} - g^{-1} \partial_s g, g \psi g^{-1} - g^{-1} \partial_t g ).  
\eeqn
The infinitesimal gauge transformation for $h: B \to {\mf k}$ is 
\beq\label{eqn22}
h \mapsto ( {\mc X}_h(u), \ - \partial_s h - [\phi, h], -\partial_t h - [\psi, h])
\eeq

\begin{defn}\label{defn24}
A sequence of smooth vortices ${\bf v}_i$ on $B$ converges to a limit ${\bf v}_\infty$ in the {\bf compact convergence topology} (c.c.t. for short) if 1) the sequence is uniformly bounded; 2) there exist smooth gauge transformations $g_i: B \to K$ such that $g_i \cdot {\bf v}_i$ converges smoothly to ${\bf v}_\infty$ on compact subsets of $B$; 3) $E({\bf v}_i)$ converges to $E({\bf v}_\infty)$.
\end{defn}

The set of all smooth bounded vortices over $B$ is denoted by $\wt{M}^K (B, X)$. The quotient of $\wt{M}^K(B, X)$ by gauge equivalence is denoted by $M^K(B, X)$, equipped with the compact convergence topology. For ${\bf v}\in \wt{M}^K (B, X)$, denote by $[{\bf v}]\in M^K (B, X)$ the corresponding gauge equivalence class. Given $k \geq 1$ and $p>2$, we also consider $\wt{M}^K(B,X)^{k, p}_{\rm loc}$, i.e., the set of vortices of regularity $W^{k, p}_{\rm loc}$. 

\subsection{Linearized operators}

Abbreviate $E_X = TX$, $E_X^+ = TX \oplus {\mf k} \oplus {\mf k}$ where the last two factors are trivial bundles with fibre ${\mf k}$. For any continuous map $u: B \to X$, denote 
\beqn
E_u = u^* TX,\ E_u^+ = E_u \oplus {\mf k} \oplus {\mf k}.
\eeqn
For any gauged map ${\bf v} = (u, \phi, \psi)$ from $B$ to $X$ and an infinitesimal deformation ${\bm \xi} = (\xi, \eta, \zeta) \in \Gamma(B, E_u^+)$, one can identify a nearby gauged map ${\bf v}' = \exp_{{\bf v}} {\bm\xi}:= (\exp_u \xi, \phi + \eta, \psi  + \zeta)$ where $\exp$ is the exponential map of the metric we chose for which $\mu^{-1}(0)$ is totally geodesic. On the other hand, using the parallel transport of the Levi-Civita connection $\nabla$ of the metric $g$ determined by $\omega_X$ and $J$, one identifies $E_u$ with $E_{u'}$ and hence can differentiate the left-hand-side of \eqref{eqn21}.

The derivative of \eqref{eqn21} at a vortex ${\bf v} = (u, \phi, \psi)$ on $B \subset {\mb C}$ is given by 
\begin{multline*}
{\mc D}_{\bf v}:  \Gamma \big( B, E_u^+ \big) \to \Gamma\big( B, E_u \oplus {\mf k} \big),\\
(\xi, \eta, \zeta)  \mapsto \left( \begin{array}{c}  \partial_s \xi + \nabla_\xi {\mc X}_\phi  + J (\nabla_t \xi + \nabla_\xi {\mc X}_\psi) + (\nabla_\xi J) (\partial_t u + {\mc X}_\psi) + {\mc X}_\eta + J{\mc X}_\zeta\\
\partial_s \zeta + [\phi, \zeta] - \partial_t \eta - [\psi, \eta] + d\mu(u) \cdot \xi \end{array} \right).
\end{multline*}
We will also abbreviate
\beq\label{eqn23}
\begin{split}
\nabla^a \left( \begin{array}{c}\xi \\ \eta\\ \zeta \end{array} \right) = ds \otimes \left( \begin{array}{c} \nabla_s \xi + \nabla_\xi {\mc X}_\phi \\ \partial_s \eta + [\phi, \eta] \\ \partial_s \zeta + [\phi, \zeta] \end{array} \right) + dt \otimes \left( \begin{array}{c}  \nabla_t \xi + \nabla_\xi {\mc X}_\psi \\  \partial_t \eta + [\psi, \eta] \\ \partial_t \zeta + [\psi, \zeta] \end{array} \right)
\end{split}	
\eeq
$\nabla^a$ is then a metric connection on $E_u^+$ with respect to the metric $\omega_X( \cdot, J \cdot)$.

Since the moduli space is defined by quotienting the space of vortices by gauge transformations, we include a gauge fixing condition.
\begin{defn}\label{defn25}
Given a gauged map ${\bf v} = (u, a)$ from $B$ to $X$, another gauged map from $B$ to $X$ of the form ${\bf v}' = (\exp_u \xi, a + \alpha)$ is in {\bf Coulomb gauge} with respect to ${\bf v}$ if
\beq\label{eqn24}
\dad \alpha + d\mu(u) \cdot J \xi =0.
\eeq
\end{defn}
Here the left-hand-side is actually the formal adjoint of \eqref{eqn22}. The {\bf augmented linearized operator} $\hat {\mc D}_{\bf v}: \Gamma ( B, E_u^+ ) \to \Gamma ( B, E_u^+)$ is defined by
\beq
\label{eqn25}
\left( \begin{array}{c} \xi \\ \eta \\ \zeta \end{array}\right) \mapsto \left( \begin{array}{c}  \nabla_s \xi + \nabla_\xi {\mc X}_\phi + J (\nabla_t \xi + \nabla_\xi {\mc X}_\psi) + (\nabla_\xi J) (\partial_t u + {\mc X}_\psi) + {\mc X}_\eta + J {\mc X}_\zeta \\
\partial_s \eta + [\phi, \eta] + \partial_t \zeta + [\psi, \zeta] + d\mu(u) \cdot J \xi \\
\partial_s \zeta + [\phi, \zeta] - \partial_t \eta - [\psi, \eta] + d\mu(u) \cdot \xi 
\end{array} \right)
\eeq
where the left-hand-side of \eqref{eqn24} is inserted as the second entry.

\subsection{Asymptotic behavior of ${\mb C}$-vortices}

Given a gauged map ${\bf v} = (u, \phi, \psi)$, define its energy density function by
\beqn
e({\bf v}) = \frac{1}{2} \Big[ | \partial_s u + {\mc X}_{\phi}(u) |^2 + |\partial_t u + {\mc X}_{\psi} (u)|^2 + |\partial_s \psi - \partial_t \phi + [\phi, \psi] |^2 + |\mu(u)|^2 \Big].
\eeqn

\begin{prop}\label{prop26} \cite[Corollary 1.4]{Ziltener_decay} Given ${\bf v} \in \wt{M}^K(C_R, X)$. Then for any $\eps>0$,
\beq\label{eqn26}
\limsup_{z \to \infty} \Big[ |z|^{4- \epsilon} e({\bf v})(z) \Big] < +\infty.
\eeq
Further, if ${\bf v}_i\in \wt{M}^K(C_R, X)$ is a sequence that converges to ${\bf v}_\infty\in \wt{M}^K(C_R, X)$ in the c.c.t., then 
\beqn
\sup_{i} \limsup_{z \to \infty} \Big[ |z|^{4- \epsilon} e({\bf v}_i)(z) \Big] < +\infty.
\eeqn

\end{prop}
The last statement of the proposition is not included in \cite[Corollary 1.4]{Ziltener_decay}, but it can be proved by the same method of proving \eqref{eqn26}.

\begin{prop}\label{prop27} 
Given $p>2$ and $0 < \gamma <\frac 2 p$. Given ${\bf v} = (u, a)\in \wt{M}^K(C_R, X)$. Suppose ${\bf v}$ is in radial gauge, namely 
\beqn
a = f d\theta,\ f\in C^\infty(C_R, {\mf k}).
\eeqn
Denote $f_r (\theta) = f(r \cos \theta, r \sin \theta)$. Then there exist $x \in \mu^{-1}(0)$ and $k_0 \in W^{1,p}(S^1, K)$ satisfying the following conditions.
\begin{enumerate}

\item \label{prop27a} {\hfil $\displaystyle \lim_{r \to \infty}\max_{\theta \in S^1}d \Big( x, k_0 (\theta)^{-1} u(r \cos \theta, r \sin \theta) \Big) =0$;}

\item \label{prop27b} {\hfil $\displaystyle \limsup_{r \to +\infty}  \left[ r^\gamma \left\| f_r  + (\partial_\theta k_0 ) k_0^{-1}\right\|_{L^p(S^1)} \right] < +\infty$.}

\end{enumerate}
Further, given a sequence ${\bf v}_i = (u_i,  a_i)\in \wt{M}^K(C_R, X)$ such that $a_i = f_i d\theta$ and ${\bf v}_i$ converges to ${\bf v}_\infty = (u_\infty, a_\infty)\in \wt{M}^K(C_R, X)$ in c.c.t., then for all $i$ (including $i=\infty$) there exist $x_i \in \mu^{-1}(0)$ and $k_{0, i}\in W^{1, p}(S^1, K)$ such that $k_{0, i}^{-1}(\theta) u_i (r\cos\theta, r \sin \theta)$ converges to $x_i$ as $r\to \infty$ uniformly in $i$ and 
\beqn
\sup_{1\leq i\leq \infty} \limsup_{r \to +\infty}  \left[ r^\gamma \left\| f_{i, r} + (\partial_\theta k_{0, i}) k_{0, i}^{-1} \right\|_{L^p(S^1)} \right] < +\infty.
\eeqn
\end{prop}

Proposition \ref{prop27} is an improved version of \cite[Proposition D.7]{Ziltener_thesis}. The improved constants are obtained by applying the results in \cite{Ziltener_decay} (see discussion in \cite[Section 5]{VW_affine}).

\begin{remark}\label{rem28} 
Take ${\bf v}= (u, a) \in \wt{M}^K({\mb C}, X)$. By Proposition \ref{prop27}, there is a $K$-bundle $P \to {\mb P}^1$ such that $u$ extends to a {\it continuous} section of $P(X)$, which represents a class $\beta \in H_2^K(X)$.  Using a calculation similar to \cite[Theorem 3.1]{Cieliebak_Gaio_Salamon_2000}, we can derive 
\beq\label{eqn27}
E({\bf v})= \omega_X^K(\beta), 
\eeq
where $\omega_X^K \in H^2_K (X)$ is represented by the equivariant 2-form $\omega_X -\mu$. Ziltener \cite{Ziltener_thesis} proves that when a sequence of affine vortices converges to a (possibly reducible) limit, the equivariant homology class is preserved in the limit. Our main theorem yields a weaker result as a corollary. The existence of the moduli space of regular vortices implies that if a sequence of regular vortices converges to a regular vortex in c.c.t., then the elements of the sequence and the limit have the same equivariant homology class, because they are part of a continuous family.	
\end{remark}
\begin{remark}\label{rem29}
Let $X_K$ be the Borel construction of $X$. The projection $X_K \to BK$ induces a map in equivariant homology $H_2^K(X) \to H_2^K({\rm pt})$. By the previous remark, an affine vortex ${\bf v}$ representing a class $\beta \in H_2^K(X)$ induces a class in $H_2^K({\rm pt})$, which classifies the $K$-bundle $P \to {\mb P}^1$. The topology of $P$ is uniquely determined by the homotopy class of a loop in $K$, hence there is map
\beq\label{eqn28}
\hol: M^K ({\mb C} ,X) \to \pi_1(K).
\eeq
Indeed, $\hol({\bf v})$ is the homotopy class of $k_0$ which is from Proposition \ref{prop27}. We call the class $[k_o]\in \pi_1(K)$ the {\em holonomy of ${\bf v}$}. 
\end{remark}

\begin{notation}
It is often convenient to use a geodesic representative of the class $[k_0] \in \pi_1(K)$, namely a loop $\theta \mapsto e^{\lambda \theta}$ homotopic to $k_0$, where $\lambda \in \frac 1 {2\pi} \exp^{-1}({\rm Id})$. We abbreviate
\beqn
\Lambda_K:= \frac{1}{2\pi} \exp^{-1}({\rm Id}) \subset {\mf k}.
\eeqn
\end{notation}

By Proposition \ref{prop27} one can define the ``evaluation at infinity'',
\beq\label{eqn29}
\ev_\infty:M^K ({\mb C},X) \to \bar{X}.
\eeq
Using Proposition \ref{prop27} it is easy to derive the continuity of $\hol$ and $\ev_\infty$ with respect to the compact convergence topology on $M^K ({\mb C},X)$. The proof is left to the reader.

\begin{cor}\label{cor211} 
Suppose a sequence ${\bf v}_i\in \wt{M}^K({\mb C}, X)$ converge to ${\bf v}_\infty$ in c.c.t.. Then,
\begin{enumerate}
\item \label{cor211a} $\ev_\infty({\bf v}_\infty) = \displaystyle \lim_{i \to \infty}\ev_\infty({\bf v}_i)$ and 

\item \label{cor211b} (for large $i$) $\hol( {\bf v}_i) = \hol( {\bf v}_\infty )\in \pi_1(K)$.
\end{enumerate}
\end{cor}

\subsection{Weighted Sobolev spaces}

In this section, we define weighted Sobolev spaces for functions on ${\mb C}$, that are used to define the domain and target space of the vortex differential operator. Choose a rotationally symmetric smooth function $\rho: {\mb C} \to [1, \infty)$ which is equal to $|z|$ outside a compact subset. For $1<p<\infty$, $\delta \in {\mb R}$ and $k \in {\mb Z}_{\geq 0}$ and a compactly supported smooth function $\sig \in C^\infty_0({\mb C})$, define the following norms:
\beqn
\| \sig \|_{L^{p,\delta}}:= \| \rho^\delta  \sigma \|_{L^p};
\eeqn
\beqn
\| \sig \|_{W^{k,p,\delta}} := \sum_{i=0}^k \| \nabla^i \sig \|_{L^{p,\delta}}, \quad  \| \sig  \|_{L^{k,p,\delta}} :=\sum_{i=0}^k \|\nabla^i 
    \sig \|_{L^{p,\delta+i}}. 
\eeqn
The spaces $L^{p,\delta}({\mb C})$, $W^{k,p,\delta}({\mb C})$ and $L^{k,p,\delta}({\mb C})$ are the completions of $C^\infty_0({\mb C})$ under the norms $\| \cdot \|_{L^{p,\delta}}$, $\| \cdot \|_{W^{k,p,\delta}}$ and $\| \cdot \|_{L^{k,p,\delta}}$ respectively.

\begin{lemma}\label{lemma212}\cite[Proposition 91]{Ziltener_book}\cite[Proposition E.4]{Ziltener_thesis} Given $p>2$, $\delta>1-\frac 2 p$. If $f \in W^{1,p}_{\rm loc}({\mb C})$ satisfies $\| \nabla f \|_{L^{p,\delta}}<\infty$, then the limit $f(\infty):= \displaystyle \lim_{z \to \infty}f(z)$ exists and $f-f(\infty) \in L^{p,\delta-1}$. Moreover, we have the following estimates.
\begin{enumerate}
\item \label{lemma212a} There is a constant $c>0$ independent of $f$ such that 
\beq\label{eqn212}
\| f-f(\infty)  \|_{L^{p,\delta-1}} \leq c \| \nabla f \|_{L^{p,\delta}}.
\eeq
\beq\label{eqn213}
| f(x) - f(y) | \leq  c R^{1 - \frac{2}{p} - \delta} \| \nabla f \|_{L^{p, \delta}( C_R)},\ \forall R>0,\ x, y \in C_R.
\eeq

\item \label{lemma212b} There is an equivalence of norms:
\beq\label{eqn214}
\| f \|_{L^\infty} +  \| \nabla f \|_{L^{p,\delta}} \approx | f(\infty) | +  \| f - f(\infty) \|_{L^{1,p,\delta-1}}.
\eeq
\end{enumerate}
\end{lemma}

\begin{remark}
The equivalence \eqref{eqn214} was not explicitly stated in Ziltener's work, but can be easily deduced from \eqref{eqn212} and \eqref{eqn213}.
\end{remark}

\begin{prop}\label{prop214}
Suppose $\Om \subset {\mb C}$ be a non-compact connected set with smooth boundary.  Let $s_1$, $s_2 \in {\mb Z}_{\geq 0}$, $p_1$, $p_2>0$ and $\delta_1$, $\delta_2 \in {\mb R}$. Further, let $F: W^{s_1,p_1,\delta_1}(\Om) \to W^{s_2,p_2,\delta_2}(\Om)$ be a bounded linear differential operator satisfying the following conditions: 1) for any compact set $S \subset \Om$, the restriction $F|_S: W^{s_1,p_1}(S) \to W^{s_2,p_2}(S)$ is a compact operator; 2) for any $R$, the restriction $F|_{\Om \bs B_R}:W^{s_1,p_1,\delta_1}(\Om \bs B_R) \to W^{s_2,p_2,\delta_2} (\Om\bs B_R)$ is bounded and the operator norm $\| F|_{\Om \bs B_R} \|$ approaches $0$ as $R\to \infty$. Then, the operator $F:W^{s_1,p_1,\delta_1}(\Om) \to W^{s_2,p_2,\delta_2}(\Om)$ is compact.
\end{prop}

The proof is similar to Lemma 2.1 in \cite{choquet} and Proposition
E.6 (v) in \cite{Ziltener_thesis}, but is reproduced for completeness.

\begin{proof}[Proof of Proposition \ref{prop214}] Suppose $\sig_i$ is a bounded sequence in $W^{s_1,p_1,\delta_1}(\Om)$. We assume $\| \sig_i \|_{W^{s_1,p_1,\delta_1}(\Om)} \leq c$ for all $i$. By the Banach-Alaoglu theorem, after passing to a subsequence, the sequence $\sig_i$ has a weak limit $\sig_\infty \in W^{s_1,p_1,\delta_1}(\Om)$. We will show that $F\sig_i$ strongly converges to $F\sig_\infty$ in $W^{s_2,p_2,\delta_2}({\mb C})$. Let $\eps>0$ be an arbitrary small number. Choose $R$ such that $\| F|_{\Om \bs B_R} \|<\frac \eps {2c}$. This implies, in $W^{s_2,p_2,\delta_2}(\Om \bs B_R)$, 
\beq\label{eqn215}
\| F\sig_i-F\sig_\infty \| \leq \| F\sig_i \| +  \| F\sig_\infty \|  < \frac \eps 2.
\eeq
On the compact set $\Om \cap B_R$, the compactness of $F|_{\Om \cap B_R}$ implies that, $F\sig_i|_{\Om \cap B_R}$ converges to $F\sig_\infty|_{\Om \cap B_R}$. Together with \eqref{eqn215}, this proves the proposition.
\end{proof}

The following result is an application of Proposition \ref{prop214}.
\begin{lemma} \label{lemma215}
Suppose $k_1>k_2$, $\delta_1>\delta_2$ and $p>1$. The inclusion $W^{k_1,p,\delta_1} \hookrightarrow W^{k_2,p,\delta_2}$ is compact.
\end{lemma}

We also need to discuss weighted Sobolev spaces with negative orders. For $k \in {\mb Z}_{\geq 0}$, $1<p<\infty$ and $\delta \in {\mb R}$, and $\sig \in C^\infty_0({\mb C})$, define
\beqn
  \| \sig \|_{W^{-k,p,\delta}}:=\sup \Big\{  \langle f, \sigma \rangle_{L^2} \ |\ f \in C_0^\infty, \| f \|_{W^{k,q,-\delta}}=1 \Big\},
\eeqn
where $q:=\frac p {p-1}$. The space $W^{-k,p, \delta}({\mb C})$ is the completion of $C_0^\infty({\mb C})$ under the above norm, which is the dual Banach space of $W^{k,q,-\delta}({\mb C})$. For any first order differential operator $D$ which gives a bounded linear operator $D: W^{1, p, \delta}({\mb C}) \to L^{p, \delta}({\mb C})$, one has the distributional dual
\beqn
D^\dagger: L^{q, -\delta}({\mb C}) \to W^{-1, q, -\delta}({\mb C}) \simeq ( W^{1, p, \delta}({\mb C}))^*
\eeqn
defined by $\langle D^\dagger u, \varphi \rangle = \langle u, D \varphi \rangle$ for all test functions $\varphi \in C_0^\infty({\mb C})$.

Lastly we would like to compare Sobolev norms defined by using different connections. 

\begin{lemma}\label{lemma216}
Given $p>2$, $\delta > 0$ and $\Omega \subset {\mb C}$. Let $E \to \Omega$ be an Euclidean vector bundle equipped with a covariant derivative $\nabla^E$ which preserves the metric. Suppose $a \in W^{1, p, \delta}( \Omega, \Lambda^1 \otimes {\mf o}(E))$. Then for $\sigma \in W^{1, p}_{\rm loc}(E)$, there are the following two equivalences of norms.
\beqn
\| \sigma \|_{L^{p, \delta}(\Omega)} + \| \nabla^E \sig  \|_{L^{p,\delta}(\Omega)} \approx \|\sigma\|_{L^{p, \delta}(\Omega)} + \| \nabla^E \sigma + a(\sigma) \|_{L^{p, \delta}(\Omega)},
\eeqn
\beqn
\| \sigma \|_{L^\infty(\Omega)} + \| \nabla^E \sigma \|_{L^{p,\delta}(\Omega)} \approx \|\sigma \|_{L^\infty(\Omega)} +  \| \nabla^E \sigma + a(\sigma) \|_{L^{p,\delta}(\Omega)}.
\eeqn

\end{lemma}
\begin{proof}
This lemma can be proved by straightforward calculation and we left the detail to the reader.
\end{proof}

\section{Proof of Main Theorem}\label{section3}

Given an affine vortex ${\bf v} = (u, \phi, \psi)$, we first describe a Banach space parametrizing gauged maps near ${\bf v}$. Given $p>2$ and $\delta \in (1 -\frac 2 p,1)$, we define the following norm for ${\bm \xi} = (\xi, \eta, \zeta) \in W^{1, p}_{\rm loc} ({\mb C}, E_u^+ )$ as
\begin{multline}\label{eqn31}
\|  {\bm \xi} \|_{p,\delta}:= \| \eta \|_{L^{p,\delta}} +  \|\zeta\|_{L^{p, \delta}} + \|\nabla^a \eta \|_{L^{p,\delta}}  + \| \nabla^a \zeta \|_{L^{p, \delta}} \\
+  \| \xi \|_{L^\infty} + \| d\mu(u) \cdot \xi \|_{L^{p,\delta}} +  \| d\mu(u) \cdot J\xi  \|_{L^{p,\delta} } +  \| \nabla^a \xi \|_{L^{p,\delta}}.
\end{multline}
This norm is equivalent to the norm defined in \cite[Equation (1.18)]{Ziltener_book} although the latter has a different expression. Define
\beqn
\hat{\mc B}_{\bf v}^{p, \delta}:= \Big\{ {\bm \xi} \in W^{1, p}_{\rm loc} ( {\mb C}, E_u^+ )\ |\ \|  {\bm \xi} \|_{p,\delta} < \infty \Big\}.
\eeqn

\begin{lemma}\label{lemma31}
The norm \eqref{eqn31} is complete and gauge invariant. 
\end{lemma}
This lemma was proved by Ziltener (see \cite[Theorem 4, part (i)]{Ziltener_book}\footnote{However, we remark that the condition ${\rm dim} X > 2 {\rm dim} K$ assumed by Ziltener is unnecessary for all the statements of \cite[Theorem 4]{Ziltener_book}.}). 

When $p$ and $\delta$ are understood from the context, we will drop it from the notation. For $\epsilon>0$ sufficiently small, let $\hat{\mc B}_{\bf v}^\epsilon \subset \hat{\mc B}_{\bf v}$ be the $\epsilon$-ball centered at the origin. We identify ${\bm \xi}  \in \hat{\mc B}_{{\bf v}}^\epsilon$ with the gauged map ${\bf v}' := \exp_{{\bf v}} {\bm \xi}:= (u', \phi', \psi'):= ( \exp_u \xi, \phi + \eta, \psi + \zeta)$. Consider the infinite-dimensional vector bundle $\hat{\mc E}:= \hat{\mc E}^{p, \delta} \to \hat{\mc B}_{\bf v}^\epsilon$, whose fiber over the point ${\bf v}'$ is
\beqn
\hat{\mc E}_{{\bf v}'}: = \hat{\mc E}_{\bm \xi}:= L^{p,\delta} ( {\mb C},  E_{u'}^+ )  = L^{p, \delta}({\mb C}, (u')^* TX \oplus {\mf k}\oplus {\mf k}).
\eeqn
Consider the section $\hat{\mc F}_{\bf v}: \hat{\mc B}_{\bf v}^\epsilon \to \hat{\mc E}$ defined by
\beq\label{eqn32}
\hat{\mc F}_{\bf v} ({\bf v}') = \left( \begin{array}{c}  \partial_s u' + {\mc X}_{\phi'}(u') + J (\partial_t u' + {\mc X}_{\psi'}(u'))\\
 \partial_s \psi' - \partial_t \phi' + [\phi', \psi'] + \mu(u')\\ 
\partial_s \eta + [\phi, \eta] + \partial_t \zeta + [\psi, \zeta] + \d \mu(u) \cdot  J \xi \end{array} \right).
\eeq
One can trivialize $\hat{\mc E}$ by using parallel transport to identify $(u')^* TX$ with $u^* TX$, and $\hat{\mc F}_{\bf v}$ then becomes a smooth map\footnote{The smoothness of $\hat{\mc F}_{\bf v}$ is not very obvious. Indeed, one can use Proposition \ref{prop36} below to transform ${\bf v}$ to a good gauge. Then the norm \eqref{eqn31} is equivalent to the Sobolev norm using the Levi-Civita connection. Then one can argue that each term of $\hat{\mc F}_{\bf v}$ depends smoothly on points of $\hat{\mc B}_{\bf v}^\epsilon$ with respect to this new, gauge-dependent norm.} from $\hat{\mc B}_{\bf v}^\epsilon$ to $\hat{\mc E}_{\bf v}$. Its linearization at the origin has the same expression as the operator $\hat{\mc D}_{\bf v}$ given by \eqref{eqn25}. It is not hard to see that $\hat{\mc D}_{\bf v}$ extends to a bounded linear operator (which is denoted by the same symbol)
\beq\label{eqn33}
\hat{\mc D}_{\bf v}: \hat{\mc B}_{\bf v} \to \hat{\mc E}_{\bf v}.
\eeq

For constructing a manifold chart around regular affine vortices using the zero locus of $\hat{\mc F}_{\bf v}$, a necessary ingredient is the linear Fredholm theory of the linearized operator. In Section \ref{section4} we reprove the following results of Ziltener.

\begin{prop}\label{prop32}\cite{Ziltener_book}
Suppose ${\bf v}\in \wt{M}^K({\mb C},X; \beta)$. Then for any $p>2$ and $\delta \in (1-\frac{2}{p}, 1)$, the operator \eqref{eqn33} is Fredholm and its real index is
\beq\label{eqn34}
{\rm ind} \big( \hat{\mc D}_{\bf v} \big) = \dim(\bar{X}) + 2 \langle c_{\, 1}^K(TX), \beta \rangle.
\eeq
\end{prop}

\begin{defn}\label{defn33}
A bounded affine vortex ${\bf v}$ is called {\bf regular} if the operator \eqref{eqn33} is surjective for all $p>2$ and $\delta \in (1-\frac{2}{p}, 1)$. Let $\wt{M}^K ({\mb C}, X; \beta)^{\rm reg}$ be the set of all smooth, regular, bounded affine vortices representing $\beta\in H_2^K (X; {\mb Z})$. Let $M^K ({\mb C}, X; \beta)^{\rm reg}$ be the quotient set of $\wt{M}^K ({\mb C}, X; \beta)^{\rm reg}$ modulo smooth gauge transformations, equipped with the compact convergence topology.
\end{defn}

\subsection{Constructing the local charts}

Given ${\bf v} \in \wt{M}^K ({\mb C}, X; \beta)^{\rm reg}$. By the smoothness of $\hat{\mc F}_{\bf v}$, for $\epsilon$ small enough
\beqn
{\bf v}' \in \hat{\mc B}_{\bf v}^\epsilon \Longrightarrow {\rm coker} \big( \d \hat{\mc F}_{\bf v} |_{{\bf v}'} \big) = \{0\}.
\eeqn
Denote
\beqn
U_{\bf v}^\epsilon: = \hat{\mc F}_{\bf v}^{-1}(0) \cap \hat{\mc B}_{\bf v}^\epsilon.
\eeqn
By the implicit function theorem and Proposition \ref{prop32}, for $\epsilon$ sufficiently small, $U_{\bf v}^\epsilon$ is a smooth manifold, whose dimension is equal to ${\rm ind} ( \hat{\mc D}_{\bf v} )$. Therefore, it induces an (obviously continuous) map 
\begin{align*}
& \hat\phi_{\bf v}^\epsilon: U_{\bf v}^\epsilon \to M^K ({\mb C}, X; \beta),  &  {\bf v}' \mapsto \big[ {\bf v}' \big].
\end{align*}

In order to make $\hat\phi_{\bf v}^\epsilon$ a chart of smooth manifold, we have to show it is one-to-one and its image contains a neighborhood of $[ {\bf v} ]$ in $M^K({\mb C}, X; \beta)^{\rm reg}$. This type of question is nearly trivial when the domain is compact, for example, when discussing the moduli space of regular holomorphic curves or the moduli space of regular symplectic vortices over a compact Riemann surface (see \cite{Cieliebak_Gaio_Mundet_Salamon_2002}). In our case the difficulty appears because smooth objects {\it a priori} may not have the desired regularity at infinity. The following two propositions are proved in Subsection \ref{ss32}, relying on a few results whose proof occupy the remaining sections.

\begin{prop}\label{prop34}
There exists $\epsilon_0>0$, which may depend on $p$ and $\delta$, such that $\hat\phi_{\bf v}^{\epsilon_0}$ is one-to-one.
\end{prop}

\begin{prop}\label{prop35}
For each $\epsilon\in (0, \epsilon_0]$, there is a neighborhood of $[{\bf v}]$ in $M^K ({\mb C}, X; \beta)$ contained in the image of $\hat{\phi}_{\bf v}^\epsilon$.
\end{prop}

It is easy to prove that $M^K ({\mb C}, X; \beta)$ is Hausdorff. Hence Proposition \ref{prop34} and Proposition \ref{prop35} imply that for $\epsilon$ sufficiently small, $\hat\phi_{\bf v}^\epsilon$ is a homeomorphism onto its image, which is denoted by $V_{\bf v}^\epsilon$. Therefore it gives a chart of smooth manifold of the correct dimension on $M^K ({\mb C}, X; \beta)$ over a neighborhood of $[{\bf v}]$. Let $\epsilon({\bf v})$ be the maximal $\epsilon$ such that $\hat\phi_{\bf v}^\epsilon$ is a homeomorphism.

So far we have constructed a collection of charts (which depend on $p$ and $\delta$)
\beqn
{\ms C}^\infty:= \Big\{  \Big(  U_{\bf v}^\epsilon, \hat\phi_{\bf v}^\epsilon \Big)\ |\ {\bf v} \in \wt{M}^K ( {\mb C}, X; \beta),\ \epsilon \in (0, \epsilon({\bf v}))\Big\}.
\eeqn
Moreover, since every bounded affine vortex of regularity $W^{1, p}_{\rm loc}$ is gauge equivalent via a gauge transformation of regularity $W^{2, p}_{\rm loc}$ to a smooth one, we can enlarge the collection by including charts centered at vortices ${\bf v}$ of regularity $W^{1, p}_{\rm loc}$. The enlarged collection is denoted by ${\ms C}^{1, p}_{\rm loc}$. Let $V_{\bf v}^\epsilon\subset M^K({\mb C}, X;\beta)$ be the image of $\hat\phi_{\bf v}^\epsilon$. Every $g \in W^{2,p}_{\rm loc}({\mb C}, K)$ induces a diffeomorphism $g: U_{\bf v}^\epsilon \to U_{g{\bf v}}^\epsilon$ such that the following diagram commutes
\beq\label{eqn35}
\xymatrix{ U_{\bf v}^\epsilon \ar[r]^{\hat\phi_{\bf v}^\epsilon } \ar[d]^g & V_{\bf v}^\epsilon \ar[d]^{\rm id}\\
           U_{g{\bf v}}^\epsilon \ar[r]^{\hat\phi_{g{\bf v}}^\epsilon}        & V_{g{\bf v}}^\epsilon }.
\eeq

\subsection{Proof of Proposition \ref{prop34} and Proposition \ref{prop35}}\label{ss32}

The proofs depend on the following few technical results. We always assume that $p>2$ and $1- \frac{2}{p} < \delta < 1$. Firstly, in Section \ref{section5}, we prove the following uniform bound on the connection form.

\begin{prop}\label{prop36}
Given $R\geq 1$ and a sequence ${\bf v}_i = (u_i, a_i)\in \wt{M}^K(C_R, X)$ that converges to ${\bf v}_\infty = (u_\infty, a_\infty)\in \wt{M}^K(C_R, X)$ in c.c.t., then there exist a sequence $\check{g}_i \in {\mc K}^{2, p}_{\rm loc}(C_R)$ 
(including $i = \infty$) such that, if we denote $\check{\bf v}_i = ( \check{u}_i, \check{a}_i):= \check{g}_i \cdot {\bf v}_i$, then 
\beqn
\displaystyle \sup_{1 \leq i \leq \infty} \| \check{a}_i\|_{W^{1,p, \delta}(C_R)}  < +\infty
\eeqn
\end{prop}

In Section \ref{section6} we prove the following result, which says that a sequence of affine vortices converging in c.c.t. will also converge, modulo gauge, in the topology induced by a weighted Sobolev norm.

\begin{prop}\label{prop37}
Suppose a sequence ${\bf v}_i^o\in \wt{M}^K({\mb C}, X)$ converges to a limit ${\bf v}_\infty^o$ in c.c.t.. Then there exist $\lambda \in \Lambda_K$, and a sequence of gauge transformations $g_i \in {\mc K}^{2,p}_{\rm loc}({\mb C})$ 
(including $i= \infty$) such that the following are satisfied. Denote $(u_i, a_i) = g_i \cdot (u_i^o, a_i^o)$. Then $u_i$ converges to $u_\infty$ uniformly. Moreover, for $i$ sufficiently large, define $\xi_i \in W^{2, p}_{\rm loc}( {\mb C}, E_{u_\infty})$ by $\xi_i = \exp_{u_\infty}^{-1} u_i$. Then 
\beqn
\lim_{i \to \infty} \Big( \| a_i - a_\infty \|_{L^{p, \delta}} +  \| F_{a_i} - F_{a_\infty} \|_{L^{p, \delta}} + \| \nabla^{a_\infty} \xi_i\|_{L^{p,\delta}} +  \| d\mu(u_\infty) \cdot J\xi_i \|_{L^{p, \delta}} + \| d\mu(u_\infty) \cdot \xi_i \|_{L^{p, \delta}} \Big) = 0.
\eeqn
\end{prop}

By this proposition one can define ${\bm \xi}_i = (\xi_i, a_i - a_\infty)$ such that at least formally $\exp_{{\bf v}_\infty} {\bm \xi}_i = {\bf v}_i$. To prove that ${\bm \xi}_i\in \hat{\mc B}_{{\bf v}_\infty}$ and to control its norm, one has to estimate the derivatives of $a_i - a_\infty$, which requires certain gauge fixing. The following proposition says that one can transform vortices into Coulomb gauge relative to the limit over the {\it unbounded} domain ${\mb C}$. It will be proved in Section \ref{section7}.
\begin{prop}\label{prop38}
Suppose $k=0,1$. Suppose ${\bf v} = (u, a) \in \wt{M}^K({\mb C}, X)^{1, p}_{\rm loc}$ such that $a + \lambda d\theta \in W^{1, p, \delta}(C_R,\Lambda^1 \otimes {\mf k})$ for some $\lambda \in \Lambda_K$ and all $R \geq 1$. Then there are constants $\epsilon_{\rm coul}, c_{\rm coul}, \delta_{\rm coul}>0$ satisfying the following condition. For any ${\bm \xi} = (\xi, \alpha)$ where $\xi \in W^{k, p}_{\rm loc}({\mb C}, E_u)$, $\alpha \in W_{\rm loc}^{k, p}({\mb C}, \Lambda^1 \otimes {\mf k})$, which satisfy
\beq\label{eqn36}
\| \alpha \|_{W^{k,p,\delta}}  + \| \xi \|_{L^\infty} + \|  d_a^{\,*} \alpha + d\mu(u) \cdot J \xi \|_{W^{k-1,p,\delta}} \leq \epsilon_{\rm coul},    
\eeq
there exists a unique $s \in W^{k+1, p, \delta}({\mb C}, {\mf k})$ satisfying the following conditions

\begin{enumerate}
\item $e^s \exp_{{\bf v}} {\bm \xi}$ is in Coulomb gauge with respect to ${\bf v}$.

\item $\| s \|_{W^{k+1, p, \delta}} \leq \delta_{\rm coul}$.
\end{enumerate}
Further, $s$ satisfies the estimate
\beqn
\| s \|_{W^{k+1,p,\delta}} \leq c_{\rm coul} \| \dad \alpha + d\mu(u) \cdot J \xi \|_{W^{k-1,p,\delta}}.
\eeqn

\end{prop}

\begin{remark}
When $k=0$, one has the distributional adjoint $d_a^{\,\dagger}: L^{p, \delta}(\Lambda^1 \otimes {\mf k}) \to W^{-1, p, \delta}({\mf k})$ of $d_a$. The restriction of $d_a^{\, \dagger}$ to $W^{1, p}_{\rm loc} \cap L^{p, \delta}$ coincides with the formal adjoint $\dad$, so we abused the notation in the statement of the above proposition.
\end{remark}

The last piece is the following local interior regularity of connections, which is proved at the end of this section.
\begin{lemma}\label{lemma310}
Suppose $\Om', \Om$ are compact sets with smooth boundary in ${\mb H}$, such that $\Om' \subset \on{int}(\Om)$. Suppose $p>2$. Then for any $M >0$, there exists a constant $c( M )>0$, also depending on $\Omega, \Omega'$ and $p$, that satisfies the following conditions.

Suppose there are two connections forms $a, b \in W^{1, p}_{\rm loc}(\Omega, \Lambda^1 \otimes {\mf k})$, satisfying $
* (a - b)|_{\Omega\cap {\mb R}} = 0$ and
\beq\label{eqn37}
\| a - b\|_{L^p(\Omega)} + \| F_a - F_b \|_{L^p(\Omega)} + \| \dad (a - b) \|_{L^p(\Omega)} \leq M.
\eeq
\beq\label{eqn38}
\| a \|_{L^\infty(\Omega)} \leq M.
\eeq
Then
\beq\label{eqn39}
\big\| a - b \big\|_{W^{1, p}(\Omega')} \leq c(M) \Big[ \| a - b \|_{L^p(\Omega)}+ \| F_a - F_b \|_{L^p(\Omega)} + \| \dad (a- b) \|_{L^p(\Om)} \Big].
\eeq
\end{lemma}

Now we can prove Proposition \ref{prop34} and Proposition \ref{prop35}.

\begin{proof}[Proof of Proposition \ref{prop34}]
We prove by contradiction. Suppose there exist a sequence $\epsilon_i \to 0$ and two sequences ${\bm \xi}_i, {\bm \xi}_i' \in U_{\bf v}^{\epsilon_i}$, such that, denoting ${\bf v}_i = \exp_{{\bf v}} {\bm \xi}_i = (u_i, a_i)$, ${\bf v}_i' = \exp_{{\bf v}_i} {\bm \xi}_i' = (u_i', a_i')$,
\beqn
[{\bf v}_i] = [ {\bf v}_i'] \in M^K ({\mb C}, X; \beta).
\eeqn
Then there exist a sequence of smooth gauge transformations $g_i = e^{s_i}$ such that 
\beq\label{eqn310}
e^{s_i} {\bf v}_i = {\bf v}_i'.
\eeq
To derive a contradiction, firstly we would like to show that (for large $i$) $s_i \in W^{1, p, \delta}$. 

Since the norm of $\hat{\mc B}_{\bf v}$ controls the $L^\infty$ norm, there exists $c_1>0$ such that, 
\beq\label{eqn311}
\sup_{z\in {\mb C}} {\rm dist} ( u_i(z), u_i'(z)) \leq c_1 \Big[ \| \xi_i \|_{L^\infty} + \| \xi_i' \|_{L^\infty}\Big] \leq 2 c_1 \epsilon_i.
\eeq
Let $R>0$ such that $u(C_R)$ is contained in a neighborhood of $\mu^{-1}(0)$ where the $K$-action is free. Since locally ${\rm ker} (d\mu \circ J)$ parametrizes a slice of the $K$-action, for $i$ sufficiently large, there exist unique $\wt{s}_i, \wt{s}_i': C_R \to {\mf g}$, $\wt\xi_i, \wt\xi_i' \in \Gamma( C_R, {\rm ker} ( d \mu(u) \circ J) )$ such that 
\beqn
u_i = e^{\wt{s}_i} \exp_u \wt\xi_i,\ u_i' = e^{\wt{s}_i'} \exp_u \wt\xi_i'.
\eeqn
Moreover, there exists $c_2>0$ such that $| \wt{s}_i| \leq c_2| d\mu(u) \cdot J \xi_i|$, $|\wt{s}_i'| \leq c_2 |d\mu( u) \cdot J \xi_i'|$ for all large $i$. By \eqref{eqn310}, $e^{s_i} u_i = u_i'$ and hence over $C_R$,
\beqn
\wt\xi_i = \wt\xi_i',\ \wt{s}_i = \wt{s}_i' + s_i.
\eeqn
Therefore 
\beqn
\| s_i \|_{L^{p, \delta}(C_R)} \leq \| \wt{s}_i \|_{L^{p, \delta}(C_R)} + \| \wt{s}_i' \|_{L^{p, \delta}(C_R)} \leq c_2 \Big[ \| d \mu(u) \cdot J \xi_i \|_{L^{p, \delta}} + \| d \mu(u) \cdot J \xi_i' \|_{L^{p, \delta}} \Big] \leq 2c_2 \epsilon_i.
\eeqn

On the other hand, by using a smooth gauge transformation, we may assume $\| a \|_{L^\infty} < +\infty$. Since we have
\beq\label{eqn312}
a + \alpha_i' = - d s_i + e^{s_i}( a + \alpha_i)e^{-s_i},
\eeq
there are constants $c_3, c_4, c_5>0$ such that 
\beqn
\| d s_i \|_{L^{p, \delta}(C_R)} \leq \| \alpha_i \|_{L^{p, \delta}} + \| \alpha_i' \|_{L^{p, \delta}} +	c_3 \| s_i \|_{L^{p, \delta}(C_R)} \leq c_4 \epsilon_i.
\eeqn
\beqn
\| d s_i \|_{L^p(B_R)} \leq \| \alpha_i \|_{L^{p, \delta}} + \| \alpha_i' \|_{L^{p, \delta}} + 2 \| a \|_{W^{1, p}(B_R)} \leq c_5.
\eeqn
Therefore there is $c_6>0$ such that 
\beqn
\| s_i \|_{W^{1, p, \delta}(C_R)} \leq c_6 \epsilon_i,\ \| s_i \|_{W^{1,p, \delta}} \leq c_6.
\eeqn

Take $\delta_- \in ( 1- \frac{2}{p}, \delta)$. Then as $i \to \infty$, by Lemma \ref{lemma215}, there exists a subsequence of $s_i$ (still indexed by $i$) which converge to $s_\infty$ weakly in $W^{1,p, \delta}$ and strongly in $L^{p, \delta_-}$. Moreover,  since
\beqn
0 = \dad \alpha_i' + d\mu(u) \cdot J \xi_i'   = \dad ( - d s_i + e^{s_i} ( a + \alpha_i) e^{-s_i}  - a) + d\mu(u) \cdot J\xi_i',
\eeqn
$s_\infty$ is a weak solution of a second order linear equation. Since $s_\infty$ vanishes on $C_R$, by the unique continuation principle, $s_\infty \equiv 0$. Hence $\| s_i \|_{L^{p, \delta_-}} \to 0$. Then using \eqref{eqn312} again, we have
\beqn
\| d s_i \|_{W^{1, p, \delta_-}} \leq \| \alpha_i \|_{L^{p, \delta_-}} + \| \alpha_i' \|_{L^{p, \delta_-}} + c_3 \| s_i \|_{L^{p, \delta_-}}  \to 0.
\eeqn
Therefore $\| s_i \|_{W^{1, p, \delta_-}} \to 0$. 	
However, for $i$ large enough, this contradicts the uniqueness part of the $k=1$ case of Proposition \ref{prop38} (with $\delta$ replaced by $\delta_-$).
\end{proof}

\begin{proof}[Proof of Proposition \ref{prop35}]
Suppose the statement is not true. Then there exists a sequence of affine vortices ${\bf v}_i^o$ which converges to ${\bf v}_\infty^o = {\bf v}$ in c.c.t. but none of $[{\bf v}_i^o]$ lies in the image of $\hat\phi_{{\bf v}}^\epsilon$ for any $\epsilon > 0$. By Proposition \ref{prop37}, there exist $i_0\in{\mb N}$ such that, for all $i \geq i_0$ (including $i = \infty$) we can transform ${\bf v}_i^o$ via a gauge transformation of regularity $W^{2, p}_{\rm loc}$ to affine vortices ${\bf v}_i = (u_i, a_i)$ such that $u_i$ converges to $u_\infty$ uniformly and, if denote $\xi_i = \exp_{u_\infty}^{-1} u_i$, $\alpha_i = a_i - a_\infty$ for $i\geq i_0$, then 
\beq\label{eqn313}
\lim_{i \to \infty} \Big[ \| \xi_i\|_{L^\infty} + \| \alpha_i\|_{L^{p, \delta}} + \| \nabla^{a_\infty} \xi_i \|_{L^{p, \delta}} + \| d\mu(u_\infty) \cdot J \xi_i \|_{L^{p, \delta}} + \| d\mu(u_\infty) \cdot \xi_i \|_{L^{p, \delta}} \Big] = 0.
\eeq
Denote ${\bm \xi}_i = (\xi_i, \alpha_i)$. We would like to show that ${\bm \xi}_i\in \hat{\mc B}_{{\bf v}_\infty}$ and estimate its norm. The only missing piece is the norm of the derivatives of $\alpha_i$. Indeed, \eqref{eqn313} implies that 
\begin{multline*}
\lim_{i \to \infty}\Big[ \| \xi_i \|_{L^\infty} +  \| \alpha_i \|_{L^{p, \delta}}  + \| d_{a_\infty}^{\, *} \alpha_i + d\mu(u) \cdot J \xi_i \|_{W^{-1, p, \delta}({\mb C})} \Big] \\
\leq \lim_{i \to \infty} \Big[ \| \xi_i \|_{L^\infty} + \| \alpha_i \|_{L^{p, \delta}({\mb C})} + \| a_\infty \|_{L^\infty} \| \alpha_i \|_{L^{p, \delta}({\mb C})} +  \| d\mu(u) \cdot J \xi_i \|_{L^{p, \delta}({\mb C})} \Big] = 0.
\end{multline*}
Hence by the $k=0$ case of Proposition \ref{prop38}, for large $i$, one can turn ${\bf v}_i$ into Coulomb gauge relative to ${\bf v}_\infty$, i.e., there are gauge transformations $e^{s_i}$ with $s_i \in W^{1, p, \delta}$, such that
\beq\label{eqn314}
\| s_i \|_{W^{1, p, \delta}({\mb C})} \leq c_{\rm coul} \| d_{a_\infty}^{\, *} \alpha_i + d\mu(u) \cdot J \xi_i \|_{W^{-1, p, \delta}({\mb C})},
\eeq
and, if we denote $\wt{\bf v}_i = (\wt{u}_i, \wt{a}_i) = e^{s_i} \cdot {\bf v}_i = \exp_{{\bf v}} \wt{\bm \xi}_i$, $\wt{\bm \xi}_i = (\wt\xi_i, \wt\alpha_i)$, then  
\beq\label{eqn315}
d_{a_\infty}^{\,*} \wt\alpha_i + d\mu(u) \cdot J \wt\xi_i = 0. 
\eeq
Moreover, since $a_i$ and $a_\infty$ are both of regularity $W^{1, p}_{\rm loc}$, $e^{s_i}$ is indeed of regularity $W^{2, p}_{\rm loc}$ and $\wt{\bm \xi}_i$ is of regularity $W^{1, p}_{\rm loc}$. We would like to show that $\| \wt{\bm \xi}_i \|_{p, \delta}$ (which could be infinity {\it a priori}) converges to zero. 

We first estimate $\wt\alpha_i$. Since the right-hand-side of \eqref{eqn314} converges to zero, we have
\beq\label{eqn316}
\lim_{i \to \infty} \| s_i \|_{W^{1, p, \delta}} = 0.
\eeq
Moreover, by Proposition \ref{prop37}, $\| F_{a_i} - F_{a_\infty} \|_{L^{p, \delta}}$ converges to zero. So  
\beq\label{eqn317}
\lim_{i \to \infty} \Big[ \| \wt\alpha_i \|_{L^{p, \delta}} + \| F_{\wt{a}_i} - F_a \|_{L^{p, \delta}} + \| d_{a_\infty}^{\,*} \wt\alpha_i \|_{L^{p, \delta}} \Big] = 0.
\eeq
For any $z \in {\mb C}$, we apply the $(\Omega, \Omega') = (B_2(z), B_1(z))$ case of Lemma \ref{lemma310} to $\wt{a}_i$ and $a_\infty$. Notice that by \eqref{eqn317}, the prerequisites \eqref{eqn37} and \eqref{eqn38} hold for certain $M >0$ independent of $z$ and $i$. Therefore there is a constant $c > 0$ (independent of $z$ and $i$) such that 
\beqn
 \| \wt\alpha_i \|_{W^{1, p}(B_1(z))} \leq c \Big[ \| \wt\alpha_i \|_{L^p(B_2(z))} + \| F_{\wt{a}_i} - F_{a_\infty} \|_{L^p(B_2(z))} + \| d_{a_\infty}^{\,*} \wt\alpha_i \|_{L^p(B_2(z))} \Big].
\eeqn
Cover ${\mb C}$ by countably many unit disks $B_1(z_k)$ ($k=1, 2, \ldots$), we see that there are constants $c', c'', c'''>0$ (independent of $i$) such that 
\begin{multline*}
\| \wt\alpha_i \|_{W^{1, p, \delta}({\mb C})} \leq c' \sum_{k = 1}^\infty  ( 1 + |z_k|)^{\delta} \|  \wt\alpha_i \|_{W^{1, p}(B_1(z_k))}\\
\leq c'' \sum_{k=1}^\infty ( 1 + |z_k|)^{\delta} \Big[ \| \wt\alpha_i \|_{L^p(B_2(z_k))} + \| F_{\wt{a}_i} - F_{a_\infty} \|_{L^p(B_2(z_k))} + \| d_{a_\infty}^{\,*} \wt\alpha_i \|_{L^p(B_2(z_k))} \Big] \\
\leq  c''' \Big[ \| \wt\alpha_i \|_{L^{p, \delta}({\mb C})} + \| F_{\wt{a}_i} - F_{a_\infty} \|_{L^{p, \delta}({\mb C})} + \| d_{a_\infty}^{\,*}  \wt\alpha_i \|_{L^{p, \delta}({\mb C})} \Big].
\end{multline*}
By \eqref{eqn317}, $\| \wt\alpha_i \|_{W^{1, p, \delta}}$ converges to zero. Since $|a_\infty|_{L^\infty} < +\infty$, it follows that $\| \wt\alpha_i \|_{L^{p, \delta}} + \| \nabla^{a_\infty} \wt\alpha_i\|_{L^{p, \delta}}$ is comparable to $\| \wt\alpha_i\|_{W^{1, p, \delta}}$. Hence 
\beq\label{eqn318}
\lim_{i \to \infty} \Big[ \| \wt\alpha_i \|_{L^{p, \delta}} + \| \nabla^{a_\infty} \wt\alpha_i \|_{L^{p, \delta}}\Big] = 0.
\eeq

On the other hand, from \eqref{eqn313}, \eqref{eqn316} one can see that 
\beq\label{eqn319}
\lim_{i \to \infty} \Big[ \| \wt\xi_i\|_{L^\infty} + \|\nabla^{a_\infty} \wt\xi_i\|_{L^{p, \delta}} + \| d\mu(u)\cdot J \wt\xi_i \|_{L^{p, \delta}} + \| d\mu(u) \cdot \wt\xi_i \|_{L^{p, \delta}} \Big] = 0.
\eeq
\eqref{eqn318}, \eqref{eqn319} and the above convergence show that $\wt{\bm \xi}_i \to {\bm 0}$ in $\hat{\mc B}_{{\bf v}_\infty}$. By the implicit function theorem, for each $\epsilon>0$ and $i$ sufficiently large, $\wt{\bm \xi}_i \in U_{{\bf v}_\infty}^\epsilon$. So $[{\bf v}_i^o] = [{\bf v}_i] = [\wt{\bf v}_i]\in {\rm Im} ( \hat\phi_{\bf v}^\epsilon)$, which is a contraction. This finishes the proof of Proposition \ref{prop35}.
\end{proof}

\subsection{Constructing transition functions}\label{ss33}

The last step of proving the main theorem is to show that the various charts of $M^K ({\mb C}, X; \beta)$ constructed so far can be glued smoothly into a manifold atlas. 

We first need to construct transition functions among the Banach charts. Let ${\bf v} = (u, a)$ be an affine vortices (not necessarily regular). Denote
\beqn
{\mc B}_{\bf v}:= \Big\{ {\bm \xi} = (\xi, \alpha) \in \hat{\mc B}_{{\bf v}}\ |\ \dad \alpha + d\mu(u) \cdot J \xi = 0\Big\}.
\eeqn
Since $(\xi, \alpha) \mapsto \dad \alpha + d\mu(u) \cdot J \xi$ is a bounded linear map, ${\mc B}_{\bf v}$ is a closed subspace of $\hat{\mc B}_{{\bf v}}$ and hence also a Banach space. For $\epsilon>0$ sufficiently small, denote ${\mc B}_{\bf v}^\epsilon:= {\mc B}_{\bf v} \cap \hat{\mc B}_{\bf v}^\epsilon$. 

Suppose two charts ${\mc B}_{{\bf v}_1}^{\epsilon_1}$ and ${\mc B}_{{\bf v}_2}^{\epsilon_2}$ have an overlap
\beqn
{\mc U}_{12}:= \Big\{ ( {\bm \xi}_1, {\bm \xi}_2) \in {\mc B}_{{\bf v}_1}^{\epsilon_1} \times {\mc B}_{{\bf v}_2}^{\epsilon_2}\ |\ \exp_{{\bf v}_1} {\bm \xi}_1\ {\rm is\ gauge\ equivalent\ to\ } \exp_{{\bf v}_2} {\bm \xi}_2 \Big\}.
\eeqn
There are obvious inclusion maps $\varphi_i: {\mc U}_{12} \hookrightarrow {\mc B}_{{\bf v}_i}^{\epsilon_i}$ for $i = 1, 2$. 

\begin{prop}
$\varphi_1({\mc U}_{12})$ and $\varphi_2({\mc U}_{12})$ are open and the map $\varphi_1 \circ \varphi_2^{-1}: \varphi_2({\mc U}_{12}) \to \varphi_1({\mc U}_{12})$ is a diffeomorphism between open subsets of Banach spaces.
\end{prop}

\begin{proof}
Suppose $({\bm \xi}_1, {\bm \xi}_2)\in {\mc U}_{12}$. Then there exists $g\in W^{2, p}_{\rm loc}({\mb C},K)$ such that 
\beqn
g \cdot \exp_{{\bf v}_2} {\bm \xi}_2 = \exp_{{\bf v}_1} {\bm \xi}_1.
\eeqn
Since the norm of $\hat{\mc B}_{{\bf v}_2}$, the exponential map, and the Coulomb gauge condition are all gauge invariant, one can assume that $ g= {\rm id}$. Then 
\beq\label{eqn320}
a_2 + \alpha_2 = a_1 + \alpha_1,
\eeq
\beq\label{eqn321}
\exp_{u_2} \xi_2 = \exp_{u_1} \xi_1
\eeq

To prove the openness of $\varphi_i({\mc U}_{12})$, we would like to show that if we perturb $(\xi_2, \alpha_2)$ a little, then one can still solve the above equation for $(\xi_1, \alpha_1) \in {\mc B}_{{\bf v}_1}$. Take $\wt{\bm \xi}_2' \in {\mc B}_{{\bf v}_2}$ such that the norm of ${\bm \xi}_2':= \wt{\bm \xi}_2 - {\bm \xi}_2 = (\xi_2',\alpha_2')$ is sufficiently small. Define $\alpha_1' = \alpha_2'$. On the other hand, for each $z \in {\mb C}$, there exists a unique $\xi_1'$ along $u_1$ such that 
\beq\label{eqn322}
\exp_{u_2 (z) }\big( \xi_2 (z) + \xi_2'(z) \big) = \exp_{u_1(z)} \big( \xi_1(z) + \xi_1'(z) \big).
\eeq
Denote ${\bm \xi}_1' = (\xi_1', \alpha_1')$. Although we do not know its regularity yet, pointwise one has
\beqn
\exp_{{\bf v}_1} ( {\bm \xi}_1 + {\bm \xi}_1') = \exp_{{\bf v}_2} ( {\bm \xi}_2 + {\bm \xi}_2').
\eeqn

\begin{claim}
There exists $\epsilon_2'>0$ such that for all ${\bm \xi}_2'$ with $\| {\bm \xi}_2'\|_{p, \delta} < \epsilon_2'$, ${\bm \xi}_1'\in \hat{\mc B}_{{\bf v}_1}$ and the map ${\bm \xi}_2' \mapsto {\bm \xi}_1'$ is a smooth map from ${\mc B}_{{\bf v}_2}^{\epsilon_2'}$ to ${\mc B}_{{\bf v}_1}$.
\end{claim}

If this claim is proved, for $\| {\bm \xi}_2'\|$ sufficiently small, take the corresponding ${\bm \xi}_1' = (\xi_1', \alpha_1')$. Then ${\bm \xi}_1 + {\bm \xi}_1'$ satisfies the hypothesis of Proposition \ref{prop38}, i.e., 
\beqn
\| \alpha_1 + \alpha_1'\|_{W^{1, p, \delta}} + \| \xi_1 + \xi_1'\|_{L^\infty} + \| d_{a_1}^{\,*}(\alpha_1 + \alpha_1') + d\mu(u_1) \cdot (J \xi_1 + J \xi_1') \|_{L^{p, \delta}} \leq \epsilon_{\rm coul}({\bf v}_1).
\eeqn
Then there exists a unique $s_1\in W^{2, p,\delta}({\mb C}, {\mf k})$ such that 
\beqn
\| s_1 \|_{W^{1, p, \delta}}\leq \delta_{\rm coul}({\bf v}_1),\  e^{s_1} \cdot \exp_{{\bf v}_1} ({\bm \xi}_1 + {\bm \xi}_1') \in {\mc B}_{{\bf v}_1}.
\eeqn
We can write $e^{s_1} \cdot \exp_{{\bf v}_1} ({\bm \xi}_1 + {\bm \xi}_1') = \exp_{{\bf v}_1} \wt{\bm \xi}_1$. Therefore $(\wt{\bm \xi}_1, \wt{\bm \xi}_2)\in {\mc U}_{12}$ and this proves the openness of $\varphi_2({\mc U}_{12})$. $\varphi_1({\mc U}_{12})$ is open in ${\mc B}_{{\bf v}_1}^{\epsilon_1}$ for the same reason. Moreover, since $s_1$ is constructed using the implicit function theorem, which depends smoothly on ${\bm \xi}_1'$. Therefore the map $\varphi_1\circ \varphi_2^{-1}: \wt{\bm \xi}_2\mapsto \wt{\bm \xi}_1$ is smooth. 

Now we prove the above claim. Since the norm of $\hat{\mc B}_{{\bf v}_i}$ can control the $L^\infty$ norm, one has $\|a_1 - a_2\|_{L^\infty} = \| \alpha_1 - \alpha_2\|_{L^\infty} <\infty$. Therefore 
\beqn
\|\alpha_1'\|_{L^{p, \delta}} + \| \nabla^{a_1} \alpha_1'\|_{L^{p, \delta}} \leq \| \nabla^{a_2} \alpha_2'\|_{L^{p, \delta}} + ( 1 + \| a_2 - a_1\|_{L^\infty}) \| \alpha_2' \|_{L^{p, \delta}}.
\eeqn
To control the norm of $\xi_1'$, notice that for each $z \in {\mb C}$, \eqref{eqn322} defines a map $\xi_2'(z) \mapsto \xi_1'(z)$ from $T_{u_2(z)}^{\ \delta} X$ to $T_{u_1(z)}^{\ \delta} X$ for sufficiently small $\delta$. This map depends smoothly on $u_1(z), \xi_1(z), u_2(z), \xi_2(z)$ and is the identity if $\xi_1(z) = \xi_2(z) = 0$. Therefore, there exists $c_1 > 0$ such that when $|\xi_2'(z)|$ is small enough, 
\beqn
|\xi_1'(z)| \leq c_1 |\xi_2'(z)|;
\eeqn
\beqn
|d\mu(u_1(z)) \cdot \xi_1'| + |d\mu(u_1(z)) \cdot J \xi_1'| \leq c_1 \Big[ |d\mu(u_2(z) ) \cdot \xi_2'| + |d\mu(u_2(z)) \cdot J \xi_2' |  + (|\xi_1 | + |\xi_2|) |\xi_2'| \Big].
\eeqn
Therefore, for appropriate value of $c_2>0$ (which is independent of $\xi_2'$)
\begin{multline}\label{eqn323}
\| \xi_1' \|_{L^\infty} + \| \d\mu(u_1) \cdot J \xi_1' \|_{L^{p, \delta}} + \| \d\mu(u_1) \cdot \xi_1' \|_{L^{p, \delta}} \\
\leq c_1 \Big[ \| \xi_2' \|_{L^\infty} + \| \d\mu(u_2) \cdot J \xi_2' \|_{L^{p, \delta}} + \| \d\mu(u_2) \cdot \xi_2' \|_{L^{p, \delta}} + \| \xi_2'\|_{L^\infty} \|\xi_1 \|_{L^{p, \delta}} + \| \xi_2'\|_{L^\infty} \| \xi_2\|_{L^{p, \delta}} \Big]\\
\leq c_2 \Big[ \| \xi_2' \|_{L^\infty} + \| \d\mu(u_2) \cdot J \xi_2' \|_{L^{p, \delta}} + \| \d\mu(u_2) \cdot \xi_2' \|_{L^{p, \delta}} \Big].
\end{multline}

We also need to estimate the norm of the derivative of $\xi_1'$. Represent the derivative of the exponential map as 
\beqn
d \exp_x \xi = E_1(x, \xi) dx + E_2(x, \xi) \nabla \xi,\ E_1(x, \xi), E_2(x, \xi)\in {\rm Hom}( T_x X, T_{\exp_x \xi} X).
\eeqn
Then differentiating both sides of \eqref{eqn322} and \eqref{eqn321} with respect to $z$, we obtain
\begin{multline}\label{eqn324}
E_1 (u_2, \xi_2 + \xi_2' ) d u_2 + E_2 ( u_2, \xi_2 + \xi_2' ) ( \nabla \xi_2 + \nabla \xi_2' ) \\
 = E_1 (u_1, \xi_1 + \xi_1' ) \d u_1 + E_2( u_1, \xi_1 + \xi_1' ) ( \nabla \xi_1 + \nabla \xi_1' ),
\end{multline}
\beq\label{eqn325}
E_1 ( u_2, \xi_2) d u_2 + E_2 ( u_2, \xi_2 ) \nabla \xi_2  = E_1 ( u_1, \xi_1 ) d u_1 + E_2 ( u_1, \xi_1) \nabla \xi_1.
\eeq
Let $P$ be the parallel transport from the tangent space at $\exp_{u_1} \xi_1 = \exp_{u_2} \xi_2$ to the tangent space at $\exp_{u_1} (\xi_1 + \xi_1') = \exp_{u_2} (\xi_2 + \xi_2')$. Then by \eqref{eqn324} and \eqref{eqn325}, we have
\begin{multline*}
E_2(u_1, \xi_1 + \xi_1') \nabla \xi_1' \\
=  E_1 (u_2, \xi_2 + \xi_2') d u_2 - E_1(u_1, \xi_1 + \xi_1') d u_1 + E_2(u_2, \xi_2 + \xi_2') ( \nabla \xi_2 + \nabla \xi_2') - E_2 (u_1, \xi_1 + \xi_1') \nabla \xi_1\\
= \Big( E_1(u_2, \xi_2 + \xi_2') d u_2 - P E_1(u_2, \xi_2) d u_2 \Big) - \Big( E_1(u_1, \xi_1 + \xi_1') d u_1 -  P E_1 (u_1, \xi_1) d u_1 \Big) \\
 + \Big( E_2(u_2, \xi_2 + \xi_2') \nabla \xi_2  - P E_2( u_2, \xi_2) \nabla \xi_2 \Big) -  \Big( E_2 (u_1, \xi_1 + \xi_1') \nabla \xi_1 - P E_2( u_1,  \xi_1) \nabla \xi_1 \Big) \\
 + E_2(u_2, \xi_2 + \xi_2') \nabla \xi_2'.
\end{multline*}
Since $E_1, E_2$ are differentiable, there exists $c_3 >0$ such that 
\beq\label{eqn326}
\| \nabla \xi_1' \|_{L^p(B_R)} \leq c_3 \sum_{j=1}^2 \| \xi_j' \|_{L^\infty} \Big[ \| d u_j \|_{L^p(B_R)} + \| \nabla \xi_j \|_{L^p(B_R)} \Big] + c_3 \| \nabla \xi_2' \|_{L^p(B_R)}.
\eeq
To estimate the $L^{p, \delta}$ norm of the derivative of $\xi_1'$ over $C_R$, one can twist the maps by $e^{-\lambda \theta}$ and apply the same argument. Thus there exists $c_3>0$ such that 
\beq\label{eqn327}
\| \nabla^{a_1} \xi_1'  \|_{L^{p, \delta}(C_R)} \leq c_3 \Big[ \| \xi_1' \|_{L^\infty} + \| \xi_2' \|_{L^\infty} +  \| \nabla^{a_2} \xi_2' \|_{L^{p, \delta}(C_R)} \Big].
\eeq

\eqref{eqn323}, \eqref{eqn326} and \eqref{eqn327} together show that if $\| {\bm \xi}_2' \|_{p, \delta}$ is sufficiently small, ${\bm \xi}_1' \in \hat{\mc B}_{{\bf v}_1}$ and $\|{\bm \xi}_1' \|_{p, \delta} \leq c_4 \| {\bm \xi}_1'\|_{p, \delta}$ for some $c_4>0$. So we have proved the first part of the claim and the continuity of the map ${\bm \xi}_1' \mapsto {\bm \xi}_2'$. 

As for the smoothness, the nontrivial part is about the map $\varphi: \xi_2'\mapsto \xi_1'$, which is defined by \eqref{eqn322}. Since everything happens in a compact subset of $X$, all constants from the geometry of $X$ can be uniformly bounded. Moreover, \eqref{eqn322} is a pointwise relation. So for any unit disk $B_1(z) \subset {\mb C}$, if we denote $\varphi_z: W^{1, p}(B_1(z), u_2^* TX) \to W^{1, p}(B_1(z), u_1^* TX)$ by restricting $\varphi$ to $B_1(z)$, then $\varphi_z$ is smooth and its derivatives can be bounded by numbers that are independent of $z$. Moreover, take $R \geq 1$. Restricting \eqref{eqn322} to $C_R$ is equivalent to 
\beqn
\exp_{\check{u}_2} (\check\xi_2 + \check\xi_2') = \exp_{\check{u}_1} (\check\xi_1 + \check\xi_1'),
\eeqn
where objects with a $\check{}$ are obtained by twisting by $e^{-\lambda \theta}$. We also denote the map $\check\xi_2' \mapsto \check\xi_1'$ by $\check\varphi_z$. Its derivatives can also be bounded by numbers that are independent of $z$. Take a countable cover $\{ B_1(z_k)\}_{k=1}^\infty$ of $C_R$. Since we have the following equivalences of norms
\beqn
\| \xi \|_{L^{p, \delta}} + \| \nabla^{a_i} \xi \|_{L^{p, \delta}} \approx \| \xi\|_{W^{1, p}(B_R)} +  \| \check\xi \|_{W^{1, p, \delta}(C_R)} \approx  \| \xi\|_{W^{1, p}(B_R)} + \sum_{k=1}^\infty |z_k|^\delta \| \check\xi\|_{W^{1, p}(B_1(z_k))} 
\eeqn
(for $i=1, 2$), it leads to the smoothness of $\varphi$ with respect to the norms of ${\mc B}_{{\bf v}_1}$ and ${\mc B}_{{\bf v}_2}$. \end{proof}

The transition between ${\mc B}_{{\bf v}_1}$ and ${\mc B}_{{\bf v}_2}$ immediately gives a transition function between the finite dimensional charts $(U_1, \hat\phi_1): = (U_{{\bf v}_1}^{\epsilon_1}, \hat\phi_{{\bf v}_1})$ and $(U_2, \hat\phi_2):= ( U_{{\bf v}_2}^{\epsilon_2}, \hat\phi_{{\bf v}_2})$. Indeed, if $\hat\phi_1(U_1) \cap \hat\phi_2(U_2) \neq \emptyset$, then the restriction of $\varphi_1 \circ \varphi_2^{-1}$ to $U_2\subset {\mc B}_{{\bf v}_2}^{\epsilon_2}$ is smooth and its image is in $U_1$. Because $\varphi_1\circ \varphi_2^{-1}$ is a diffeomorphism and $U_1$, $U_2$ have the same dimension, the restriction of $\varphi_1 \circ \varphi_2^{-1}$, which coincides with $\hat\phi_1^{-1} \circ \hat\phi_2$, is a diffeomorphism between the charts. Therefore, over the regular part $M^K({\mb C}, X; \beta)^{\rm reg}$, we have constructed a collection of charts ${\ms C}^\infty$ and a collection of transition functions. The cocycle condition is immediate. Therefore, it gives a smooth atlas on $M^K ({\mb C}, X; \beta)^{\rm reg}$ which induces a structure of smooth manifold. Therefore we finish proving the main theorem.

\subsection{Proof of Lemma \ref{lemma310}}

The proof is by elliptic regularity of the operator $d \oplus d^{\, *}$ on $\alpha:= a- b$. Firstly we have
\begin{multline}\label{eqn328}
\| d^{\,*} \alpha \|_{L^p(\Omega)} = \| \dad \alpha - * [ a \wedge * \alpha] \|_{L^p(\Omega)} \leq \| \dad \alpha \|_{L^p(\Omega)} + \| a \|_{L^\infty(\Omega)} \| \alpha \|_{L^p(\Omega)} \\
\leq \Big[ 1+ \| a \|_{L^\infty(\Omega)} \Big]  \Big[ \| \dad \alpha \|_{L^p(\Omega)} + \| \alpha \|_{L^p(\Omega)} \Big].
\end{multline}
We also need to bound the $L^p$ norm of $\d \alpha$ over a domain larger than $\Omega'$. We have
\beqn
F_a = F_b + d \alpha + [a \wedge \alpha] + [\alpha \wedge \alpha].
\eeqn
Then for $1\leq q\leq p$ and any compact domain $B \subset \Omega$, we have
\beq\label{eqn329}
\| d \alpha \|_{L^q(B)} \leq \Big[ 1 + \| a \|_{L^\infty(B)} \Big] \Big[  \| F_a - F_b \|_{L^q (B)} + \| \alpha \|_{L^q(B)} + \| \alpha \|_{L^{2q}(B)}^2 \Big].
\eeq
Enlarge $\Omega'$ a little to $\Omega''$, which is still contained in the interior of $\Omega$. We would like to prove that, there is a constant $c_1(M)>0$, also depending on $\Omega, \Omega''$ and $p$, such that 
\beq\label{eqn330}
\| \alpha\|_{L^\infty(\Omega'')} \leq c_1(M),
\eeq
which, together with \eqref{eqn328} and the $q = p$, $B = \Omega''$ case of \eqref{eqn329}, will imply 
\beqn
\| d \alpha \|_{L^p(\Omega'')}  + \| \dad \alpha \|_{L^p(\Omega'')} \leq c_2(M) \Big[ \| F_a - F_b \|_{L^p(\Omega)} +  \| \dad \alpha \|_{L^p(\Omega)} + \| \alpha \|_{L^p(\Omega)} \Big]
\eeqn
for certain $c_2(M)>0$. Then elliptic regularity for $d \oplus d^{\,*}$ implies \eqref{eqn39}.

By Sobolev embedding, to prove \eqref{eqn330}, it suffices to prove 
\beq\label{eqn331}
\| \alpha \|_{W^{1, q}(\Omega'')} \leq c_3(M)
\eeq
for some $q > 2$ and $c_3(M)>0$. Depending on the value of $p$, we argue in the following three cases.
\begin{enumerate}
\item $2<p<4$. Set $q_0:=p$. There exists $m \geq 1$ and a sequence $2<q_0<\dots<q_{m-1}\leq 4$, $q_m>4$ such that
\beq\label{eqn332}
q_0 = p,\quad q_j <\frac {2q_{j - 1}} {4-q_{j -1}} \quad j = 1, \ldots, m.
\eeq
We also fix a sequence of domains 
\begin{align*}
&\ \Om=\Om_0 \supset \dots \supset \Om_m  \supset \Om'',&\  \on{int}(\Om_k) \supset \Om_{k+1}.
\end{align*}
We prove inductively that for all $j = 0, 1, \ldots, m$, there exist constants $C_j(M)>0$ such that 
\beq\label{eqn333}
\| \alpha \|_{L^{q_j}(\Omega_j)} \leq C_j(M).
\eeq
The $j = 0$ case is given by the hypothesis. Assuming for $j<m$, there is $C_j(M)>0$ such that \eqref{eqn333} holds. Then the $q = q_j/2$, $B = \Omega_j$ case of \eqref{eqn329} implies that $d \alpha |_{\Omega_j} \in L^{q_j /2}(\Omega_j)$ with 
\beqn
\| d \alpha\|_{L^{q_j/2}(\Omega_j)} \leq C_j'(M).
\eeqn
Since $q_j/2 < p$, applying elliptic regularity with the estimates \eqref{eqn328}, the $q = q_j/2$, $B = \Omega_{j+1}$ case of \eqref{eqn329}, and the $L^p$ bound on $\alpha$, we get 
\beqn
\| \alpha \|_{W^{1, q_j/2}(\Omega_{j+1})} \leq C_j''(M).
\eeqn
By Sobolev embedding $W^{1,q_j /2} \hookrightarrow L^{q_{j+1}}$ (which is valid by \eqref{eqn332}), we obtain \eqref{eqn333} with $j$ replaced by $j+1$. By repeating this procedure, we end up with \eqref{eqn333} with $j =m$. Using the elliptic regularity again, we obtain \eqref{eqn331}.

\item If $p=4$, then modify \eqref{eqn332} by choosing $q_0$ slightly smaller than $4$ with $q_1 > 4$. One step of the above induction can prove \eqref{eqn333} for $j = m = 1$ and hence \eqref{eqn331}. 

\item  if $p>4$, then take $m = 0$. \eqref{eqn331} is also obtained.
\end{enumerate}
Therefore \eqref{eqn331} holds for certain $q>2$. This finishes the proof of Lemma \ref{lemma310}.

\section{Proof of Proposition \ref{prop32}}\label{section4}

In this section, we prove that the augmented linearized operator $\hat{\mc D}_{\bf v}: \hat{\mc B}_{\bf v} \to \hat{\mc E}_{\bf v}$ at ${\bf v}\in \wt{M}^K({\mb C}, X)$ is a Fredholm operator of the expected index. This result is proved by Ziltener as \cite[Theorem 4]{Ziltener_book}, but we provide an alternate presentation. Here we remark again that the condition ${\rm dim} \bar{X} > 0$ assumed by Ziltener is unnecessary.

\subsection{Bundle completions to ${\mb P}^1$}

We fix $p>2$ and $R>1$. Let $E\to {\mb C}$ be a rank $n$ Hermitian vector bundle of class $W^{1, p}_{\rm loc}$ (that means, it consists of a collection of local charts with $U(n)$-valued transition functions of regularity $W^{1, p}_{\rm loc}$). A {\bf completion} of $E$ consists of $R > 0$ and a trivialization
\beqn
\psi_R: E|_{C_R} \to C_R \times {\mb C}^n
\eeqn
of regularity $W^{1, p}_{\rm loc}$. The pair $(E, \psi_R)$ is called a {\bf completed bundle}, which induces a Hermitian vector bundle $\ov{E}(\psi_R) \to {\mb P}^1$, by adding a new chart over ${\mb P}^1 \bs B_R$ together with the gluing map $\psi_R$. The degree of $\ov{E}(\psi_R)$ is called the degree of the completed bundle, denoted by ${\rm deg} (E, \psi_R) \in {\mb Z}$. 

Given $\delta \in (1- \frac{2}{p}, 1)$, define
\beqn
{\mc A}^{1, p, \delta}(E, \psi_R): = \Big\{ B \in {\mc A}^{1, p}_{\rm loc}(E) \ |\  \psi_R \circ (B|_{C_R}) \circ \psi_R^{-1} - d \in W^{1, p, \delta} \big( C_R, {\mf u}(n) \big) \Big\}.
\eeqn
For any $B \in {\mc A}^{1,p, \delta} (E, \psi_R)$, define the norm
\beq\label{eqn41}
\| \sig\|_{p, \delta}^B:= \|\sig \|_{L^\infty} + \| \nabla^{B} \sig \|_{L^{p,\delta}}.
\eeq
The Sobolev completion of sections under the above norm is denoted by $L_+^{1, p, \delta}(E, \psi_R)^B$. 

It is routine to check the following lemma.
\begin{lemma}\label{lemma41} Let $(E, \psi_R)$ be a completed bundle.
\begin{enumerate}
\item \label{lemma41b} Given $B\in {\mc A}^{1, p, \delta}(E, \psi_R)$, the covariant derivative $\nabla^{B}$ extends to a bounded map between Sobolev completions 
\beqn
\nabla^B: L_+^{1, p, \delta}(E, \psi_R)^B \to L^{p,\delta}({\mb C}, \Lambda^1 \otimes E).
\eeqn

\item \label{lemma41a} If $A, B \in {\mc A}^{1, p, \delta} (E, \psi_R)$, then $\| \cdot \|_{p, \delta}^A$ and $\| \cdot \|_{p, \delta}^{B}$ are equivalent.

\item \label{lemma41c} If $\delta_+ \in (\delta, 1)$ and $A, B \in {\mc A}^{1, p, \delta_+} (E, \psi_R)$, then 
\beqn
\nabla^B - \nabla^A: L_+^{1, p, \delta}(E, \psi_R)^B \to L^{p,\delta}({\mb C}, \Lambda^1 \otimes E)
\eeqn
is a compact operator.
\end{enumerate}
\end{lemma}

Therefore we can remove the symbol $B$ from the notations $\| \cdot \|_{p, \delta}^B$ and $L_+^{1, p, \delta}({\mb C}, E)^B$. 
	
\begin{lemma}\label{lemma42}
Given $p>2$, $\delta \in (1- \frac{2}{p}, 1)$ and $\delta_+\in (\delta, 1)$. Suppose $(E, \psi_R)$ is a completed line bundle and $B \in {\mc A}^{1, p, \delta_+}(E, \psi_R)$, then 
\beqn
\ov\partial_B: L_+^{1, p, \delta}(E, \psi_R) \to L^{p, \delta}({\mb C}, \Lambda^{0,1}\otimes E)
\eeqn
is Fredholm with real index ${\rm ind} (\ov\partial_B) = 2 {\rm deg} (E, \psi_R) + 2$.
\end{lemma}

\begin{proof}
By Lemma \ref{lemma41}, it is enough to prove the result for a specific $B \in {\mc A}^{1, p, \delta_+}(E, \psi_R)$. So, we assume that $B$ is flat over ${\mb P}^1 \bs B_R$, with a flat nonvanishing section $s_\infty^E$ of unit length. Denote $d = {\rm deg} (E, \psi_R)$. Then the section $z^{-d} s_\infty^E \in W^{1, p}_{\rm loc}(C_R, E)$ can be extended to a nowhere vanishing $W^{1, p}_{\rm loc}$ section $s_0^E$ over ${\mb C}$. Regard $E$ be defined as the collection $\{ s_0^E|_{B_{2R}}, s_\infty^E|_{{\mb P}^1 \bs B_R}\}$ with the transition function $z^d$. 

We construct a smooth Hermitian line bundle $L$, consisting of a chart over ${\mb P}^1\bs B_R$ generated by a unit section $s_\infty^L$, another chart over $B_{2R}$ generated by a unit section $s_0^L$ such that on the overlap $s_0^L = z^{-d} s_\infty^L$. Then $L$ is a holomorphic Hermitian line bundle with an operator 
\beq\label{eqn42}
\ov\partial_L: L_+^{1, p, \delta}({\mb C}, L) \to L^{p, \delta}({\mb C}, L).
\eeq

The formulae $\mu_0 (s_0^E) = s_0^L$, $\mu_\infty(s_\infty^E) = s_\infty^L$ define a bundle isomorphism $\mu: E \to L$ of regularity $W^{1, p}_{\rm loc}$, which induces bounded, invertible linear maps
\begin{align*}
& \mu_0: L^{1, p, \delta}_+ ({\mb C}, E) \to L^{1, p, \delta}_+ ({\mb C}, L), & \mu_1: L^{p, \delta}({\mb C}, E) \to L^{p, \delta}({\mb C}, L).
\end{align*}
It is easy to see that $\mu_0 \circ \ov\partial_B \circ \mu_1^{-1} - \ov\partial_L$ has compact support, hence a compact operator. Therefore, it suffices to consider the operator \eqref{eqn42}. It is a standard result that when $\delta \in ( 1- \frac{2}{p}, 1)$, the operator $\ov\partial_L: L^{1, p, \delta}_+({\mb C}, L) \to L^{p, \delta}({\mb C}, L)$ is Fredholm with real index $2 + 2d$.
\end{proof}

The following lemma related to Hermitian vector bundles on ${\mb P}^1$ is also used in the proof of the Fredholm result.

\begin{lemma}\label{lemma43}
Given a completed Hermitian vector bundle $(E, \psi_R) \to {\mb P}^1$, an orthonormal frame of $\{s_i:{\mb P}^1 \bs B_R \to \ov{E}_R(\psi_R): 1 \leq i \leq n \}$ over the chart at infinity, and integers $d_1, \dots,d_n$ such that $\sum_i d_i=d:= {\rm deg} (E)$, there exist completed line bundles
\beqn
(\ell_k, \psi_{R, k}) \to {\mb P}^1
\eeqn
such that $(E, \psi_R) = \bigoplus_{k=1}^n ( \ell_k, \psi_{R, k})$, ${\rm deg} \ell_k = d_k$ and $\ell_k$ is spanned by $s_k$ over ${\mb P}^1 \bs B_R$. 
\end{lemma}

\begin{proof} Recall that ${\rm deg} (E)$ is determined by the homotopy class in $\pi_1(U(n))$ of the transition function between frames of $E$ around $0$ and around $\infty$. Since ${\rm deg} (E) = d$, this homotopy class can be represented by the loop of matrices 
\beqn
\theta \mapsto {\rm diag} \big( e^{-2 \pi{\bm i} d_1 \theta}, \ldots, e^{ - 2\pi {\bm i} d_n \theta} \big)
\eeqn
This means the sections $s_k':= e^{- 2\pi {\bm i} d_k \theta} s_k$ of $E|_{C_R}$ can be extended to a frame of $E$ over ${\mb C}$ (of regularity $W^{1, p}_{\rm loc}$). This determines sub line bundles $\ell_k \subset E|_{\mb C}$ with completions that have the degrees $d_k$.
\end{proof}

\subsection{Proof of Fredholm result}

With above preparations, now we can prove Proposition \ref{prop32}. We fix some notation
\beqn 
m:=\dim K, \quad n:=\frac{1}{2} \dim X, \quad \bar{n}:=n - m= \frac{1}{2} \dim \bar{X}, \quad d:= \langle c_{\,1}^K(TX), \beta\rangle.
\eeqn
Firstly, the operator $\hat{\mc D}_{\bf v}$ is equivariant with respect to gauge transformations. That is, if ${\bf v}, {\bf v}'\in \wt{M}^K({\mb C}, X)^{1, p}_{\rm loc}$ are related by a gauge transformation $g\in {\mc K}^{2, p}_{\rm loc}$, then the following diagram commutes.
\beqn
\xymatrix{ \hat{\mc B}_{\bf v} \ar[r]^{\hat{\mc D}_{\bf v}} \ar[d]^{g} & {\mc E}_{\bf v} \ar[d]^{g} \\
           \hat{\mc B}_{{\bf v}'} \ar[r]^{\hat{\mc D}_{{\bf v}'}} & {\mc E}_{{\bf v}'}}.
\eeqn
Therefore, by Lemma \ref{lemma61} (which is independent of the proposition we are proving), one can make the following assumptions. ${\bf v} = (u, a)$ is of regularity $W^{1, p}_{\rm loc}$ and there exist $\lambda \in \Lambda_K$ and $x \in \mu^{-1}(0)$ such that 
\beqn
a|_{C_R} + \lambda d\theta \in W^{1,p, \delta_+}(C_R, \Lambda^1 \otimes {\mf k}),\ \lim_{z \to \infty} e^{-\lambda \theta} u(z) = x.
\eeqn
where $\delta_+$ is slightly bigger than $\delta$. Moreover, if we write $e^{-\lambda \theta} u (z) = \exp_x \xi_u (z)$, then $\nabla \xi_u \in W^{1, p, \delta_+}(C_R, T_x X)$. 

One constructs a completion of $u^* TX$ as follows. There is a neighborhood $U_x$ of $x$ and a smooth trivialization $\psi_x: (TU_x, J) \simeq (U_x \times T_{x} X, J_\infty )$ which is the identity at $x$ and which maps $G_X$ onto $U_\infty \times G_{X, x}$. On the other hand, for large $R$, $u(C_R) \subset U_\infty$. The trivialization $\psi_x$ then induces the chart 
\beqn
\psi_R = \psi_x (u): u^* TX|_{C_R} \to C_R \times T_x X.
\eeqn
Since $u$ is of class $W^{1, p}_{\rm loc}$, this gives a completion of $u^* TX$. By the asymptotic property of the gauge we chose, it is easy to see that under the trivialization $\psi_R$, the connection matrix of $\nabla^a$ on $u^* TX|_{C_R}$ is of class $W^{1, p, \delta_+}$.

The following remark shows a crucial fact related to the asymptotic behavior of the operator $\hat{\mc D}_{\bf v}$ at infinity.

\begin{remark}\label{remark44}{\rm(An eigenvalue decomposition of ${\mf k}$)} 
For any point $x \in \mu^{-1}(0) \subset X$, denote the infinitesimal action of ${\mf k}$ by
\begin{align}\label{eqn43}
&\ L_x: {\mf k} \to K_{X, x},   & \zeta \mapsto {\mc X}_\zeta (x).
\end{align}
Its adjoint with respect to the metric $\omega_X (\cdot,J\cdot)$ at $T_x X$ is 
\beqn
L_x^* (\xi) = d \mu(x) \cdot J \xi.
\eeqn
Therefore $L_x^*L_x$ is symmetric and positive definite, having eigenvalues $\lambda_1,\dots,\lambda_m > 0$. 
\end{remark}

We choose eigenvectors $\{ e_1,\ldots, e_m \} \subset {\mf k}$ corresponding to the eigenvalues of $L_x^* L_x$, such that $L_x(e_1), \ldots, L_x(e_m)$ form an orthonormal basis of $K_{X, x}$. They become a Hermitian orthonormal basis of the complexification $G_{X, x}$. We extend them into a Hermitian orthonormal basis $L_x(e_1), \ldots, L_x(e_m), v_{m+1}, \ldots, v_n$ of $T_x X$. Via $\psi_R$, these vectors extends an orthonormal frame $s_1, \ldots, s_n$ of $\ov{u^* TX}(\psi_R)$ over ${\mb P}^1 \bs B_R$. 

Now we separate the argument into two cases: $\bar{n}>0$ and $\bar{n} = 0$. When $\bar{n}>0$, by Lemma \ref{lemma43}, there are completed line bundles $(\ell_k, \psi_{R, k}) \to {\mb P}^1$ for $1 \leq k \leq n$ and an isomorphism  
\beq\label{eqn44}
\rho_1: (u^* TX, \psi_R) \simeq \bigoplus_{k=1}^n (\ell_k, \psi_{R, k}),
\eeq
\beqn
\ell_k|_{{\mb P}^1 \bs B_R} = {\rm span}_{\mb C} \big\{ \rho_1(s_k) \big\},\ {\rm deg} (\ell_k, \psi_{R, k}) = \left\{ \begin{array}{cc} 0, &\ k < n;\\ d, &\ k=n \end{array}\right.
\eeqn
We view $\ell_k \to {\mb C}$ as a trivial line bundle for $k< n$. We choose a connection $B \in {\mc A}^{1, p, \delta_+} ( u^* TX, \psi_R)$ satisfying the following conditions.
\begin{enumerate}
\item Over $C_R$, $\psi_R \circ \nabla^B \circ \psi_R^{-1} = d$, where the latter is the trivial connection on $C_R \times T_x X$. 

\item $B$ is diagonal with respect to the splitting \eqref{eqn44} and its restriction to $\ell_k$ for $k < n$ is the trivial connection. Denote the restriction of $\ov\partial_B$ onto $\ell_k$ by $\ov\partial_{B, \ell_k}$.
\end{enumerate}

For $k= 1, \ldots, m$, let $\ell_k'\to {\mb C}$ be a copy of the trivial complex line bundle generated by $e_k$. Then one has the isomorphism $L_X: \bigoplus_{k=1}^m \ell_k' \simeq \bigoplus_{k=1}^m \ell_k$. Since $\ell_k$ and $\ell_k'$ are all trivial, we can extend $L_X$ to an isomorphism over ${\mb C}$, which is still denoted by $
L_X$. 

We would like to identify the bundle $\Lambda^1 \otimes {\mf k}$, i.e., the bundle of deformations of the connection form, with $\bigoplus_{k=1}^m \ell_k' \simeq {\mf g}$. Define over $C_R$
\begin{align*}
&\ \rho_2: \Lambda^1 \otimes {\mf k} \simeq \bigoplus_{k=1}^m \ell_k',\ & \rho_2(\eta ds + \zeta dt) = {\rm Ad}_{e^{-\lambda \theta}} (\eta + {\bm i} \zeta) = e^{-\lambda \theta} ( \eta+ {\bm i} \zeta ) e^{\lambda \theta}.
\end{align*}
Since the adjoint action on ${\mf g}$ is real, the determinant of ${\rm Ad}_{e^{-\lambda \theta}}$ is always $\pm 1$. So as a loop in ${\rm Aut}({\mf g})$, ${\rm Ad}_{e^{-\lambda\theta}}$ is homotopically trivial, therefore it can be extended to an isomorphism $\rho_2: \Lambda^1 \otimes {\mf k} \to \bigoplus_{k=1}^m \ell_k'$ over ${\mb C}$. We also identify the last two ${\mf k}$ factors of $\hat{\mc E}_{\bf v}$ with $\bigoplus_{k=1}^m \ell_k'$ by an isomorphism $\rho_3: {\mf k} \oplus {\mf k} \to \bigoplus_{k=1}^m \ell_k' \simeq {\mf g}$ whose restriction to $C_R$ is 
\beqn
\rho_3 (\kappa, \varsigma) = {\rm Ad}_{e^{-\lambda \theta}} (\eta + {\bm i} \zeta).
\eeqn
Then $\rho_1$, $\rho_2$, $\rho_3$ induce bundle isomorphisms 
\beqn
{\bm \rho} = (\rho_1, \rho_2): u^* TX	\oplus {\mf k} \oplus {\mf k} \simeq \bigoplus_{k=1}^{n} \ell_k \oplus \bigoplus_{k=1}^m \ell_k';
\eeqn
\beqn
{\bm \rho}' = (\rho_1, \rho_3): u^* TX  \oplus {\mf k} \oplus {\mf k} \simeq \bigoplus_{k=1}^n \ell_k \oplus \bigoplus_{k=1}^m \ell_k'.
\eeqn

\begin{lemma}
${\bm \rho}$ and ${\bm \rho}'$ induce equivalences of Banach spaces
\beqn
{\bm \rho}: \hat{\mc B}_{\bf v} \sim \bigoplus_{k={m+1}}^n L^{1, p, \delta}_+ (\ell_k)  \oplus \bigoplus_{k=1}^m W^{1, p, \delta}(\ell_k \oplus \ell_k').
\eeqn
\beqn
{\bm \rho}: \hat{\mc E}_{{\bf v}} \sim \bigoplus_{k=m+1}^n L^{p, \delta} (\ell_k) \oplus \bigoplus_{k=1}^m L^{p, \delta}(\ell_k \oplus \ell_k').
\eeqn
Here the norm of $W^{1, p, \delta}(\ell_k \oplus \ell_k')$ is defined without ambiguity since $\ell_k$ and $\ell_k'$ are trivial.
\end{lemma}

Then consider the operator $\hat{\mc D}_B:  \hat{\mc B}_{\bf v} \to \hat{\mc E}_{\bf v}$ defined by
\beq\label{eqn45}
\hat{\mc D}_B:= ({\bm \rho}')^{-1} \circ \left( \begin{array}{ccc}  \displaystyle \bigoplus_{k= m +1}^n \ov\partial_{B, \ell_k} & 0 & 0 \\
                                                                          0 &  \displaystyle \bigoplus_{k=1}^m \ov\partial_{B, \ell_k} &  L_X \\
                                                                          0 & L_X^{-1} & \displaystyle \bigoplus_{k=1}^m \partial_{\ell_k'} 
																			  \end{array}  \right) \circ {\bm \rho}
\eeq

One can see that $\hat{\mc D}_B$ is a first order elliptic operator, whose restriction to $C_R$ reads
\beq\label{eqn46}
\hat{\mc D}_B \left( \begin{array}{c} \xi \\ \eta\\ \zeta \end{array} \right) = \left( \begin{array}{c}  \psi_R^{-1} \partial_{\ov{z}}( \psi_R(\xi))+  \psi_R^{-1} \big( {\mc X}_\eta(x) + J {\mc X}_\zeta(x) \big) \\  \partial_s \eta + [\phi_\lambda, \eta] + \partial_t \zeta + [\psi_\lambda, \zeta]  + d\mu(x) \cdot \psi_R( J \xi) \\  \partial_s \zeta  + [ \phi_\lambda, \zeta] - \partial_t \eta - [\psi_\lambda, \eta] + d\mu(x) \cdot \psi_R( \xi) 
\end{array} \right).
\eeq
Here $\phi_\lambda$ and $\psi_\lambda$ are defined by $-\lambda d\theta = \phi_\lambda ds + \psi_\lambda dt$. 

\begin{lemma}\label{lemma45}
$\hat{\mc D}_B - \hat{\mc D}_{\bf v}$ is a compact operator.
\end{lemma}

\begin{proof}
Comparing \eqref{eqn46} with \eqref{eqn25}, there are the following terms which contribute to the difference $\hat{\mc D}_B - \hat{\mc D}_{\bf v}$. We show that they are all compact operators.
\begin{enumerate}
\item The operator $\xi \mapsto \psi_R^{-1} (  \partial_{\ov{z}} (\psi_R(\xi))) - \nabla^a_s \xi - J \nabla^a_t \xi$ is compact by the property of the connection form $a$ in this trivialization and Lemma \ref{lemma41} Part \eqref{lemma41c}.

\item By Proposition \ref{prop214}, the operator $\xi \mapsto (\nabla_\xi J) (\partial_t u + {\mc X}_\psi(u))$ is compact since $\partial_t u + {\mc X}_\psi(u)$ converges to zero as $z \to \infty$.

\item By Proposition \ref{prop214}, the operator $(\eta, \zeta) \mapsto \psi_R^{-1} ({\mc X}_\eta(x) + J {\mc X}_\zeta(x)) - {\mc X}_\eta(u) - J {\mc X}_\zeta(u)$ and the operator 
\beqn
\xi \mapsto \left( \begin{array}{c}  d\mu(x) \cdot (\psi_R(J\xi)) - d\mu(u) \cdot J\xi \\
                                      d\mu(x) \cdot (\psi_R \xi) - d\mu(u) \cdot J \xi 
\end{array}  \right)
\eeqn
are compact because of the convergence $e^{-\lambda \theta} u(z) = x$. 
  
\item If $a = \phi ds + \psi dt$, then the last part of the difference of $\hat{\mc D}_B - \hat{\mc D}_{\bf v}$ on $C_R$ is 
\beqn
\left( \begin{array}{c} \eta \\ \zeta 
\end{array} \right) \mapsto \left( \begin{array}{c} [ \phi_\lambda - \phi, \eta] + [\psi_\lambda - \psi, \zeta] \\
                                                      {[} \phi_\lambda - \phi, \zeta] - [\psi_\lambda - \psi, \eta]
\end{array} \right).
\eeqn
Since $a + \lambda d\theta$ decays to zero as $z \to \infty$, the above is also a compact operator.
\end{enumerate}
The lemma is then proved.
\end{proof}

Therefore, the Fredholm theory of $\hat{\mc D}_{\bf v}$ is equivalent to that of $\hat{\mc D}_B$ where the latter is ``diagonalized''. Denote 
\beqn
\begin{split}
\ov\partial_K: W^{1, p, \delta} \Big( \bigoplus_{k=1}^m (\ell_k \oplus \ell_k')\Big) &\ \to L^{p, \delta}\Big( \bigoplus_{k=1}^m (\ell_k \oplus \ell_k') \Big)\\
                  \left( \begin{array}{cc} h \\ h' \end{array} \right)  &\ \mapsto  \left( \begin{array}{cc} \partial_{\ov{z}} & L_X \\ L_X^{-1} & \partial_z \end{array} \right) \left( \begin{array}{c} h \\ h'
										\end{array} \right).
\end{split}
\eeqn

\begin{lemma}\label{lemma46}{\rm (\cite[Proposition 97]{Ziltener_book})}
$\ov\partial_K$ is Fredholm with real index zero.
\end{lemma}

On the other hand, by the Riemann-Roch formula (Lemma \ref{lemma42}), we have
\beqn
{\rm ind} \Big( \ov\partial_{B, \ell_k}: L^{1, p, \delta}_+({\mb C}, \ell_k) \to L^{p, \delta}({\mb C}, \ell_k) \Big) = 2\ {\rm deg}\ \ell_k + 2,\ k = m +1, \ldots, n.
\eeqn
Together with Lemma \ref{lemma46}, we see $\hat{\mc D}_{\bf v}$ is Fredholm with index 
\beqn
{\rm ind} ( \hat{\mc D}_{\bf v} ) = \sum_{k=m +1}^n {\rm ind} (\ov\partial_{B, k})  = \sum_{k=m +1}^n 2\ {\rm deg}\ \ell_k + (2n - 2m) = {\rm dim}\bar{X} + 2 \langle c_{\, 1}^K(TX), \beta \rangle.
\eeqn
By Lemma \ref{lemma45}, this finishes the proof of Proposition \ref{prop32} in the case ${\rm dim}\bar{X}  > 0$. 

Finally, we consider the case when the symplectic quotient has zero dimension. We artificially increase the dimension of $X$, by defining $\hat X:=X \times {\mb C}$ and let $K$ act trivially on the second factor. This gives a new Hamiltonian $K$-manifold and we view the original vortex ${\bf v}$ as a vortex in $\hat{X}$ mapped to $X \times \{0\}$. Denote the vortex differential operator with respect $\hat X$ by $\hat{\mc D}_{\bf v}^+$, which has an extra block
\beqn
\frac{\partial}{\partial \ov{z}}: L_+^{1, p, \delta}({\mb C}, {\mb C}) \to L^{p, \delta}({\mb C}, {\mb C}).
\eeqn
Therefore $\hat{\mc D}_{\bf v}$ is Fredholm and
\beqn
{\rm ind} \big( \hat{\mc D}_{\bf v} \big) = {\rm ind} \big( \hat{\mc D}_{\bf v}^+ \big) - {\rm ind} \Big( \partial_{\ov{z}}: L_+^{1, p, \delta}({\mb C}, {\mb C}) \to L^{p, \delta}({\mb C}, {\mb C}) \Big).
\eeqn
By Lemma \ref{lemma42}, the last term above is equal to $2$. Therefore the formula for $\hat{\mc D}_{\bf v}$ is obtained by using the previous case for $\hat{X}$.

\section{Proof of Proposition \ref{prop36}}\label{section5}

Let $U\subset {\mb C}$ be a precompact open subset with smooth boundary. Given $p>2$. We say that $a \in W^{1, p}(U, \Lambda^1 \otimes {\mf k})$ is in {\bf Coulomb gauge} if
\begin{align*}
  & d^{\,*} a = 0,\ & * a|_{\partial U}=0.
\end{align*}
We say that $a \in L^p(U, \Lambda^1\otimes {\mf k})$ is in {\bf Coulomb gauge} if 
\beqn
\int_U \langle a, d \psi \rangle = 0, \quad \forall \psi \in C^\infty(U,{\mf k}). 
\eeqn

\begin{lemma}\label{lemma51}\cite[Theorem 8.1, Theorem 8.3]{Wehrheim_Uhlenbeck}
Let $k$ be $0$ or $1$. There exist constants $\delta_U > 0$ and $c_U >0$ such that for any $a \in W^{k, p}(U, \Lambda^1 \otimes {\mf k})$ satisfying $\| a \|_{W^{k,p}} \leq \delta_U$, there exists $g \in {\mc K}^{k+1, p}(U)$ such that $g\cdot a$ is in Coulomb gauge and
\beq\label{eqn51}
\| g \cdot a \|_{W^{k,p}} \leq c_U \| a \|_{W^{k,p}}.  
\eeq
\end{lemma}

The following results are used in the proof of Proposition \ref{prop36}.
\begin{lemma}\label{lemma52}\cite[Theorem 5.1]{Wehrheim_Uhlenbeck}
Let $p>2$ and $U$ be as above. Then there exists $c>0$ such that, for any $a \in W^{1, p}(U, \Lambda^1)$ satisfying $*a|_{\partial U}=0$,
\beq\label{eqn52} 
\| a \|_{W^{1,p}(U)} \leq c \Big( \| d a\|_{L^p(U)} +  \| d^{\, *} a \|_{L^p(U)} +  \| a \|_{L^p(U)} \Big).
\eeq
\end{lemma}

\begin{proof}[Proof of Proposition \ref{prop36}]
We remark that a special case of this proposition is when $a_i = a$ is independent of $i$. For the general case, all constants involved in the estimates (which come from Propositions \ref{prop26} and \ref{prop27}) can be chosen independent of $i$. To save notations, we carry out the proof for a single $a$.

The cylindrical coordinates on $C_R$ are defined via the exponential map
\begin{align*}
&\ \exp:[\log R,\infty) \times S^1 \to C_R, &\ (\tau,\theta) \mapsto e^{\tau + {\bm i} \theta}.
\end{align*}
We assume $a$ is in radial gauge, i.e., $\Psi^* a = \wt\psi(\tau, \theta) d \theta$. Let $k_0 \in W^{1,p}(S^1,K)$ be the map from Proposition \ref{prop27}. We view $k_0 : [\log R, \infty) \times S^1 \to K$ as a gauge transformation that is independent of the first coordinate. We fix a constant $ \frac{1}{2}  < \sig < 1$ and consider bands on the cylinder $S_n:=[n\sig,n\sig+1] \times S^1$ equipped with the cylindrical metric. In the following discussion, the constants denoted by $c$ are functions of $\sig$, but are independent of $n$. 

By Proposition \ref{prop27}, for any $0<\gamma<\frac 2 p$,
\beq\label{eqn53}
\| k_0^{-1} \cdot \wt\psi d \theta \|_{L^p(S_n)} \leq ce^{-\gamma  n \sigma }.
\eeq
By Lemma \ref{lemma51}, for any sufficiently large $n$, there exists a gauge transformation $g_n \in W^{1,p}(S_n, K)$ that transforms $k_0^{-1} \cdot \wt\psi d\theta$ to Coulomb gauge on $S_n$. By elliptic regularity for the vortex equation plus the Coulomb gauge condition, $( k_0 g_n)^{-1} \cdot \wt\psi d\theta$ is smooth. Because $\wt\psi$ is smooth, $k_0 g_n$ is also smooth. Denote the new connection form by
\beqn
\wt{a}_n:= ( k_0 g_n)^{-1} \cdot \wt\psi d\theta \in \Omega^1(S_n, {\mf k}).
\eeqn
\eqref{eqn51} and \eqref{eqn53} imply that 
\beq\label{eqn54}
\| \wt{a}_n \|_{L^p(S_n)} \leq c \| k_0\cdot  \wt\psi d \theta \|_{L^p(S_n)} \leq  ce^{-\gamma n\sigma}.
\eeq
A similar computation can be carried out by replacing $(p,\gamma)$ with $(2p,\gamma/2)$ to obtain
\beq\label{eqn55}
\| \wt{a}_n \|_{L^{2p}(S_n)} \leq  ce^{-\gamma n\sigma /2}.
\eeq

Now we estimate the derivative of $\wt{a}_n$. By Proposition \ref{prop26}, for any $\eps>0$, there is a constant $c$ such that
\beq\label{eqn56}
\sup_{S_n} | F_{\wt{a}_n} | \leq ce^{\eps n\sigma}.
\eeq
By Lemma \ref{lemma52} and the Coulomb gauge condition, we get
\beq\label{eqn57}
\begin{split}  
\| \wt{a}_n \|_{W^{1,p}(S_n)} &\leq c \Big( \| d \wt{a}_n \|_{L^p(S_n)} + \| \wt{a}_n\|_{L^p(S_n)} \Big) \\
                                &\leq c \Big( \| F_{\wt{a}_n} \|_{L^p(S_n)} + \| \wt{a}_n  \|_{L^{2p}(S_n)}^2 +  \| \wt{a}_n \|_{L^p(S_n)} \Big) \\
                                &\leq c \Big( e^{\eps n\sigma } +e^{-\gamma n} \Big) \\
																&\leq c e ^{\eps n\sigma},
\end{split}
\eeq
where the second last inequality comes from \eqref{eqn54}, \eqref{eqn55} and \eqref{eqn56}.

Next, we glue together the gauge transformations $k_0 g_n$ defined on the bands $S_n$. Notice that multiplying $g_n$ by a constant element does not alter the above estimates. So we may assume that for every $n$, there is a point $x_n \in S_n \cap S_{n+1}$ at which $g_n(x_n)=g_{n+1}(x_n)$. 

\noindent{\bf Claim.} {\it There exists $N_0 \in {\mb N}$ such that for $n \geq N_0$, there exists a smooth function $s_n: S_n \cap S_{n+1} \to {\mf k}$ such that $k_0 g_{n+1}  = k_0 g_n e^{s_n}$ and
\beq\label{eqn58}
\| s_n \|_{W^{1,p}(S_n \cap S_{n+1})} \leq ce^{-\gamma n \sigma }, \quad \| s_n \|_{W^{2,p}(S_n \cap S_{n+1})} \leq ce^{\eps n\sigma }.
\eeq
}

\noindent{\it Proof of the claim.} Define the quotient map 
\beqn
\pi_n: B_n:=[ n \sigma + \sigma, n \sigma + 1]\times [0, 2\pi] \to S_n \cap  S_{n+1}
\eeqn
by identifying $(s, 0)$ with $(s, 2\pi)$. Then there exists $s: B_n \to {\mf g}$ such that $k_0 g_{n+1} = k_0 g_n e^{s_n}$ and $s_n$ vanishes at a point of $B_n$. By this relation, we have
\beqn
\pi_n^* \wt{a}_{n+1} = e^{- s_n} ( \pi_n^* \wt{a}_n ) e^{s_n} + d s_n.
\eeqn
Therefore by \eqref{eqn54}, we have
\beqn
\|  d s_n \|_{L^p(B_n)} \leq  \| \wt{a}_n\|_{L^p(S_n)} + \| \wt{a}_{n+1}\|_{L^p(S_{n+1})} \leq c e^{-\gamma n \sigma}.
\eeqn
Since $s_n$ vanishes at one point, $\| s_n \|_{W^{1, p}(B_n)}$ has similar bound with possibly different value of $c$. Therefore, by Sobolev embedding, for $n$ sufficiently large, $\| s_n\|_{C^0(B_n)}$ is very small and hence it descends to a smooth function on $S_n \cap S_{n+1}$. Then the first inequality of \eqref{eqn58} follows. On the other hand, by \eqref{eqn57} and Sobolev embedding, one has
\begin{multline*}
\| d s_n \|_{W^{1, p}(S_n\cap S_{n+1})} \\
\leq \| \wt{a}_n \|_{W^{1, p}(S_n)} + \| \wt{a}_{n+1}\|_{W^{1, p}(S_{n+1})} +  \| d s_n \|_{L^p (S_n\cap S_{n+1})} \| \wt{a}_n \|_{L^\infty(S_n)} \leq c e^{\epsilon n \sigma}.
\end{multline*}
The claim is then proved.

Suppose $\beta_n:S_n \cap S_{n+1} \to [0,1]$ be a cut-off function that is $1$ in the neighborhood of $\partial S_{n+1}$ and $0$ in the neighborhood of $\partial S_n$. Further, for different values of $n$, the functions $\beta_n$ are translations of each other. We define a gauge transformation  
\begin{align}\label{eqn510}
&\ \wt{g}: \bigcup_{n \geq N_0} S_n \to K,\ & \wt{g}(z):= \left\{ \begin{array}{cc} k_0 g_n,\ & z\in S_n\bs (S_{n-1} \cup S_{n+1})\\
                       k_0 g_n e^{\beta_n s_n} & z \in S_n \cap S_{n+1}.
											\end{array} \right.
\end{align}
Extend $\wt{g}$ arbitrarily to $S_n$ for $n < N_0$. On the intersection $S_n \cap S_{n+1}$, $e^{\beta_n s_n}$ satisfies similar bounds as $e^{s_n}$ (see \eqref{eqn58}). Therefore, after applying the global gauge transformation $\wt{g} : [\log R,\infty) \times S^1 \to K$, the connection $\wt{a}:= \wt{g}^{-1} \cdot  \wt\psi d\theta$ satisfies the same asymptotic bounds as $\wt{a}_n$, namely,
\beq\label{eqn511}
\| \wt{a} \|_{W^{1, p}(S_n)} \leq c e ^{\eps n \sigma }, \quad  \| \wt{a} \|_{L^p(S_n)} \leq c e^{-\gamma n \sigma}.
\eeq

Lastly, denote the pullback of $\wt{a}_n$ to $C_R$ via the exponential map by $\ck a$, one has
\beqn
\| \check{a} \|_{L^p(\exp(S_n))} \leq c e^{(-1 + \frac 2 p -\gamma)n\sigma}, \quad \| \nabla \check{a} \|_{L^p(\exp(S_n))}  \leq c e^{(-2+\frac 2 p + \eps)n\sigma }.  
\eeqn
Then, given $\delta \in ( 1- \frac{2}{p}, 1)$, if we choose $\gamma>0$ and $\epsilon>0$ such that  
\beqn
\delta - 1 + \frac 2 p -\gamma<0, \quad \delta -2+ \frac{2}{p} + \eps<0.
\eeqn
then $a \in W^{1, p, \delta}(C_R, \Lambda^1 \otimes {\mf g})$. This finishes the proof of Proposition \ref{prop36}.
\end{proof}

\section{Proof of Proposition \ref{prop37}}\label{section6}

To proceed, we need to bound the derivatives $d u_i$ (after appropriate gauge transformations). For convenience we choose a smooth embedding $X \hookrightarrow {\mb R}^N$ (or just an embedding of a $K$-invariant, precompact open subset of $X$ which contains the images of all $u_i$) and view each $u_i$ as an ${\mb R}^N$-valued function.

\begin{lemma}\label{lemma61} 
Suppose a sequence ${\bf v}_i = (u_i, a_i) \in \wt{M}^K(C_R, X)$ converge to ${\bf v}_\infty = (u_\infty, a_\infty)$ in c.c.t.. Then, for all $i$ including $i = \infty$, there are gauge transformations $\check{g}_i \in {\mc K}^{2, p}_{\rm loc}(C_R)$ such that the following are satisfied. Denote $\check{\bf v}_i = (\check{u}_i, \check{a}_i):= \check{g}_i \cdot {\bf v}_i$. 
\begin{enumerate}
\item \label{lemma61a} {\hfil $\displaystyle \sup_i \| \check{a}_i \|_{W^{1, p, \delta}(C_R)} < +\infty$.}

\item \label{lemma61b} $\displaystyle \lim_{z \to \infty} \check{u}_i(z) = x_i$ uniformly in $i$ (including $i = \infty$) and $\displaystyle \lim_{i\to \infty} x_i = x_\infty$.

\item \label{lemma61c} {\hfil $\displaystyle \sup_i \| \nabla \cu_i\|_{W^{1,p,\delta}(C_R)} < +\infty$.}

\end{enumerate}
\end{lemma}

\begin{proof}
By Proposition \ref{prop36}, we can assume that $\| a_i \|_{W^{1, p, \delta}(C_R)}$ is uniformly bounded. Since the sequence ${\bf v}_i$ is uniformly bounded, we have
\beqn
\sup_i  \Big( \| {\mc X}_{a_i}(u_i) \|_{L^\infty(C_R)} +  \| {\mc X}_{a_i}(u_i) \|_{L^{p, \delta}(C_R)} \Big) < +\infty.
\eeqn
Using bound on the energy densities of ${\bf v}_i$ provided by Proposition \ref{prop26}, we have
\beq\label{eqn61}
\begin{split}
\sup_i  \| d u_i \|_{L^\infty(C_R)} \leq &\ \sup_i \Big( \| d_{a_i} u_i \|_{L^\infty(C_R)} + \| {\mc X}_{a_i}(u_i) \|_{L^\infty(C_R)} \Big) < \infty; \\
\sup_i  \| d u_i \|_{L^{p, \delta}(C_R)} \leq &\ \sup_i \Big( \| d_{a_i} u_i \|_{L^{p, \delta} (C_R)} + \| {\mc X}_{a_i}(u_i) \|_{L^{p, \delta} (C_R)} \Big) < \infty.
\end{split}
\eeq
Notice that the norm of the 1-form $d u_i$ is comparable to the norm of $\nabla u_i$ computed as an ${\mb R}^N$-valued function, the above gives an upper bound on the corresponding norms of $\nabla u_i$. Then by Lemma \ref{lemma212}, each $u_i$ has a limit $u_i(\infty) \in X \subset {\mb R}^N$ at infinity. Moreover, by \eqref{eqn213}, there is a constant $C>0$ independent of $i$ such that 
\beqn
| u_i(z) - u_i(\infty) | \leq C |z|^{1 - \frac{2}{p}- \delta} \| \nabla u_i \|_{L^{p, \delta}(C_R)},\ \forall z\in C_R.
\eeqn
Hence the convergence $\displaystyle \lim_{i \to \infty} u_i(z) = u_i(\infty)$ is uniform in $i$. Further, by Corollary \ref{cor211}, the corresponding sequence $\bar{x}_i \in \bar{X}$ in the symplectic quotient $\bar{X}$ converges to $\bar{x}_\infty \in \bar{X}$. Therefore we can modify each $g_i$ by a constant element $g_i' \in K$ such that 
\beqn
\lim_{i \to \infty} g_i' u_i(\infty) = g_\infty' u_\infty(\infty).
\eeqn
Denote $\check{g}_i(z) = g_i' g_i(z)$ and $x_i = g_i' u_i(\infty)$ for all $i$ including $i = \infty$. Then the sequence $\check{v}_i:= \check{g}_i \cdot {\bf v}_i$ satisfies Part \eqref{lemma61a} and \eqref{lemma61b}. 

Now we bound the second derivatives of $\check{u}_i$. By Part \eqref{lemma61b}, there exist $i_0$, $R' \geq R$ such that for all $i \geq i_0$, $\check{u}_i(C_{R'})$ is contained in a local coordinate chart of $x_\infty$. The c.c.t. convergence of ${\bf v}_i \to {\bf v}_\infty$ implies that up to gauge transformation $\check{u}_i|_{C_R \smallsetminus C_{R'}}$ converges to $\check{u}_\infty|_{C_R \smallsetminus C_{R'}}$ up to the second order, which implies the uniform bound of $\| \nabla \check{u}_i\|_{W^{1, p, \delta}(C_R \smallsetminus C_{R'})}$. Hence it suffices to estimate $\| \nabla \cu_i\|_{W^{1, p, \delta}(C_{R'})}$. 

Using the local coordinate chart around $x_\infty$ one can view $\check{u}_i|_{C_{R'}}$ as maps into ${\mb R}^{2n}$. By the holomorphicity equation, one has
\begin{align*}
\partial_s \check{u}_i + J( \check{u}_i)\partial_t \check{u}_i= - {\mc X}_{\check{a}_i}^{0,1}( \check{u}_i):= - {\mc X}_{\check\phi_i}( \cu_i) - J(u_i) {\mc X}_{\check\psi_i}( \cu_i).
\end{align*}
Differentiating with respect to $s$, 
\begin{align}\label{eqn62}    
(\partial_s + J \partial_t) \partial_s \cu_i=-(\partial_s J)\partial_t \cu_i-  \partial_s {\mc X}_{\check{a}_i}^{0,1} (\cu_i).
\end{align}
On any radius 2 disk $B_2(z)$, centered at $z \in C_{R'}$, using \eqref{eqn61}, the terms in the right-hand-side of \eqref{eqn62} can be bounded as 
\beqn
\| \partial_s {\mc X}_{\check{a}_i}^{0,1}(\cu_i) \|_{L^p(B_2(z))} \leq c_1 \Big( \| d \cu_i \|_{L^\infty(B_2(z))}  \| \check{a}_i \|_{L^p(B_2(z))} + \| \nabla \check{a}_i \|_{L^p(B_2(z))} \Big) \leq c_2  \| \check{a}_i \|_{W^{1, p}(B_2(z))};
\eeqn
\beqn
\| (\partial_s J)\partial_t \cu_i \|_{L^p(B_2(z))} \leq c_3 \| d \cu_i \|_{L^\infty(B_2(z))}  \| d \cu_i \|_{L^p(B_2(z))} \leq c_4 \|  d \cu_i \|_{L^p(B_2(z))}.
\eeqn 
Here $c_1, c_2, c_3, c_4$ are independent of $i$. By \eqref{eqn62}, and the interior elliptic estimate for the $\ov\partial$ operator (see \cite[Proposition B.4.9]{McDuff_Salamon_2004}), there is $c_5>0$ such that for all $i \geq i_0$, 
\begin{align}\label{eqn63}
\| \partial_s u_i \|_{W^{1,p}(B_1(z))} \leq c_5 \Big( \| a_i \|_{W^{1, p}(B_2(z))} + \| d u_i\|_{L^p(B_2(z))}  \Big).
\end{align}
A similar estimate holds for $\partial_t u_i$. Cover $C_{R'+2}$ by countably many $B_1(z_k)$ ($k=1, 2, \ldots$), we obtain
\begin{multline*}
\| d u_i \|_{W^{1, p, \delta}(C_{R'+2})} \leq c_6 \sum_{k=1}^\infty |z_k|^\delta \| d u_i \|_{W^{1, p}(B_1(z_k))} \\
\leq  c_7 \sum_{k=1}^\infty |z_k|^\delta \Big( \| a_i \|_{W^{1, p}(B_2(z_k))} +  \| d u_i \|_{L^p(B_2(z_k ))} \Big)\\
\leq c_8  \Big( \| a_i \|_{W^{1, p, \delta}(C_{R'})} + \| d u_i \|_{L^{p, \delta}(C_{R'})} \Big)< +\infty.
\end{multline*}
This proves Part \eqref{lemma61c} and hence this proposition.
\end{proof}

\begin{lemma}\label{lemma62}
Given $R\geq 1$. Suppose we have a sequence $\check{\bf v}_i = (\cu_i, \check{a}_i) \in \wt{M}^K(C_R, X)^{1, p}_{\rm loc}$ (including $i=\infty$) and a sequence $x_i \in \mu^{-1}(0)$ (including $i=\infty$) that satisfying the following conditions.
\begin{enumerate}
\item\label{lemma62a} $\displaystyle \sup_{1\leq i\leq \infty}  \Big( \| \check{a}_i \|_{W^{1, p, \delta}(C_R)} + \| \nabla \cu_i \|_{W^{1, p, \delta}(C_R)} \Big) < +\infty$.
	
\item \label{lemma62b} $\displaystyle \lim_{z \to \infty} \cu_i(z) = x_i$ uniformly in $i$ (including $i = \infty$) and $\displaystyle \lim_{i \to \infty} x_i = x_\infty$.
\end{enumerate}
Then there exist a sequence $g_i \in {\mc K}^{2, p}_{\rm loc}(C_R)$ 
(including $i=\infty$) such that \eqref{lemma62a} and \eqref{lemma62b} continue to hold for $g_i \cdot \check{\bf v}_i$ and moreover
\begin{enumerate}
\item[(c)] after removing finitely many items from the sequence, for certain $R' \geq R$, $\check\xi_i:= \exp_{\check{u}_\infty}^{-1} (g_i \cu_i)$ is well-defined on $C_{R'}$ such that $ d \mu( \cu_\infty) \cdot J \xi_i = 0$.

\item[(d)] if $\check{\bf v}_i$ converges to $\check{\bf v}_\infty$ in c.c.t., then $\check\xi_i$ converges to $0$ uniformly on any compact subset of ${\mb C}$.
\end{enumerate}
\end{lemma}

\begin{proof}
Let $U_X$ be a $K$-invariant neighborhood of $\mu^{-1}(0)$ where the $K$-action is free. One can split $TU_X$ into the direct sum of the tangent direction of $K$-orbits and its orthogonal complement with respect to the metric $\omega_X(\cdot, J \cdot)$, i.e.,
\beqn
TU_X = ({\mf k}_X)^\bot \oplus {\mf k}_X = {\rm ker} (d\mu \circ J ) \oplus {\mf k}_X.
\eeqn
By shrinking $U_X$ if necessary, and taking $\epsilon>0$ small enough, the map
\beqn
{\mc L}: B_\epsilon ( {\rm ker}(d\mu\circ J) ) \times B_\epsilon ( {\mf k} ) \to X \times X, \quad ((x,\xi), s) \mapsto (x,e^s \exp_x \xi)
\eeqn	
is a diffeomorphism onto its image and its image contains a neighborhood $N^\epsilon(\Delta  U_X)$ of $U_X \times U_X$ in $X \times X$. By hypothesis \eqref{lemma62b}, there exist $i_0\in {\mb N}$ and $R' \geq R$ such that 
\beqn
i \geq i_0,\ |z| \geq R' \Longrightarrow \Big( \cu_\infty(z), \cu_i(z) \Big) \in U_X \times U_X.
\eeqn
Therefore, there exist unique $s_i: C_{R'} \to {\mf k}$ and $\xi_i \in \Gamma( C_{R'}, u_\infty^* {\rm ker} (d\mu\circ J ) )$ such that 
\beqn
\Big(  \cu_\infty(z), \cu_i(z) \Big) = {\mc L} \Big( ( \cu_\infty(z), \check\xi_i(z)), s_i(z) \Big) = \Big( \cu_\infty(z),\  e^{s_i(z)} \exp_{\cu_\infty(z)} \check\xi_i(z) \Big). 
\eeqn
Since $\| \nabla \cu_i \|_{W^{1, p, \delta}(C_{R'})}$ is uniformly bounded, so is $\| d s_i \|_{W^{1, p, \delta}(C_{R'})}$. Extend $s_i$ to $C_R$ of regularity $W^{2, p}_{\rm loc}$ such that $\| ds_i \|_{W^{1, p, \delta}(C_R)}$ is uniformly bounded. We claim that $g_i:= e^{s_i}$ satisfy the requirement.

Indeed, for $g_i \cdot \check{\bf v}_i$, \eqref{lemma62a} and \eqref{lemma62b} still hold because of the uniform bounds on $\| d s_i \|_{W^{1, p, \delta}(C_R)}$ and $\| s_i \|_{L^\infty}$. (c) follows from the construction. Further, (c) gives a pointwise gauge fixing over $C_{R'}$. Hence if $\ck{\bf v}_i$ converges to $\ck{\bf v}_\infty$ in c.c.t., which implies that $g_i \cdot \cu_i$ converges to $\cu_\infty$ modulo $K$-action (uniformly over compact sets), $g_i \cdot \cu_i|_{C_{R'}}$ converges to $\cu_\infty|_{C_{R'}}$ without applying any further gauge transformations (uniformly over compact sets). This proves (d). 
\end{proof}

\begin{lemma}\label{lemma63} 
Suppose ${\bf v}_i = (u_i, a_i) \in \wt{M}^K({\mb C}, X)^{1, p}_{\rm loc}$ converges to ${\bf v}_\infty = (u_\infty, a_\infty) \in \wt{M}^K({\mb C}, X)^{1, p}_{\rm loc}$ in c.c.t. and $u_i$ converges to $u_\infty$ uniformly on ${\mb C}$. Then\
\beq\label{eqn64}
\lim_{i \to \infty} \Big( \| F_{a_i} - F_{a_\infty}\|_{L^{p, \delta}} +  \| d\mu(u_i) \cdot ( \exp_{u_\infty}^{-1} u_i) \|_{L^{p, \delta}}  \Big) = 0.
\eeq
\end{lemma}

\begin{proof}
By Proposition \ref{prop26}, for any $\nu>0$, one has
\beqn
\sup_i \limsup_{z \to \infty}  \Big( |z|^{2-\nu} |\mu(u_i(z))| \Big) < \infty.
\eeqn
Take $\nu<2-\frac 2 p -\delta$ and $\gamma:=2-\frac 2 p -\delta -\nu$. Then for all $r$ and $i$ including $\infty$, 
\beqn
\| \mu(u_i) \|_{L^{p,\delta}( C_r)} \leq c_1 r^{-\gamma}
\eeqn
for certain $c_1>0$ independent of $i$. On the other hand, for $r$ sufficiently large, $u_\infty(C_r)$ is contained in a $K$-invariant open subset of $X$ for which there exist $c_2, \delta > 0$ such that 
\beqn
x \in S,\ v \in T_x X,\ |v| \leq \delta \Longrightarrow | d \mu(x) \cdot v | \leq c_2 |\mu(\exp_xv)-\mu(x)|.
\eeqn
For large $i$, we can define $\xi_i$ by $u_i = \exp_{u_\infty} \xi_i$ which satisfies $\| \xi_i \|_{L^\infty} \leq \delta$. Therefore
\beq\label{eqn67}
\begin{split}
\| F_{a_i} - F_{a_\infty}\|_{L^{p, \delta}(C_r)} \leq &\ \|\mu(u_\infty)\|_{L^{p, \delta}(C_r)} + \| \mu(u_i) \|_{L^{p, \delta}(C_r)}  \leq 2 c_1 r^{-\gamma}; \\
\| d \mu(u_\infty) \cdot \xi_1\|_{L^{p, \delta}(C_r)}  \leq &\ c_2 \Big( \|\mu(u_\infty)\|_{L^{p, \delta}(C_r)} + \| \mu(u_i) \|_{L^{p, \delta}(C_r)} \Big) \leq 2 c_1 c_2 r^{-\gamma}.
\end{split}
\eeq
Moreover, for any fixed $r$, by the uniform convergence $u_i|_{B_r} \to u_\infty|_{B_r}$, $\| F_{a_i} - F_{a_\infty} \|_{L^{p, \delta}(B_r)}$ and $\| d \mu(u_\infty)\cdot \xi_i \|_{L^{p, \delta}(B_r)}$ can be arbitrarily small. Together with \eqref{eqn67}, we prove \eqref{eqn64}.
\end{proof}

\begin{proof}[Proof of Proposition \ref{prop37}]
Choose $\delta_+$ slightly bigger than $\delta$. Apply Lemma \ref{lemma61} and \ref{lemma62} to the restrictions of ${\bf v}_i^o$ on the complement of a disk, one can find $R \geq 1$, a sequence $\ck{g}_i \in {\mc K}^{2, p}_{\rm loc}(C_R)$ (including $i = \infty$) and a sequence $x_i \in \mu^{-1}(0)$ (including $i=\infty$) satisfying the following conditions after removing finitely many items from the sequence. Denote $\ck{\bf v}_i:= (\ck{u}_i, \ck{a}_i):= \ck{g}_i \cdot {\bf v}_i^o$.
\begin{enumerate}

\item $\displaystyle \sup_{1\leq i\leq \infty} \big( \| \ck{a}_i \|_{W^{1, p, \delta_+}(C_R)} + \|  \nabla \ck{u}_i \|_{W^{1, p, \delta_+}(C_R)} \big) < +\infty$.

\item $\displaystyle \lim_{z \to \infty} \ck{u}_i (z) = x_i$ uniformly in $i$ and $\displaystyle \lim_{i\to \infty} x_i = x_\infty$.

\item $\ck\xi_i:= \exp_{\ck u_\infty'}^{-1} \ck u_i \in W^{2, p}_{\rm loc}(C_R, \ck u_\infty^* TX)$ is well-defined and $d \mu( \ck u_\infty) \cdot J \ck \xi_i = 0$.

\item $\ck u_i$ converges to $\ck u_\infty$ uniformly on any compact subsets of $C_R$.
\end{enumerate}

By Corollary \ref{cor211}, for large $i$, ${\rm hol}({\bf v}_i^o) = {\rm hol}({\bf v}_\infty^o) \in \pi_1(K)$. Choose a geodesic representative $e^{\lambda \theta}$, one can see that $e^{\lambda \theta} \ck g_i$ is null-homotopic, hence extends to some $g_i \in {\mc K}_{\rm loc}^{2, p}({\mb C})$. Moreover, since ${\bf v}_i^o$ converges to ${\bf v}_\infty^o$ in c.c.t., we can choose $g_i$ in such a way that ${\bf v}_i:= g_i \cdot {\bf v}_i^o:= (u_i, a_i)$ converges to ${\bf v}_\infty:= (u_\infty, a_\infty):= g_\infty \cdot {\bf v}_\infty^o$ in $W^{2, p}(B_R)$. It implies that $u_i$ converges to $u_\infty$ uniformly on compact subsets of ${\mb C}$.

We first prove that the uniform convergence $u_i \to u_\infty$ over ${\mb C}$. Indeed, since we have the uniform convergence $\displaystyle \ck u_i (z) = x_i$ and the convergence $\displaystyle \lim_{i \to \infty} x_i =x_\infty$, for any $\epsilon>0$, there exist $R_\epsilon>0$ and $i_\epsilon \in {\mb N}$ such that
\begin{multline*}
\sup_{i \geq i_\epsilon} \sup_{z \in C_{R_\epsilon}} |  u_i(z) - u_\infty(z) | = \sup_{i \geq i_\epsilon} \sup_{z \in C_{R_\epsilon}} |  \ck u_i (z) - \ck u_\infty (z) | \\
\leq \sup_{i \geq i_\epsilon} \sup_{z \in C_{R_\epsilon}} \Big( | \ck u_i(z) - x_i | + | x_i - x_\infty | + | x_\infty - \ck u_\infty (z) | \Big) \leq \epsilon.
\end{multline*}
Since $u_i \to u_\infty$ uniformly over $B_{R_\epsilon}$, we obtain the uniform convergence $u_i \to u_\infty$ on ${\mb C}$. 

To prove the $L^{p, \delta}$-convergence $a_i \to a_\infty$, notice that for some $c_1 >0$,
\beqn
\| a_i - a_\infty\|_{W^{1, p, \delta_+}} \leq c_1  \| a_i - a_\infty\|_{W^{1, p}(B_R)} + \| e^{\lambda \theta} \cdot \ck a_i - e^{\lambda \theta} \cdot \ck a_\infty \|_{W^{1, p, \delta_+}(C_R)}.
\eeqn
The right hand side is uniformly bounded in $i$. Hence $a_i - a_\infty$ has subsequential weak limits. We prove that the weak limit can only be zero. It suffices to show that $a_i$ converges to $a_\infty$ as distributions, which is already the case over $B_R$. On the other hand, for any compact subset $S \subset C_R$, by the condition that ${\bf v}_i |_S \to {\bf v}_\infty |_S$ modulo gauge transformation, there exist $g_{S, i}: S \to K$ such that $g_{S, i} \cdot {\bf v}_i |_S$ converges to ${\bf v}_\infty |_S$ uniformly. However since $u_i$ converges to $u_\infty$ uniformly, and the fact that $u_i(C_R)$ are all contained in a region where the $K$-action is free, $g_{S, i}$ must converge uniformly to the identity. Therefore, $a_i|_S$ converges to $a_\infty|_S$ as distributions and hence $a_i - a_\infty$ converges to zero weakly in $W^{1,p, \delta_+}$. Since $\delta_+ > \delta$, by Lemma \ref{lemma215}, 
\beqn
\lim_{i \to \infty} \| a_i - a_\infty\|_{L^{p, \delta}} = 0.
\eeqn

Now consider the norm $\nabla^{a_\infty} \xi_i$ where $\xi_i = \exp_{u_\infty}^{-1} u_i$. The $W^{2, p}$ convergence of ${\bf v}_i|_{B_R}$ to ${\bf v}_\infty|_{B_R}$ already tells that
\beq\label{eqn66}
\lim_{i \to \infty} \| \nabla^{a_\infty} \xi_i \|_{L^p(B_R)} = 0. 
\eeq
To estimate the norm on $C_R$, by the condition that $\| \nabla \ck u_i \|_{W^{1, p, \delta}_+}$ is uniformly bounded, one has subsequential weak limits of $\nabla \ck u_i$. However since $\ck u_i$ converges to $\ck u_\infty$ uniformly, one has the distributional convergence $\nabla \ck u_i \to \nabla \ck u_\infty$. Therefore the subsequential weak limits of $\nabla \ck u_i$ can only be $\nabla \ck u_\infty$, which, by Lemma \ref{lemma215}, implies $\| \nabla (\ck u_i- \ck u_\infty)\|_{L^{p, \delta}(C_R)} \to 0$. This is equivalent to 
\beqn
\lim_{i \to \infty} \| \nabla \ck \xi_i\|_{L^{p, \delta}(C_R)} = 0
\eeqn
which refers to no embedding into ${\mb R}^N$. On the other hand, by the $K$-equivariance of the covariant derivative and Lemma \ref{lemma216}, there is $c_2>0$ such that 
\beqn
\lim_{i \to \infty} \| \nabla^{a_\infty} \xi_i \|_{L^{p, \delta}(C_R)} = \lim_{i \to \infty}  \| \nabla^{\ck a_\infty} \ck \xi_i\|_{L^{p, \delta}(C_R)} \leq c_2 \lim_{i\to \infty} \Big( \| \nabla \ck \xi_i \|_{L^{p, \delta}(C_R)} + \| \ck\xi_i \|_{L^\infty} \Big) = 0.
\eeqn
Together with \eqref{eqn66} this shows 
\beqn
\lim_{i \to \infty} \| \nabla^{a_\infty} \xi_i\|_{L^{p, \delta}} = 0.
\eeqn

Lastly, by the uniform convergence $u_i \to u_\infty$ and the fact $d\mu(u_\infty) \cdot J \xi_i$ vanishes on $C_R$,
\beqn
\lim_{i \to \infty} \| d\mu(u_\infty) \cdot J \xi_i \|_{L^{p, \delta}({\mb C})} = \lim_{i \to \infty} \| d\mu(u_\infty) \cdot J \xi_i\|_{L^{p, \delta}(B_R)} = 0.
\eeqn
Together with Lemma \ref{lemma63} this finishes the proof of Proposition \ref{prop37}.
\end{proof}

\section{Proof of Proposition \ref{prop38}}\label{section7}

Our proof is a typical application of the implicit function theorem (Proposition \ref{prop76}) and condition \eqref{eqn36} is to guarantee the invertibility of certain linearized operator.

\begin{proof}[Proof of Proposition \ref{prop38}]
We recall that vortices close to ${\bf v}=(u,a)$ can be written as $(\exp_u\xi,a+\alpha)$. For any value of ${\bm \xi} = (\xi, \alpha) \in \hat{\mc B}_{\bf v}^\epsilon$, we define an operator, called the {\em Coulomb operator} as
\begin{multline}\label{eqn71}
{\mc T}^{\bm \xi}: B_\epsilon \big( W^{k+1,p,\delta}({\mb C}, {\mf k}) \big) \to W^{k-1,p,\delta}({\mb C}, {\mf k})\\
s \mapsto \dad \Big( e^s \cdot (a+\alpha) - a \Big) + d\mu(u) \cdot  \Big( J \exp_u^{-1} ( e^s \exp_u \xi ).  \Big)
\end{multline}
This operator is well-defined for a small $\epsilon = \epsilon(u) > 0$ (depending on the injectivity radius near the image of $u$). We remark that the condition $T^\xi(s)=0$ corresponds to the vortex $e^s(\exp_u\xi,a+\alpha)$ being in Coulomb gauge with respect to ${\bf v}$.

To simplify the notations, for $x \in X$, denote
\begin{align*}
&\ R_x:T_x X \to  {\rm End} ( {\mf k } ), & R_x(\xi) (\xi') = d\mu(x) \cdot \Big( J \Psi_x(\xi)^{-1} {\mc X}_{\xi'} ({\exp_x \xi}) \Big),
\end{align*}
where $\Psi_x (\xi): T_x X \to T_{\exp_x \xi} X$ is the parallel transport along the geodesic $\exp_x (t\xi)$ with respect to the Levi-Civita connection of $\omega_X( \cdot, J \cdot)$. $R_x$ varies smoothly with $x$ and induces a bundle map
\beqn
R_u: \Gamma({\mb C}, E_u) \to \Gamma({\mb C}, {\rm End} ({\mf k})).
\eeqn
The linearization of ${\mc T}^{\bm \xi}$ at $s \in B_\epsilon ( W^{k+1,p,\delta}({\mb C},{\mf k}) )$ is given by 
\beqn
\begin{split}
D{\mc T}^{\bm \xi}_s:W^{k+1,p,\delta}({\mb C},{\mf k}) &\to W^{k-1,p,\delta}({\mb C}, {\mf k}), \\
\xi' &\mapsto \dad d_{ e^s \cdot ( a+\alpha)} ( \xi' ) + R_u(\xi_s) (\xi')
\end{split}
\eeqn
where $\xi_s:=\exp_u^{-1}(e^s\exp_u\xi)$. It is continuous in $s$, so ${\mc T}^{\bm \xi}$ is differentiable. Denote
\beq\label{eqn72}
{\mc L}_{\bf v} (\xi'):= D{\mc T}_0^{\bf 0} (\zeta) = \dad d_a \xi' + d\mu(u) \cdot J {\mc X}_{\xi'}. 
\eeq

\begin{lemma}\label{lemma71}
There exist constants $\epsilon_1, c_1 >0$ such that if 
\beqn
\| \alpha \|_{W^{k,p,\delta}} +  \| \xi \|_{L^\infty} + \| s \|_{W^{k+1, p, \delta}} \leq \epsilon_1,
\eeqn
then 
\beqn
\| D {\mc T}^{\bm \xi}_s - D{\mc T}_0^{\bf 0} \| \leq c_1 \Big( \| \alpha \|_{W^{k, p, \delta}} + \| \xi \|_{L^\infty} +  \| s \|_{W^{k+1, p, \delta}} \Big).
\eeqn
\end{lemma}

\begin{proof}
Given $\xi' \in W^{k+1, p, \delta}({\mb C}, {\mf k})$, we have the following computation.
\beqn
\begin{split}
&\ \| ( D{\mc T}^{\bm \xi}_s - D{\mc T}_0^{\bf 0} ) (\xi') \|_{W^{k-1, p, \delta}} \\
&\ \leq \| \dad\, [\, e^s \cdot ( a+ \alpha) - a, \xi'\, ] \|_{W^{k-1, p, \delta}} + \| R_u( \xi_s)(\xi') - R_u(0)(\xi') \|_{W^{k-1, p, \delta}}\\
&\ \leq \| \dad\,  [\, e^s \cdot (a + \alpha) - ( a+ \alpha), \xi'\,	 ] \|_{W^{k-1, p, \delta}} +  \| \dad\, [\, \alpha, \xi' \,  ] \|_{W^{k-1, p, \delta}} \\
&\ +  \| ( R_u(\xi_s) - R_u(\xi) ) (\xi' ) \|_{W^{k-1, p, \delta}} + \| R_u(\xi)(\xi' ) - R_u(0)(\xi' )  \|_{W^{k-1, p, \delta}} \\
&\ \leq c_1 \big(  \| s  \|_{W^{k+1, p, \delta}} +  \| \alpha  \|_{W^{k, p, \delta}} +  \| \xi  \|_{L^\infty} \big) \| \xi' \|_{W^{k + 1, p, \delta}}.
\end{split}
\eeqn
The last inequality is true if $\epsilon_1$ is small enough. We also used the Sobolev embedding $W^{1, p, \delta}\hookrightarrow C^0$ implicitly. So the lemma is proved.
\end{proof}

\begin{prop}\label{prop72}
Given $k=0,1$, $p>2$ and $\delta \in (1-\frac 2 p, 1)$. Suppose ${\bf v} = (u, a)$ is an affine vortex of regularity $W^{1, p}_{\rm loc}$, with 
$a- \lambda d\theta \in W^{1,p,\delta}(C_R, \Lambda^1 \otimes {\mf g})$ 
for some $\lambda \in \Lambda_K$ and $R \geq 1$. Then the operator ${\mc L}_{\bf v}$ given by \eqref{eqn72} is an isomorphism.
\end{prop}

The proof is given in Subsection \ref{subsection71}. Let ${\mc K}_{\bf v}: W^{k-1, p, \delta}({\mb C}, {\mf k}) \to W^{k+1, p, \delta}({\mb C}, {\mf k})$ be the inverse of ${\mc L}_{\bf v}$. Define 
\beqn
C({\bf v}) = \max \Big\{\  \| {\mc K}_{\bf v}\|,\ 1\ \Big\} .
\eeqn
Then by Lemma \ref{lemma71}, there exist $\delta_{\rm coul} = \delta({\bf v})>0$ such that
\beqn
\| \alpha  \|_{W^{k, p, \delta}} +  \| \xi \|_{L^\infty} +  \| s \|_{W^{k+1,p,\delta}} \leq \delta( {\bf v}) \Longrightarrow  \| D{\mc T}^{\bm \xi}_s - D{\mc T}^{\bm \xi}_0  \| +  \| D {\mc T}^{{\bm \xi}}_0 - D{\mc T}^{{\bm 0}}_0 \|  \leq \frac{1}{2C({\bf v})}.	
\eeqn
If this holds, then $D{\mc T}^{\bm \xi}_0$ has a bounded inverse ${\mc Q}^{\bm \xi}_0$ with $\| {\mc Q}^{\bm \xi}_0 \| \leq 2 C({\bf v})$. Moreover, taking $\epsilon_{\rm coul} = \delta({\bf v}) / 8 C({\bf v})$. Then we can apply Proposition \ref{prop76} to the situation where
\beqn
X = W^{k+1, p, \delta}({\mb C}, {\mf g}),\ Y = W^{k-1, p, \delta}({\mb C}, {\mf g}),\ F = {\mc T}^{(\alpha, \xi)},\ \delta = \delta({\bf v}),\ C = 2C( {\bf v}).
\eeqn
It follows that there is a unique $s\in W^{k+1, p, \delta}({\mb C}, {\mf k})$ satisfying
\beqn
{\mc T}^{\bm\xi} (s) = 0,\ \| s \|_{W^{k+1, p, \delta}} \leq \delta({\bf v}).
\eeqn
Moreover, for $c_{\rm coul} = 4 C({\bf v})$, one has
\beqn
\| s \|_{W^{k+1, p, \delta}} \leq c_{\rm coul} \| {\mc T}^{{\bm \xi}}(0)\| =  c_{\rm coul} \| \dad \alpha + d\mu(u) \cdot J\xi \|_{W^{k-1, p, \delta}}.
\eeqn
This finishes the proof of Proposition \ref{prop38}.
\end{proof}

\subsection{Proof of Proposition \ref{prop72}}\label{subsection71}

Firstly we recall the invertibility of $-\Delta + 1$. 
\begin{lemma} \label{lemma73}{\rm(Calderon's theorem, \cite[Theorem V.3]{Stein_70})} The operator
\beq\label{eqn73}
- \Delta + 1: W^{k+1,p}({\mb C},{\mb R}) \to W^{k-1,p}({\mb C},{\mb R})
\eeq
is an isomorphism for all $p>1$.
\end{lemma}

\begin{lemma}\label{lemma74} 
Given $k=0$ or $1$, $\lambda \in \Lambda_K$, $x \in \mu^{-1}(0)$. Let ${\bf v} = (u, a)$ be a smooth gauged map from ${\mb C}$ to $X$ such that over $C_R$, $u = e^{\lambda \theta} x$ and $a = - \lambda d\theta$. Then the operator 
\begin{align*}
& {\mc L}_{\bf v}: W^{k+1, p, \delta}({\mb C}, {\mf k}) \to W^{k-1, p, \delta}({\mb C}, {\mf k}), & {\mc L}_{\bf v}(s) \mapsto \dad d_a s + d\mu(u) \cdot J {\mc X}_s (u)
\end{align*}
is invertible for all $p>1$ and $\delta \geq 0$.
\end{lemma}

\begin{proof}
The proof is firstly carried out for $\delta=0$. By Remark \ref{remark44}, $L_x^* L_x$ is diagonalizable with positive eigenvalues $c_i$. Further, by the assumption of this lemma, 
\beq\label{eqn74}
{\mc L}_{\bf v}|_{C_R} = {\rm Ad}_{e^{\lambda \theta}} ( - \Delta + L_x^* L_x ) {\rm Ad}_{e^{- \lambda \theta}}.
\eeq
Therefore, the bundle $C_R \times {\mf k}$ splits into trivial line bundles $\ell_i$ ($1\leq i \leq n$) with $- \Delta + L_x^* L_x|_{\ell_i} = - \Delta + c_i$. By Lemma \ref{lemma73} and a scaling argument, the operator 
\beqn
{\mc L}_i = - \Delta  + c_i: W^{k+1, p}({\mb C}, \ell_i) \to W^{k-1, p}({\mb C}, \ell_i)
\eeqn
is an isomorphism. So there is an inverse ${\mc K}: W^{k-1,p}({\mb C}, {\mf k}) \to W^{k+1,p}( {\mb C}, {\mf k})$ of $- \Delta  + L_x^* L_x$.

We will now construct an inverse outside the ball $B_{2R}$ using ${\mc K}$.  Consider $f \in W^{k-1,p}( {\mb C}, {\mf k})$. Choose a radially symmetric cut-off function $\beta_R : {\mb C} \to [0,1]$ that is $0$ in $B_R$ and $1$ in $C_{2R}$. Since $k \leq 1$, we can define $\hat f \in W^{k-1,p}( {\mb C}, {\mf k})$ as
\beqn
\hat{f}(z) = \left\{ \begin{array}{cc} {\rm Ad}_{e^{- \lambda \theta}} (f), & |z|\geq R;\\
                                          0, & |z|\leq R.
\end{array} \right.
\eeqn
Take $\hat{u} = {\mc K}(\hat{f})$ and $u_{out}: = {\rm Ad}_{e^{- \lambda \theta}}(\beta_R \hat u)$. By \eqref{eqn74}, we have
\beqn
 {\mc L}_{\bf v} (u_{out})|_{C_{2R}} = {\rm Ad}_{e^{\lambda \theta}} ( - \Delta + L_x^* L_x) ( \hat{u}) = f.
\eeqn
Moreover, since $\beta_R \hat{u}$ is supported in $C_R$, where $e^{- \lambda \theta}$ and $\beta_R$ are uniformly bounded with derivatives up to order $k+1$, there are constants $c_1(R), c_2(R)>0$ such that 
\beqn
\| u_{out} \|_{W^{k+1, p}} \leq c(R) \| \hat{u} \|_{W^{k+1, p}} \leq c_1 (R) \| {\mc K} \| \| \hat{f} \|_{W^{k-1, p}({\mb C})} \leq c_2(R) \| f \|_{W^{k-1, p}}. 
\eeqn
Now, we extend the inverse to inside the ball $B_{2R}$. Consider the equation   
\beqn
{\mc L}_{\bf v}(u_{in}) = f - {\mc L}_{\bf v} (u_{out}) \text{ on $B_{2R}$},\ u_{in}|_{\partial B_{2R}}=0,\ u_{in} \in W^{k+1,p}(B_{2R},{\mf k}).
\eeqn
The above equation is an elliptic PDE with Dirichlet boundary condition and the differential operator is nonnegative. Hence it has a unique solution $u_{in}$. We extend $u_{in}$ by zero outside the ball $B_{2R}$ and set $u:=u_{in} + u_{out}$. The section $u$ is in $W^{1,p}$ and ${\mc L}_{\bf v} u=f$ holds weakly. In case $k=1$, by an elliptic regularity argument, we can show that $u$ is in $W^{2,p}$. By construction, the solution $u$ satisfies the norm bound $\| u \|_{W^{k+1, p}} \leq c \| f \|_{W^{k-1, p}}$.

Next, we show that ${\mc L}_{\bf v}$ is injective using the self-adjointness of the operator. Suppose $u \in W^{k+1,p}( {\mb C}, {\mf k})$ and ${\mc L}_{\bf v} u=0$. Then, $\int_{\mb C} \langle {\mc L}_{\bf v} u , f \rangle ds dt =0$ for all $f \in W^{k+1,q}$ where $q=\frac p {p-1}$. Integrating by parts, we see that $u$ satisfies $\int_{\mb C} \langle u , {\mc L}_{\bf v} f \rangle ds dt =0$ for all $f \in W^{k+1,q}$. But, we know that ${\mc L}_{\bf v}: W^{k+1, q} \to W^{k-1, q}$ is onto. Therefore, $u=0$.
  
Finally, we extend the result to the case of nonzero $\delta$. The map $W^{k,p} \ni u \mapsto \rho^{-\delta} u \in W^{k,p,\delta}$ is an isomorphism for $k=-1, 0, 1$. The map ${\mc L}_{{\bf v}, \delta}$ defined by
\beqn
\xymatrix{W^{k+1,p,\delta} \ar[r]^{\rho^\delta} & W^{k+1,p} \ar[r]^{{\mc L}_{\bf v}} &  W^{k-1,p} \ar[r]^{\rho^{-\delta}} & W^{k-1,p,\delta}},
\eeqn
is an isomorphism, and it differs from ${\mc L}_{\bf v}$ by a compact perturbation (see Proposition \ref{prop214}). Hence ${\mc L}_{\bf v} :W^{k+1,p,\delta} \to W^{k-1,p,\delta}$ is a Fredholm operator with index zero. Its kernel is contained in ${\rm ker} ( {\mc L}_{\bf v}: W^{k+1,p} \to W^{k-1,p})$, because $W^{k+1,p,\delta} \subset W^{k+1,p}$. This proves that the map ${\mc L}_{\bf v} : W^{k+1,p,\delta} \to W^{k-1,p,\delta}$ is invertible.
\end{proof}

\begin{lemma}\label{lemma75}
Suppose $p>2$, $\delta > 1 - \frac 2 p$ and $q=\frac p {p-1}$ is the conjugate exponent of $p$. Then, $L^{p,\delta}({\mb C})$ is contained in $L^{q,-\delta}({\mb C})$.
\end{lemma}
\begin{proof} 
Define $r$ by $\frac{1}{r} + \frac{2}{p} = 1$. Then $\frac{1}{r/q} + \frac{1}{p-1} = 1$. Moreover, $\delta>1 - \frac 2 p$ implies that $\rho^{-2\delta} \in L^r$. Then for $f\in L^{p, \delta}$, by H\"older's inequality, 
\beqn
\| \rho^{-\delta} f \|_{L^q}^q = \int_{{\mb C}} \rho^{-2q \delta} \rho^{q\delta} |f|^q  \leq \left[ \int_{\mb C}  ( \rho^{-2q \delta})^{\frac{r}{q}} \right]^{\frac{q}{r}} \left[ \int_{\mb C} ( \rho^{q\delta} |f|^q )^{p-1} \right]^{\frac{1}{p-1}} \leq\| \rho^{-2\delta} \|_{L^r}^q \| f \|_{L^{p, \delta}}^q.
\eeqn
So $f \in L^{q, -\delta}({\mb C})$.
\end{proof}

\begin{proof}[Proof of Proposition \ref{prop72}]
For $R>0$ sufficiently large, choose a gauged map ${\bf v}' = (u', a')$ such that $a'|_{C_R} = - \lambda d \theta$ and $u'|_{C_R} = e^{- \lambda \theta} x$. Denote $\alpha = a' - a$. We claim that the difference 
\beq\label{eqn75}
{\mc L}_{{\bf v}'} (s)  - {\mc L}_{{\bf v}} (s) = ( d_{a'}^{\, *} d_{a'} s - \dad d_a s ) + ( d\mu(u') \cdot J {\mc X}_s - d \mu(u) \cdot J {\mc X}_s )
\eeq
is a compact operator. Indeed, the first term on the right-hand-side above is
\beq\label{eqn76}
( d_{a + \alpha}^{\, *} d_{a+\alpha} - \dad d_a ) (s) = * [ \alpha \wedge * d_a s ] + \dad [ \alpha, s ] + * [ \alpha \wedge * [ \alpha, s] ].
\eeq
Then since $\alpha \in W^{1, p, \delta}$, by Proposition \ref{prop214} it follows easily that \eqref{eqn76} is compact. For similar reason the second term of the right-hand-side of \eqref{eqn75} is compact. 

Now by Lemma \ref{lemma74}, ${\mc L}_{{\bf v}'}: W^{k+1, p, \delta}({\mb C}, {\mf k}) \to W^{k-1, p, \delta}({\mb C}, {\mf k})$ is an isomorphism. Therefore ${\mc L}_{\bf v}$ is Fredholm and has index zero. So in order to prove the current proposition, it suffices to show that ${\mc L}_{\bf v}$ is injective. Suppose $s \in {\rm ker} ({\mc L}_{\bf v}) \subset W^{k+1, p, \delta}$. By Lemma \ref{lemma75}, $s \in W^{1, q, -\delta} = (W^{-1, p, \delta})^*$. By the definition of the (distributional) adjoint $\dad$, we have
\beqn
\int_{\mb C} \langle \dad d_a s +  d\mu(u) \cdot J {\mc X}_s, s \rangle dsdt = \int_{\mb C} \langle d_a s, d_a s \rangle dsdt +  \int_{\mb C} \omega_X( {\mc X}_s, J {\mc X}_s )dsdt	. 
\eeqn
The right hand side vanishes exactly if $d_a s$ and ${\mc X}_s$ vanish. Since the $K$-action is free near $u(C_R)$, it implies that $s =0$ so ${\mc L}_{\bf v}$ is injective. 
\end{proof}

The following version of implicit function theorem is used in the proof of Proposition \ref{prop38}. 

\begin{prop}\cite[Proposition A.3.4]{McDuff_Salamon_2004}\label{prop76}
Let $X, Y$ be Banach spaces and $U\subset X$ is an open neighborhood of $0\in X$. Let $F: U \to Y$ be a $C^1$ map. Suppose that $DF(0)$ is surjective and has a right inverse $Q$, with $\| Q \|  \leq C$. Suppose $\delta>0$ such that $B_\delta(X	) \subset U$ and
\beqn
x\in B_\delta(X) \Longrightarrow \| DF(x) - DF(0) \| \leq \frac{1}{2C}.
\eeqn
Suppose $\| F(0) \| < \frac{\delta}{4C}$. Then there exists a unique $x \in B_\delta(X)$ such that 
\begin{align*}
& F(x) = 0,\ &  x \in {\rm Im} Q \cap B_\delta(X).
\end{align*}
Moreover, $\| x \| \leq 2C \| F(0) \|$.
\end{prop} 

\appendix

\section{Local Model for ${\mb H}$-vortices (by Guangbo Xu)}\label{appendixa}

In the last section we provide the parallel discussion on the local model for vortices over the upper half plane, or ${\mb H}$-vortices as we call in \cite{Wang_Xu}. An important difference from the case of ${\mb C}$-vortices is that there is no notion of ``limiting holonomy'' at infinity. This fact greatly reduces the technicality in the discussion. Beside this difference, a lot of arguments can be done in a completely parallel fashion. So we only provide sketches of proofs whose details are easy to fill in.

This appendix is organized as follows. We first recall the concepts of ${\mb H}$-vortices and results proved in \cite{Wang_Xu}. Then we state and prove technical results analogous to Proposition \ref{prop36}, \ref{prop37} and \ref{prop38}.

\begin{defn} \hfill
\begin{enumerate}

\item A {\bf $K$-Lagrangian} of the Hamiltonian $K$-manifold $(X, \omega_X, \mu)$ is a compact $K$-invariant Lagrangian submanifold $L$ which is contained $\mu^{-1}(0)$.

\item A gauged map from an open subset $\Sigma \subset {\mb H}$ to $X$ (resp. a vortex on $(\Sigma, \partial \Sigma)$) is a gauged map ${\bf v} = (u, a)$ from $\Sigma \to X$ (resp. a vortex on $\Sigma$) satisfying 
\beq\label{eqna1}
u (\partial \Sigma) \subset L.
\eeq

\item An {\bf ${\mb H}$-vortex} is a bounded vortex over $({\mb H}, \partial {\mb H})$.
\end{enumerate}
\end{defn}

When use exponential maps to identify a small infinitesimal deformation with a geometric object (and for some other conveniences) we need a metric on $X$ satisfying special conditions. 
\begin{defn}\label{defna2}
A Riemannian metric $h_X$ on $X$ is called $(\omega_X, \mu, J, L)$-admissible if 
\begin{enumerate}
\item $h_X$ is $K$-invariant;

\item $J$ is isometric with respect to $h_X$;

\item $TL$ and $J TL$ are orthogonal with respect to $h_X$;

\item $T\mu^{-1}(0)$ is orthogonal to $J {\mc X}_a$ with respect to $h_X$, for all $a \in {\mf k}$.

\item $L$ is totally geodesic with respect to $h_X$.
\end{enumerate}
\end{defn}

By \cite[Lemma A.3]{Wang_Xu} there exist such admissible metrics on $X$.

\subsection{Asymptotic Behavior of ${\mb H}$-vortices}

Now we consider the asymptotic behavior of ${\mb H}$-vortices, which is very much analogous to the case of ${\mb C}$-vortices. We denote 
\beqn
S_R:= {\mb H}\smallsetminus B_R.
\eeqn

\begin{prop}{(cf. \cite{Wang_Xu})\label{propa3}}
\begin{enumerate}
\item \label{propa3a} If ${\bf v} = (u, a)$ is a bounded vortex on $S_R$, then there is a $K$-orbit $O \subset L$ such that
\beqn
\lim_{z \to \infty} d(u(z), O) = 0.
\eeqn

\item \label{propa3b} If ${\bf v}_i = (u_i, a_i)$ is a sequence of vortices on $S_R$ that converges to ${\bf v}_\infty = (u_\infty, a_\infty)$ in c.c.t., then there exist $K$-orbits $O_i \subset L$, including $i=\infty$, such that 
\beq\label{eqna2}
\lim_{z \to \infty} d (u_i(z), O_i) = 0,
\eeq
uniformly in $i$. Moreover, $O_i$ converges to $O_\infty$ in $\bar{L}$.
\end{enumerate}
\end{prop}

\begin{proof}
Only Part \eqref{propa3a} of this proposition is included in \cite{Wang_Xu}. Part \eqref{propa3b} follows from the convergence of energy and the annulus lemma for vortices over ${\mb H}$ (see \cite[Proposition A.11]{Wang_Xu}), where the estimates can be established uniformly for each element of the sequence. Therefore the convergence \eqref{eqna2} is uniform for all $i$. Since $u_i$ converges to $u_\infty$ uniformly on compact subsets, it follows that $O_i$ converges to $O_\infty$.
\end{proof}

\begin{prop}\label{propa4}
Let ${\bf v}$ be a bounded vortex on $S_R$. Then for any $\eps>0$, we have
\beqn
\limsup_{z \to \infty} \Big[ |z|^{4 - \epsilon} e({\bf v})(z) \Big] < +\infty.
\eeqn
Further, if ${\bf v}_i$ is a sequence of vortices on $S_R$ that converges to ${\bf v}_\infty$ in c.c.t., then 
\beqn
\sup_i \limsup_{z \to \infty} \Big[ |z|^{4 - \epsilon} e({\bf v}_i)(z) \Big] < +\infty.
\eeqn
\end{prop}

\begin{proof}
For paths $(x, s): [0, \pi] \to X \times {\mf k}$ such that $x(0), x(\pi) \in L$ and such that $x([0, \pi])$ has sufficiently small diameter, one can define the local equivariant symplectic action
\beqn
{\mc A}_{\rm loc}^K (x, s) := - \int_{\mb D} u^* \omega + \int_0^\pi \mu(x(\theta)) \cdot s(t) d\theta.
\eeqn
This local action is invariant under gauge transformations $g\in C^\infty(S^1, K)$. Notice that the first term is the ordinary local action ${\mc A}_{\rm loc}$ for paths. 

Let $(\tau, \theta)\in [1, +\infty) \times [0, \pi]$ be the cylindrical coordinate on ${\mb H}$ near infinity. By Theorem \ref{propa3}, by using a gauge transformation, one can assume that for $\tau$ sufficiently large, $u(\tau, \theta)$ lies in a neighborhood of the $K$-orbit of a point $x \in L$. Choose a slice of the $K$-action through $x$, which is an embedded submanifold $S_x$ through $x$ that is transverse to $K$-orbits. We choose $S_x$ in such a way that it is orthogonal to the $K$-orbit at $x$. Since the $K$-action on $L$ is free, for $\tau$ large, $u(\tau, t)$ is contained in a region where the $K$-action is free. Then there exists a unique smooth map $g(\tau, \theta)\in K$ such that $g(\tau, \theta) u(\tau, \theta) \in S_x$. Then we replace $(u, a)$ by $g \cdot (u, a)$ and denote $(x_\tau, s_\tau) = ( u(\tau, \cdot), \psi(\tau, \cdot))$, where $\psi$ is the $d\theta$ component of $a$. The local equivariant action of $(x_\tau, s_\tau)$ is then defined.

By the usual isoperimetric inequality (\cite[Theorem 4.4.1 (ii)]{McDuff_Salamon_2004}), for any $c > 1/2\pi$, there exists $\delta>0$ such that for $l(x) \leq \delta$, ${\mc A}_{\rm loc}(x) \leq c l(x)^2$. Choose $c'\in (1/2\pi, c)$. For appropriate choice of $\epsilon$ and large $s$, one has
\begin{multline*}
{\mc A}_{\rm loc}^K ( x_\tau, s_\tau ) =  {\mc A}_{\rm loc}(x_\tau) + \int_0^\pi \mu(x_\tau(\theta)) \cdot s_\tau(\theta) d\theta \\
\leq  c' l(x_\tau )^2 + \epsilon^2 \int_0^\pi |s_\tau(\theta)|^2 d \theta + \frac{1}{2\epsilon^2} \int_0^\pi |\mu(x_\tau(\theta)|^2 d\theta \\
\leq  \pi c' \left[ \int_0^\pi \big( |x_\tau'(\theta)|^2 + | {\mc X}_{s_\tau} (x_\tau(\theta))|^2 + e^{2\tau} |\mu(x_\tau( \theta))|^2 \big) d\theta \right]\\
=  \pi c'\left[ \int_0^\pi \big( |x_\tau'(\theta) + {\mc X}_{\xi_\tau(\theta)} (x_\tau(\theta))|^2 + e^{2\tau} |\mu(x_\tau(\theta))|^2 \big) d\theta  - 2 \int_0^\pi \langle x_\tau'(\theta), {\mc X}_{\xi_\tau(\theta)} (x_\tau (\theta)) \rangle d\theta \right].
\end{multline*}
Note that since $x_\tau(\theta) \in S_x$ and $S_x$ is orthogonal to $K$-orbits at $x$, one has
\beqn
\langle \ x_\tau'(\theta), {\mc X}_{s_\tau(\theta)}(x_\tau(\theta))\   \rangle \leq c_1 d( x_\tau(\theta), x ) | x_\tau'(\theta) | |{\mc X}_{s_\tau(\theta)}(x_\tau(\theta)) |
\eeqn
for some constant $c_1 >0$. Then for $\tau$ large enough so that $d(x_\tau (\theta), x)$ is sufficiently small, one has
\beqn
{\mc A}_{\rm loc}^K ( x_\tau, s_\tau ) \leq \pi c \left[ \int_0^\pi \big(|x_\tau'(\theta) + {\mc X}_{s_\tau(\theta)} (x_\tau(\theta)) |^2 + e^{2\tau} |\mu(x_\tau(\theta))|^2 \big) d \theta\right].
\eeqn
Lastly, it is standard to check that for $\tau$ large enough, ${\mc A}_{\rm loc}^K ( x_\tau, s_\tau )$ is equal to the energy of the restriction of ${\bf v}$ to $S_{e^\tau}$. Then the above implies that
\beq\label{eqna3}
\frac{d}{d\tau }\Big[ e^{\tau / c\pi} E({\bf v}; S_{e^\tau }) \Big] \leq 0 \Longrightarrow \limsup_{\tau \to +\infty} \Big[ e^{ \tau /c\pi} E( {\bf v};  S_{e^\tau} )  \Big] < +\infty. 
\eeq
Notice that $1/c\pi$ can be arbitrarily close to $2$. By the mean-value estimate and the relation between Euclidean and cylindrical coordinates, one obtain the pointwise bound
\beqn
\lim_{z\to \infty} \Big[ |z|^{4- \epsilon}  e({\bf v})(z) \Big] <  + \infty.
\eeqn

Finally, the convergence in c.c.t. implies that the above estimates can all be extended uniformly to the sequence ${\bf v}_i$. This finishes the proof of the proposition.
\end{proof}

Now we prove an analogue of Proposition \ref{prop27}.

\begin{prop}\label{propa5}
Let ${\bf v} = (u, a)$ be a bounded vortex on $S_R$ in temporal gauge, i.e., $a = f d\theta$. Then there exist $x_0 \in L$ and a map $k_0 \in W^{1, p}([0, \pi], G)$ such that 
\begin{enumerate}
\item 
\beq\label{eqna4}
\lim_{r \to +\infty} \sup_{\theta \in [0, \pi]} d \Big( x_0, k_0(\theta)^{-1} u(r \cos \theta, r \sin \theta) \Big) = 0;
\eeq

\item For any $p>2$ and $0 < \gamma <\frac{2}{p}$, there is a constant $c>0$ such that 
\beq\label{eqna5}
\sup_{r \geq 1}  \left[ r^{\gamma} \Big\| a(r, \cdot) + (\partial_\theta k_0) k_0^{-1} \Big\|_{L^p([0, \pi])}\right] < +\infty.
\eeq

\item If ${\bf v}_i$ is a uniformly bounded sequence of vortices over $S_R$ in radial gauge, then the above upper limits are uniformly bounded for all $i$.
\end{enumerate}
\end{prop}

\begin{proof}
Let $(\tau, \theta)$ be the cylindrical coordinate on ${\mb H}$, i.e., $s + it = e^{\tau + i \theta}$. Over the strip $D_{\tau_0} = [\tau_0, \tau_0 + 1] \times [0, \pi]$, one write $a = \wt\psi d\theta$. Then the affine vortex equation reads 
\beqn
\partial_\tau u  + J ( \partial_\theta u + {\mc X}_{\wt\psi})  = 0,\ \partial_\tau \wt\psi + e^{2\tau} \mu(u) = 0.
\eeqn
Let $e(\tau, \theta)$ be the energy density function with respect to the cylindrical coordinates, i.e.,
\beqn
e(\tau, \theta) = |\partial_\tau u |^2 + e^{2\tau} |\mu(u)|^2.
\eeqn
In order to carry out the following estimate, instead of the standard metric $\omega_X (\cdot, J \cdot)$, we use an admissible metric (see Definition \ref{defna2}) to define the energy density. Then, by \cite[Lemma A.8]{Wang_Xu}, there exists a constant $c_1 >0$ such that 
\beqn
\Delta e \geq e^{4\tau} \omega({\mc X}_\mu, J {\mc X}_\mu) - c_1 (e + e^2). 
\eeqn
By Proposition \ref{propa4}, the energy density function $e$ decays uniformly as $s \to +\infty$. Hence for $s$ sufficiently large, we may assume for some $c_2>0$,
\beq\label{eqna6}
\Delta e \geq \frac{1}{c_2} e^{4 \tau} |\mu(u)|^2 - c_2 e.
\eeq
On the other hand, since the metric is admissible, one can extend $e$ across the boundary of the strip by reflection. Let $\wt{e}$ be the extension of $e$, which still satisfies the differential inequality \eqref{eqna6}.

\noindent {\bf Sublemma.} \cite[Lemma 9.2]{Gaio_Salamon_2005} Suppose $h, f: B_{R+r} \to [0, +\infty)$ satisfy $\Delta h \geq f - c h$. Then
\beqn
\int_{B_R} f \leq ( c + \frac{4}{r^2}) \int_{B_{R+r}} h.
\eeqn

Apply this sublemma to $f = c_2^{-1} e^{4\tau} |\mu(u)|^2$ (which is extended across the boundary by reflection), $h = \wt{e}$, by \eqref{eqna3}, for any $\epsilon>0$ and some $c_3>0$, 
\beqn
\int_{D_{\tau_0}}  e^{4 \tau} |\mu(u)|^2 d\tau d \theta \leq c_3 e^{-(2 -\epsilon) \tau_0}.
\eeqn
Then for some $c_4>0$, 
\beqn
\| \mu(u) \|_{L^p(D_{\tau_0})} \leq \| \mu(u) \|_{L^2}^{\frac{2}{p}} \| \mu(u)\|_{L^\infty}^{1-\frac{2}{p}} \leq c_4 e^{ - (2 + \frac{2}{p} - \epsilon) \tau_0}.
\eeqn
Then if $\tau_0 \leq \tau_1 \leq \tau_2 \leq \tau_0 + 1$, we have
\begin{multline*}
\big| \wt\psi( \tau_2, \theta) - \wt\psi( \tau_1, \theta) \big|  =  \big| \int_{\tau_1}^{\tau_2}	 \partial_\tau \wt\psi ( \tau, \theta) d \tau \big| = \big| \int_{\tau_1}^{\tau_2} e^{2 \tau} \mu(u) d \tau \big| \\
\leq e^{2 \tau_0 + 2} \int_{\tau_0}^{\tau_0 + 1} | \mu(u(\tau, \theta)) | d \tau \leq e^{2 \tau_0 + 2}  \left[ \int_{\tau_0}^{\tau_0 + 1} |\mu(u( \tau, \theta)) \big|^p d\tau \right]^{\frac{1}{p}}.
\end{multline*}
Hence for some $c_5>0$,
\beqn
\| \wt\psi(\tau_2, \cdot) - \wt\psi(\tau_1, \cdot) \|_{L^p([0, \pi])} \leq e^{2 \tau_0 + 2} \| \mu(u) \|_{L^p(D_{\tau_0})} \leq c_5 e^{- (\frac{2}{p} - \epsilon) \tau_0}.
\eeqn
This implies that $\wt\psi( \tau, \cdot)$ converges in $L^p$ to a function $\wt\psi_\infty \in L^p([0, \pi], {\mf k})$. Define $k_0$ by 
\beqn
(\partial_\theta k_0) k_0^{-1} = - \wt\psi_\infty,\ k_0 (0) = {\rm Id}.
\eeqn
Then, similar to the case of ${\mb C}$-vortices, we can find $x_0 \in L$ and to see that \eqref{eqna4} and \eqref{eqna5} hold.
\end{proof}

\subsection{More technical results}\label{appendixa2}

Let $p>2$ and $\delta \in (1- \frac{2}{p}, 1)$. 

For any open subset $S \subset {\mb H}$, let $\wt{M}^K(S, X, L)$ be the set of smooth vortices ${\bf v} =(u, a)$ in $X$ with boundary condition $u(\partial S) \subset L$ and let $\wt{M}^K(S, X, L)^{1, p}_{\rm loc}$ be the set of vortices of regularity $W^{1, p}_{\rm loc}$. Let $M^K(S, X, L)$ be the set of gauge equivalence classes, which is equipped with the compact convergence topology.

\begin{figure}[ht]
\label{figure1}
\includegraphics{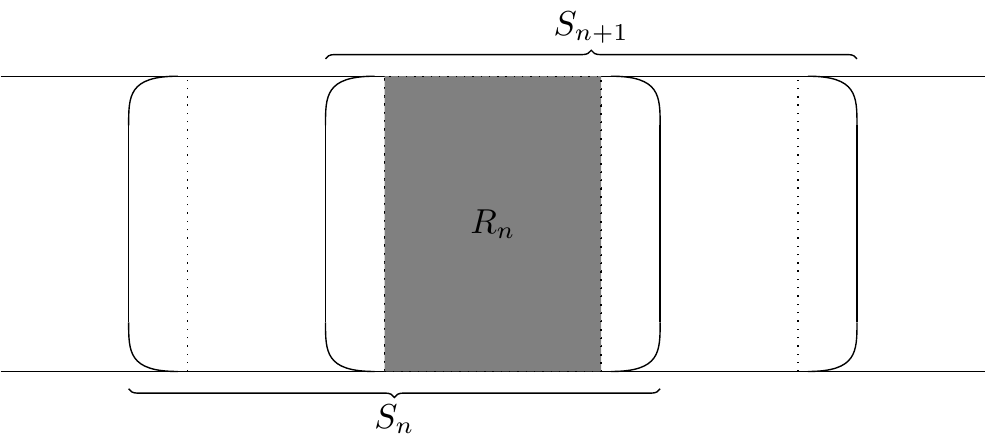}
\caption{}
\end{figure}

\begin{prop}\label{propa6}
Given $R \geq 1$. Assume a sequence $\ud {\bf v}_i  = (\ud u_i, \ud a_i) \in \wt{M}^K(S_R, X, L)$ converge to $\ud {\bf v}_\infty = (\ud u_\infty, \ud a_\infty) \in \wt{M}^K(S_R, X, L)$ in c.c.t.. Then there exist a sequence of gauge transformations $g_i \in {\mc K}^{2, p}_{\rm loc}(S_R)$ (including $i=\infty$) such that, if we write $g_i\cdot a_i = \ck a_i$, then 
\begin{align*}
&\ * \ck a_i|_{\mb R} = 0,&\ \sup_{1\leq i\leq \infty} \| \ck a_i \|_{W^{1, p, \delta}(S_R)} < +\infty.
\end{align*}

\end{prop}
\begin{proof}[Sketch of Proof] This proposition is analogous to Proposition \ref{prop36}. Again we only work with the case of a single vortex but the estimates below can be extended uniformly to the sequential case. Identify $S_R$ with an infinite strip $[\log R, +\infty) \times [0, \pi]$ via the exponential map. Transform $a$ into radial gauge, i.e., $a(\tau, \theta) = f_\tau(\theta) d \theta$. By Proposition \ref{propa5}, there exists $k_0\in W^{1, p}([0, \pi], K)$ such that 
\beqn	
\sup_{r \geq R} \Big[ e^{\tau \gamma} \| f_\tau + (\partial_\theta k_0) k_0^{-1} \|_{L^p([0, \pi])} \Big] < +\infty.
\eeqn
In other words, $k_0^{-1}$, viewed as a gauge transformation of regularity $W^{1, p}$, which is independent of $\tau	$, transforms $a$ to an $L^p$ connection with exponential decay. Cover $\wt{S}_R$ by countably many compact subsets $S_n$ with smooth boundaries such that $S_n$ and $S_{n+1}$ differ by a translation of $\wt{S}_R$ which is independent of $n$ (see Figure \ref{figure1}). Moreover, each intersection $S_n \cap S_{n+1}$ contains a rectangle $R_n$ which can cover $\wt{S}_R$. Then Lemma \ref{lemma52} implies that one can transform $k_0^{-1} \cdot a$ by a gauge transformation $g_n \in {\mc K}^{1, p}(S_n)$ such that over each $S_n$, $k_0 g_n$ is smooth and $(k_0 g_n)^{-1} \cdot a$ has similar estimates as in the proof of Proposition \ref{prop36}, where all constants are independent of $n$. Moreover, choose cut-off functions $\beta_n: S_n \cap S_{n+1} \to [0, 1]$ which equals 1 to the right of $R_n$ and equals 0 to the left of $R_n$ and such that different $\beta_n$'s are translations from each other. So when we glue $g_n$ with $g_{n+1}$ by using $\beta_n$, we can get similar estimates as in the proof of Proposition \ref{prop36}. Moreover, if we require that $\beta_n |_{R_n}$ only depends on the horizontal coordinate, then the boundary condition persists under the gluing. 
\end{proof}

\begin{prop}\label{propa7}
Suppose we have a sequence ${\bf v}_i^o \in \wt{M}^K({\mb H}, X, L)$ that converges to ${\bf v}_\infty^o \in \wt{M}^K({\mb H}, X, L)$ in c.c.t.. Then there exists a sequence of gauge transformations $g_i \in {\mc K}^{2, p}_{\rm loc}({\mb H})$ such that if we denote $g_i \cdot {\bf v}_i^o = (u_i, a_i)$, then $u_i$ converges to $u_\infty$ uniformly. Moreover, for $i$ sufficiently large, define $\xi_i \in W^{2, p}_{\rm loc}({\mb H}, u_\infty^* TX)_L$ by $\xi_i = \exp_{u_\infty}^{-1} u_i$. Then 
\beqn
\lim_{i \to \infty} \Big[ \| a_i - a_\infty \|_{L^{p, \delta}} + \| F_{a_i} - F_{a_\infty}\|_{L^{p, \delta}} + \| \nabla^{a_\infty} \xi_i \|_{L^{p, \delta}} + \| d\mu(u_\infty) \cdot J \xi_i \|_{L^{p, \delta}} + \| d\mu(u_\infty) \cdot \xi_i \|_{L^{p, \delta}} \Big] = 0.
\eeqn
\end{prop}

\begin{proof}[Sketch of Proof] This proposition is analogous to Proposition \ref{prop37} and the proof is almost the same word-by-word. The only difference is that when we do local estimates, the boundary conditions need to be addressed. We left the details to the reader.
\end{proof}

\subsubsection{Coulomb gauge}

Assume $a \in W^{1, p}_{\rm loc}({\mb H}, \Lambda^1 \otimes {\mf k})$. For $\alpha \in W^{1, p}_{\rm loc} ({\mb H}, \Lambda^1 \otimes {\mf k})$, one has the formal adjoint 
\beqn
\dad \alpha \in L^p_{\rm loc}({\mb H}, {\mf k}).
\eeqn
We also have the distributional adjoint $d_a^{\,\dagger}$ of $d_a$ defined by 
\beqn
(d_a^{\, \dagger} u )(\varphi) = \langle u, d_a \varphi\rangle
\eeqn
for smooth $u$ and $\varphi$. $\dad \neq d_a^{\, \dagger}$ classically because of the boundary contribution. Hence we distinguish the following analogue of Proposition \ref{prop38} in two cases.

\begin{prop}\label{propa8}
Suppose ${\bf v} = (u, a) \in \wt{M}^K({\mb H}, X, L)^{1, p}_{\rm loc}$ such that $a \in W^{1,p,\delta}({\mb H}, \Lambda^1 \otimes {\mf g})$. Then there are constants $\epsilon_{\rm coul}, c_{\rm coul}, \delta_{\rm coul}>0$ satisfying the following condition. 

\begin{enumerate}

\item \label{propa8a} For any $\alpha \in W^{1, p, \delta}({\mb H}, \Lambda^1 \otimes {\mf k})$ and $\xi \in W^{1, p}_{\rm loc}({\mb H}, u^* TX)_L$ satisfying
\begin{align}\label{eqna7}
&\ * \alpha|_{\mb R} = 0,&\ \| \alpha \|_{L^{p,\delta}} +  \| \xi \|_{L^\infty} + \| \dad \alpha + d\mu(u) \cdot J \xi \|_{L^{p, \delta}} \leq \epsilon_{\rm coul},
\end{align}
there exists a unique $s \in W^{2, p, \delta}_n ({\mb C}, {\mf k})$ satisfying 
\beqn
\dad ( e^s \cdot ( a + \alpha)  - a ) + d\mu(u) \cdot ( J \exp_u^{-1} ( e^s \exp_u \xi)) = 0,\ \| s \|_{W^{2, p, \delta}} \leq \delta_{\rm coul}.
\eeqn
Further, $\| s \|_{W^{2,p,\delta}} \leq c_{\rm coul} \| \dad \alpha  + d\mu(u) \cdot J \xi \|_{L^{p,\delta}}$.

\item \label{propa8b} For any $\alpha \in L_{\rm loc}^p ({\mb H}, \Lambda^1 \otimes {\mf k})$ and $\xi \in L^p_{\rm loc}({\mb H}, u^* TX)$ satisfying
\beq\label{eqna8}
 \| \alpha \|_{L^{p,\delta}} +  \| \xi \|_{L^\infty} + \| d_a^{\, \dagger} \alpha + d\mu(u) \cdot J \xi \|_{W^{-1, p, \delta}} \leq \epsilon_{\rm coul},
\eeq
there exists a unique $s\in W^{1, p, \delta}({\mb H}, {\mf g})$ satisfying
\beqn
d_a^{\, \dagger} ( e^s \cdot ( a + \alpha) - a) + d\mu(u) \cdot ( J \exp_u^{-1} (e^s \exp_u \xi)) =  0,\ \| s\|_{W^{1, p, \delta}} \leq \delta_{\rm coul}.
\eeqn
Further, $\| s \|_{W^{1, p, \delta}} \leq c_{\rm coul} \| d_a^{\, \dagger} \alpha + d\mu(u) \cdot J \xi \|_{W^{-1, p, \delta}}$.

\end{enumerate}
\end{prop}

To prove this theorem, one needs to apply the implicit function theorem. The various estimates similar to those in the proof of Proposition \ref{prop38} can be proved without much difference. However, a key ingredient is Calderon's theorem (Lemma \ref{lemma73}), which is not completely the same.

\begin{prop}\label{propa9}
 Assume $p \geq 2$.
\begin{enumerate}
\item \label{propa9a} $d^{\, *} d + {\rm Id}: W^{2, p}_n ({\mb H}) \to L^p({\mb H})$ is an isomorphism. 

\item \label{propa9b} $d^{\, \dagger} d + {\rm Id}: W^{1, p}({\mb H}) \to W^{-1, p}({\mb H})$ is an isomorphism.
\end{enumerate}
\end{prop}

\begin{proof}
For Part \eqref{propa9a}, if $u_+ \in {\rm ker}( d^{\,*} d + {\rm Id})\subset W^{2, p}_n({\mb H})$, then extend $u_+$ to the lower half plane $\ov{\mb H}$ by reflection, i.e., define
\beq\label{eqna9}
u (z) = \left\{ \begin{array}{c} u_+ (z),\ z \in {\mb H}\\
                                     u_+ (\bar{z}),\ z \in \ov{\mb H}.
\end{array} \right.
\eeq
Since $\partial_t u_+ (s, 0) = 0$, $u \in W^{2, p}({\mb C})$. By Lemma \ref{lemma73}, $u = 0$. Hence ${\rm ker}( d^{\, *} d + {\rm Id}) = \{0\}$. To prove it is surjective, for any $f_+ \in L^p({\mb H})$, extends it to $f \in L^p({\mb C})$ by the same formula \eqref{eqna9}. By Lemma \ref{lemma73}, there is $u \in W^{2, p}({\mb C})$ such that $- \Delta u + u = f$. By the reflection symmetry of the operator $\Delta$ and $f$ and the uniqueness of $u$, we know that $\partial_t u_+ (s, 0) = 0$. Then $d^{\, *} d u_+ + u_+ = f_+$ where $u_+$ is the restriction of $u$ to ${\mb H}$.

The proof of Part \eqref{propa9b} is essentially the same. For $u_+ \in {\rm ker}( d^{\, \dagger} d + {\rm Id})\subset W^{1, p}({\mb H})$, extends it to $u$ by the reflection \eqref{eqna9}, which lies in $W^{1, p}({\mb C})$. By Lemma \ref{lemma73}, $u = 0$ so the kernel is trivial. To prove it is surjective, first notice that any pair of distributions $f_+\in W^{-1, p}({\mb H})$, $g_- \in W^{-1, p}(\ov{\mb H})$ define a distribution $f_+ \cup g_- \in W^{-1, p}({\mb C})$ by 
\beqn
\langle f_+ \cup g_-, \varphi_+ \cup \varphi_- \rangle = \langle f_+, \varphi_+ \rangle_{\mb H} + \langle g_-, \varphi_- \rangle_{\ov{\mb H}},\ \forall \varphi \in W^{1, q}({\mb C}),
\eeqn
where $q = \frac{p}{p-1}$ and $\varphi_+$ and $\varphi_-$ are the restrictions of $\varphi$ to the upper and lower half plane, respectively. We call $f$ the union of $f_-$ and $f_+$. 

Let ${\mc R}: {\mb H} \to \ov{\mb H}$ be the complex conjugation. We claim that if $f_+\in W^{-1, p}({\mb H})$ and $f_-  = {\mc R}^* f_+ \in W^{-1, p}(\ov{\mb H})$ such that $f_+ \cup f_- = 0$, then $f_+ = 0$. Indeed, for any $\varphi_+ \in W^{1, q}({\mb H})$ define $\varphi_- = {\mc R}^* \varphi_+ \in W^{1, q}(\ov{\mb H})$. Then $\varphi_+ \cup \varphi_- \in W^{1, q}({\mb H})$. Therefore, 
\beqn
0 = \langle f_+ \cup f_-, \varphi_+ \cup \varphi_- \rangle = \langle f_+, \varphi_+ \rangle_{\mb H} + \langle f_-, \varphi_- \rangle_{\ov{\mb H}}. 
\eeqn
The two summands on the right hand side are equal. Hence $\langle f_+, \varphi_+ \rangle_{\mb H} = 0$.

Now take any $f_+ \in W^{-1, p}({\mb H})$ and we would like to find $u_+ \in W^{1, p}({\mb H})$ such that $d^{\, \dagger} d u_- + u_- = f_+$. By Lemma \ref{lemma73}, there is a unique $u \in W^{1, p}({\mb C})$ such that $- \Delta u + u = f_+ \cup f_-$ and ${\mc R}^* u = u$. Then for any smooth test function $\varphi \in C_0^\infty({\mb C})$, we see that
\beqn
\begin{split}
&\ \langle d^{\,\dagger} d u_+ + u_+, \varphi_+ \rangle_{\mb H} + \langle d^{\, \dagger} d u_- + u_-, \varphi_- \rangle_{\ov{\mb H}} \\
&\ = \langle u, \varphi\rangle + \langle du_+, d \varphi_+ \rangle_{{\mb H}} + \langle du_-, d\varphi_- \rangle_{\ov{\mb H}}\\
&\ =  \langle u, \varphi\rangle + \langle u_+, d^{\, *} d \varphi_+ \rangle  + \langle u_-, d^{\, *} d\varphi_- \rangle_{\ov{\mb H}}\\
&\ = \langle u + d^{\, *} d u, \varphi \rangle \\
&\ = \langle f, \varphi \rangle.
\end{split}
\eeqn
Here the second equality holds because the boundary contributions cancel. Hence $f_+ \cup f_- = (d^{\, \dagger} d u_+ + u_+) \cup (d^{\, \dagger} d u_- + u_-)$, or 
\beqn
0 = (f_+ - d^{\, \dagger} d u_+ - u_+) \cup ( f_- - d^{\, \dagger} d u_- - u_-).
\eeqn
However the two parts are reflections from each other. Therefore $d^{\, \dagger} d u_- + u_- = f_+$.
\end{proof}

\subsection{The local model}

Given an ${\mb H}$-vortex ${\bf v} = (u, a)\in \wt{M}^K({\mb H}, X, L)^{1, p}_{\rm loc}$, its infinitesimal deformations are pairs $(\xi, \alpha)$ where $\xi \in W^{1, p}_{\rm loc}({\mb H}, u^* TX)_L$ and $\alpha \in W^{1, p}_{\rm loc}({\mb H}, \Lambda^1 \otimes {\mf k})_N$. The space of gauge transformations is ${\mc K}^{2, p}_{\rm loc}({\mb H})_N$ which contains $g \in {\mc K}^{2, p}_{\rm loc}({\mb H})$ satisfying the Neumann boundary condition $\partial_t g|_{\mb R} = 0$. 

We define the norm $\| \cdot \|_{p, \delta}$ of $(\xi, \alpha)$ by the same formula as \eqref{eqn31}, and let $\ud{\hat{\mc B}}_{\bf v}$ be the space of all $(\xi, \alpha)$ with finite $\| \cdot \|_{p, \delta}$ norm. 

Take a small $\epsilon>0$ and denote $\ud{\hat{\mc B}}{}_{\bf v}^\epsilon$ the $\epsilon$-ball of $\ud{\hat{\mc B}}{}_{\bf v}$ centered at the origin. For $\epsilon$ small enough, we use the exponential map $\exp$ of an admissible metric (see Definition \ref{defna2}) to identify an element $\ud{\bm \xi} \in \ud{\hat{\mc B}}{}_{\bf v}^\epsilon$ with concrete geometric objects. Notice that since $L$ is totally geodesic, for any $\ud{\bm \xi} \in \ud{\hat{\mc B}}{}_{\bf v}^\epsilon$ and $\exp_{{\bf v}} \ud{\bm \xi} = (u', a')$, $u'$ still satisfies the Lagrangian boundary condition.

Define $\ud{\hat{\mc E}} \to \ud{\hat{\mc B}}{}_{\bf v}^\epsilon$ the Banach space bundle whose fibre at ${\bf v}' = (u', a')$ is 
\beqn
\ud{\hat{\mc E}}{}_{{\bf v}'} = L^{p, \delta}({\mb H}, E_{u'}^+) = L^{p, \delta}({\mb H}, (u')^* TX \oplus {\mf k} \oplus {\mf k}).
\eeqn
The vortex equation plus the Coulomb gauge condition relative to ${\bf v}$ gives the section $\hat{\mc F}_{\bf v}: \ud{\hat{\mc B}}{}_{\bf v}^\epsilon \to \ud{\hat{\mc E}}$ defined by the same formula as \eqref{eqn32}. 

On the other hand, one use the Levi-Civita connection of $\omega_X(\cdot, J \cdot)$ to differentiate $\hat{\mc F}_{\bf v}$. The linearization at ${\bf v}$, denoted by $\ud{\hat{\mc D}}{}_{\bf v}: \ud{\hat{\mc B}}{}_{\bf v} \to \ud{\hat{\mc E}}{}_{\bf v}$, has the same expression as \eqref{eqn25}.

For any ${\bf v} = (u, a) \in \wt{M}^K({\mb H}, X, L)$, by Proposition \ref{propa5}, one can use a gauge transformation $g \in {\mc K}^{1, p}_{\rm loc}({\mb H})$ to transform it so that $g\cdot u$ can be extended to a continuous map from the unit disk ${\mb D}$ to $X$ with boundary mapped into $L$. Hence ${\bf v}$ represents a homology class $\ud\beta \in H_2(X, L)$. It is easy to see that $\ud \beta$ only depends on the gauge equivalence class of ${\bf v}$. For any $\ud\beta \in H_2(X, L)$, let $\wt{M}^K({\mb H}, X, L; \ud\beta) \subset \wt{M}^K({\mb H}, X, L)$ be the subset of vortices that represent $\ud\beta$, and let $M^K({\mb H}, X, L;\ud \beta)$ be the corresponding set of gauge equivalence classes, which is closed under the compact convergence topology.

Let $\nu(\ud\beta) \in {\mb Z}$ be the Maslov index of $\ud\beta$.

\begin{prop}
Given ${\bf v}\in \wt{M}^K({\mb H}, X, L; \ud \beta)$, $\hat{\mc D}_{\bf v}$ is Fredholm with index 
\beqn
{\rm ind} \big( \hat{\mc D}_{\bf v} \big) = \bar{n} + \nu(\ud\beta).
\eeqn
\end{prop}

\begin{proof}[Sketch of Proof]
It is analogous to Proposition \ref{prop32} and the proof is similar. Two key ingredients, the Lagrangian boundary value problem versions of Lemma \ref{lemma42} and Lemma \ref{lemma46} can be established without difficulty.
\end{proof}

An ${\mb H}$-vortex ${\bf v}$ is {\bf regular} if $\ud{\hat{\mc D}}{}_{\bf v}: \ud{\hat{\mc B}}{}_{\bf v} \to \ud{\hat{\mc E}}{}_{\bf v}$ is surjective for all $p>2$ and $\delta \in (1- \frac{2}{p}, 1)$. For $\beta \in H_2(X, L)$, denote by $\wt{M}^K ({\mb H}, X, L; \beta)^{\rm reg}$ the set of regular ${\mb H}$-vortices representing the class $\beta$, and by $M^K ({\mb H}, X, L; \beta)^{\rm reg}$ the set of gauge equivalence classes, with the compact convergence topology which is Hausdorff. Given ${\bf v} \in \wt{M}^K ({\mb H}, X, L; \beta)^{\rm reg}$, for $\epsilon$ sufficiently small, the subset
\beqn
\ud{U}{}_{\bf v}^{\epsilon}:= ( \ud{\hat{\mc F}} )^{-1}(0) \cap \ud{\hat{\mc B}}{}_{\bf v}^{\epsilon}
\eeqn
is a smooth manifold with dimension equal to ${\rm ind} ( \ud{\hat{\mc D}}{}_{\bf v} )$. We have the continuous map 
\beqn
\ud{\hat\phi}{}_{\bf v}^{\epsilon}: \ud{U}{}_{\bf v}^{\epsilon} \to M^K ({\mb H}, X; \beta)
\eeqn
by $\ud{\bm \xi} \mapsto \exp_{\bf v} \ud{\bm \xi}$. As in the case of ${\mb C}$-vortices, we need to prove the following proposition.
\begin{prop}
There exists $\epsilon>0$ such that $\ud{\hat\phi}{}_{\bf v}^{\epsilon}: \ud{U}{}_{\bf v}^\epsilon \to M^K({\mb H}, X, L;\beta)$ is a homeomorphism onto its image.
\end{prop}

\begin{proof}[Sketch of Proof]
We first prove injectivity by contradiction. Suppose on the contrary that there exist a sequence $\epsilon_i \to 0$ and two sequences 
\beqn
\ud{\bm \xi}{}_i = (\xi_i, \alpha_i),\ \ud {\bm \xi}{}_i' = (\xi_i', \alpha_i') \in \hat{\mc B}_{{\bf v}}^{\epsilon_i}
\eeqn
such that $\exp_{{\bf v}} {\bm \xi}_i$ is gauge equivalent to $\exp_{{\bf v}} {\bm \xi}_i'$, via a smooth gauge transformation $e^{s_i}: {\mb H} \to K$. Then by the boundary condition on $\alpha_i, \alpha_i'$, one has $s_i \in \Gamma_n({\mb H},{\mf g})$. By a similar argument as in the proof of Proposition \ref{prop34}, we can show that for $i$ sufficiently large, $s_i$ must be zero (we need to use the uniqueness of Part \eqref{propa8a} of Proposition \ref{propa8}), which is a contradiction.	

For surjectivity, suppose on the contrary that a sequence ${\bf v}_i^o \in \wt{M}^K ({\mb H}, X, L; \beta)$ converge to ${\bf v}$ in c.c.t., and all $[{\bf v}_i^o]$ are not contained in the image of $\hat\phi_{\bf v}^{\epsilon}$ for any $\epsilon$. Then by Proposition \ref{propa7}, there exists a sequence of gauge transformations $g_i \in {\mc K}^{2, p}_{\rm loc}({\mb H})$ such that the sequence ${\bf v}_i:= g_i \cdot {\bf v}_i^o$ satisfy conditions listed in Proposition \ref{propa7}. Then for $i$ large, ${\bf v}_i$ satisfies \eqref{eqna8} relative to ${\bf v}$. Hence by Proposition \ref{propa8}, there exists $s_i\in W^{1, p, \delta} ({\mb H}, {\mf k})$ such that $e^s \cdot {\bf v}_i$ is in weak Coulomb gauge relative to ${\bf v}_\infty$. Elliptic regularity shows that $s_i \in W^{2, p}_{\rm loc}$ and $* d s_i |_{\mb R} = 0$. Then we can write $g_i \cdot {\bf v}_i = \exp_{{\bf v}} {\bm \xi}_i$. By extending the regularity result Lemma \ref{lemma310}, one can prove that ${\bm \xi}_i\in \hat{\mc B}_{{\bf v}}$ and $\|{\bm \xi}_i\|_{p, \delta} \to 0$. Then for large $i$, $[{\bf v}_i^o]$ should in the image of $\hat\phi_{\bf v}^{\epsilon}$, which is a contradiction. \end{proof}

This implies that $\hat\phi_{{\bf v}}^\epsilon$ provides a local coordinate chart on $M^K ({\mb H}, X, L; \beta)$, for any $\beta \in H_2(X, L)$. It is routine to argue like in Subsection \ref{ss33} that the transition functions among these charts are smooth. Therefore we finishes the construction of the smooth manifold structure on $M^K ({\mb H}, X, L; \beta)^{\rm reg}$.

\bibliography{symplectic_ref,physics_ref}

\providecommand{\bysame}{\leavevmode\hbox to3em{\hrulefill}\thinspace}
\providecommand{\MR}{\relax\ifhmode\unskip\space\fi MR }
\providecommand{\MRhref}[2]{%
  \href{http://www.ams.org/mathscinet-getitem?mr=#1}{#2}
}
\providecommand{\href}[2]{#2}
\begin{thebibliography}{10}

\bibitem{Behrend_Fantechi}
Kai Behrend and Barbara Fantechi, \emph{The intrinsic normal cone}, Inventiones
  Mathematicae \textbf{128} (1997), no.~1, 45--88.

\bibitem{BCK}
Aaron Bertram, Ionu\c{t} {Ciocan-Fontaine}, and Bumsig Kim, \emph{Two proofs of
  a conjecture of {H}ori and {V}afa}, Duke Mathematical Journal \textbf{126}
  (2005), 101--136.

\bibitem{BDHW}
Indranil Biswas, Ajneet Dhillon, Jacques Hurtubise, and Richard Wentworth,
  \emph{A generalized {Q}uot scheme and meromorphic vortices}, Advances in
  Theoretical and Mathematical Physics \textbf{19} (2015), no.~4, 905--921.

\bibitem{choquet}
Yvonne Choquet-Bruhat and Demetrios Christodoulou, \emph{Elliptic systems in
  ${H}_{s, \delta}$ spaces on manifolds which are {E}uclidean at infinity},
  Acta Mathematica \textbf{146} (1981), 129--150.

\bibitem{Cieliebak_Gaio_Mundet_Salamon_2002}
Kai Cieliebak, Ana Gaio, Ignasi {Mundet i Riera}, and Dietmar Salamon,
  \emph{The symplectic vortex equations and invariants of {H}amiltonian group
  actions}, Journal of Symplectic Geometry \textbf{1} (2002), no.~3, 543--645.

\bibitem{Cieliebak_Gaio_Salamon_2000}
Kai Cieliebak, Ana Gaio, and Dietmar Salamon, \emph{${J}$-holomorphic curves,
  moment maps, and invariants of {H}amiltonian group actions}, International
  Mathematical Research Notices \textbf{16} (2000), 831--882.

\bibitem{dimofte:vortex}
Tudor Dimofte, Sergei Gukov, and Lotte Hollands, \emph{Vortex counting and
  {L}agrangian 3-manifolds}, Letters in Mathematical Physics \textbf{98}
  (2011), 225--287.

\bibitem{dupei:vortex}
Pei Du and Sergei Gukov, \emph{Equivariant {V}erlinde formula from fivebranes
  and vortices}, \url{http://arxiv.org/abs/1501.01310}.

\bibitem{Gaio_Salamon_2005}
Ana Gaio and Dietmar Salamon, \emph{Gromov-{W}itten invariants of symplectic
  quotients and adiabatic limits}, Journal of Symplectic Geometry \textbf{3}
  (2005), no.~1, 55--159.

\bibitem{GW_Toric}
Eduardo Gonz\'alez and Chris Woodward, \emph{Quantum cohomology and toric
  minimal model program}, \url{http://arxiv.org/abs/1207.3253}.

\bibitem{Jaffe_Taubes}
Arthur Jaffe and Clifford Taubes, \emph{Vortices and monopoles}, Progress in
  Physics, no.~2, Birkh\"auser, 1980.

\bibitem{McDuff_Salamon_2004}
Dusa McDuff and Dietmar Salamon, \emph{${J}$-holomorphic curves and symplectic
  topology}, Colloquium Publications, vol.~52, American Mathematical Society,
  2004.

\bibitem{Mundet_thesis}
Ignasi {Mundet i Riera}, \emph{Yang-{M}ills-{H}iggs theory for symplectic
  fibrations}, Ph.D. thesis, Universidad Aut\'onoma de Madrid, 1999.

\bibitem{Mundet_2003}
\bysame, \emph{Hamiltonian {G}romov-{W}itten invariants}, Topology \textbf{43}
  (2003), no.~3, 525--553.

\bibitem{Mundet_Tian_2009}
Ignasi {Mundet i Riera} and Gang Tian, \emph{A compactification of the moduli
  space of twisted holomorphic maps}, Advances in Mathematics \textbf{222}
  (2009), 1117--1196.

\bibitem{Stein_70}
Elias Stein, \emph{Singular integrals and differentiability properties of
  functions}, Princeton University Press, 1970.

\bibitem{Taubes_vortex}
Clifford Taubes, \emph{Arbitrary ${N}$-vortex solutions to the first order
  {G}inzburg-{L}andau equations}, Communications in {M}athematical {P}hysics
  \textbf{72} (1980), no.~3, 277--292.

\bibitem{Tian_Xu}
Gang Tian and Guangbo Xu, \emph{Analysis of gauged {W}itten equation}, Journal
  f\"ur die reine und angewandte Mathematik \textbf{Published Online} (2015),
  1--88.

\bibitem{Tian_Xu_2}
\bysame, \emph{Correlation functions in gauged linear $\sigma$-model}, Science
  China. Mathematics \textbf{59} (2016), 823--838.

\bibitem{Tian_Xu_3}
\bysame, \emph{Virtual fundamental cycles of gauged {W}itten equation},
  \url{http://arxiv.org/abs/1602.07638}, 2016.

\bibitem{VW_affine}
Sushimita Venugopalan and Chris Woodward, \emph{Classification of vortices},
  Duke Mathematical Journal \textbf{165} (2016), 1695--1751.

\bibitem{Wang_Xu}
Dongning Wang and Guangbo Xu, \emph{A compactification of the moduli space of
  disk vortices in adiabatic limit}, Mathematische Zeitschrift \textbf{to
  appear} (2016), 59pp, \url{http://arxiv.org/abs/1505.05945}.

\bibitem{Wehrheim_Uhlenbeck}
Katrin Wehrheim, \emph{Uhlenbeck compactness}, European Mathematical Society
  series of lectures in mathematics, European Mathematical Society, 2003.

\bibitem{Woodward_toric}
Chris Woodward, \emph{Gauged {F}loer theory for toric moment fibers}, Geometric
  and Functional Analysis \textbf{21} (2011), 680--749.

\bibitem{Woodward_15}
\bysame, \emph{Quantum {K}irwan morphism and {G}romov-{W}itten invariants of
  quotients {I}, {I}{I}, {I}{I}{I}}, Transformation Groups \textbf{20} (2015),
  507--556, 881--920, 1155--1193.

\bibitem{Woodward_Xu}
Chris Woodward and Guangbo Xu, \emph{Open quantum {K}irwan map}, in
  preparation.

\bibitem{Guangbo_vortex}
Guangbo Xu, \emph{Classification of ${U}(1)$-vortices with target
  {$\mathbb{C}^N$}}, International Journal of Mathematics \textbf{26} (2015),
  no.~13, 1550129 (20 pages).

\bibitem{Xu_glue}
\bysame, \emph{Gluing affine vortices}, \url{https://arxiv.org/abs/1610.09764},
  2016.

\bibitem{Ziltener_thesis}
Fabian Ziltener, \emph{Symplectic vortices on the complex plane and quantum
  cohomology}, Ph.D. thesis, Swiss {F}ederal {I}nstitute of {T}echnology
  {Z}urich, 2005.

\bibitem{Ziltener_decay}
\bysame, \emph{The invariant symplectic action and decay for vortices}, Journal
  of {S}ymplectic {G}eometry \textbf{7} (2009), no.~3, 357--376.

\bibitem{Ziltener_book}
\bysame, \emph{A quantum {K}irwan map: bubbling and {F}redholm theory}, Memiors
  of the {A}merican Mathematical Society \textbf{230} (2014), no.~1082, 1--129.

\end{thebibliography}

\bibliographystyle{amsplain}

\end{document}